\documentclass[english,10pt]{compositio}
\usepackage{baskervald}
\usepackage[utf8x]{inputenc} %  Ces deux lignes permettent de faire les accents sans te faire chier sinon il faut taper \'e pour √© (relou)
\usepackage[english]{babel}
\usepackage{amsthm} 
\usepackage{amsmath} % Toutes ces extensions proviennent de la classe AMS-LaTeX.
\usepackage{amsfonts}
\usepackage{amssymb}
\usepackage[]{todonotes}

\newcommand\NN{\mathbb{N}} % Ensemble des entiers naturels.
\newcommand\ZZ{\mathbb{Z}} % Ensemble des entiers relatifs.
\newcommand\QQ{\mathbb{Q}} % Ensemble des nombres rationnels.
\newcommand\RR{\mathbb{R}} % Ensemble des nombres r√É¬©els.
\newcommand\CC{\mathbb{C}} % Ensemble des nombres complexes.
 
\newcommand\FF{\mathbb{F}} 
\newcommand\FP{\mathbb{F}_p} 
 
\newcommand\eps{\varepsilon} 
\newcommand\fleche{\longrightarrow} 

\DeclareMathOperator{\Ker}{Ker}
\DeclareMathOperator{\ord}{ord}
\DeclareMathOperator{\ch}{ch}
\DeclareMathOperator{\sh}{sh}

\DeclareMathOperator{\Cond}{Cond}

\DeclareMathOperator{\Stab}{Stab}
\DeclareMathOperator{\Coker}{Coker}

\DeclareMathOperator{\Isom}{Isom}

\DeclareMathOperator{\tr}{tr}
\DeclareMathOperator{\Tr}{Tr}

\DeclareMathOperator{\Hom}{Hom}

\DeclareMathOperator{\Spec}{Spec}
\DeclareMathOperator{\Spf}{Spf}

\DeclareMathOperator{\id}{id}
\DeclareMathOperator{\Gal}{Gal}
\DeclareMathOperator{\Fil}{Fil}
\DeclareMathOperator{\im}{Im}
\DeclareMathOperator{\Sym}{Sym}

\DeclareMathOperator{\Ha}{Ha}

\DeclareMathOperator{\Fitt}{Fitt}

\DeclareMathOperator{\Diag}{Diag}
\DeclareMathOperator{\End}{End}
\DeclareMathOperator{\Res}{Res}

\DeclareMathOperator{\HN}{HN}
\DeclareMathOperator{\Deg}{Deg}

\DeclareMathOperator{\Sp}{sp}
\DeclareMathOperator{\GL}{GL}
\DeclareMathOperator{\HT}{HT}
\DeclareMathOperator{\ha}{ha}
\DeclareMathOperator{\ind}{ind}
\DeclareMathOperator{\Spm}{Spm}
\DeclareMathOperator{\GSp}{GSp}

\usepackage{tikz}
\usetikzlibrary{matrix,arrows}
\definecolor{cqcqcq}{rgb}{0.752941176471,0.752941176471,0.752941176471}
\definecolor{ffqqqq}{rgb}{0.333333333333,0.333333333333,0.333333333333}
\definecolor{cqcqcq}{rgb}{0.752941176471,0.752941176471,0.752941176471}
\definecolor{qqqqff}{rgb}{0.,0.,1.}
\definecolor{cqcqcq}{rgb}{0.752941176471,0.752941176471,0.752941176471}
\definecolor{ffqqqq}{rgb}{1.,0.,0.}
\usepackage{mathabx,epsfig}

\usepackage[all]{xy}
 \def\dar[#1]{\ar@<2pt>[#1]\ar@<-2pt>[#1]}
 \def\tar[#1]{\ar@<4pt>[#1]\ar@<0pt>[#1]\ar@<-4pt>[#1]}

\theoremstyle{definition} % Pour tout ce qui ressemble √É¬† des d√É¬©finitions : d√É¬©finitions, notations...
\newtheorem{definen}{Definition}[section]
\newtheorem{defin}[definen]{Definition}

\theoremstyle{plain} % Pour tous ce qui ressemble √É¬† des th√É¬©or√É¬®mes.
\newtheorem{theoren}[definen]{Theorem}
\newtheorem{theor}[definen]{Theorem}% D√©finit √† latex les environnements (et la fa√ßon de num√©roter avec [chapter])
\newtheorem{lemmen}[definen]{Lemma}  % [theor] = num√©roter comme les th√©or√®mes
\newtheorem{propen}[definen]{Proposition}

\newtheorem{coren}[definen]{Corollary}
\newtheorem{conjecture}[definen]{Conjecture}

\theoremstyle{remark} % Pour tout ce qui ressemble √É¬† des remarques.
\newtheorem{remaen}[definen]{Remark}

\begin{document}

\title{Families of Picard modular forms and an application to the Bloch-Kato conjecture.}
\author{Valentin Hernandez}
\email{valentin.hernandez@math.cnrs.fr}

\date{}
\address{Bureau iC1\\LMO Orsay}
%\dedication{A dedication can be included here.}

\classification{11G18, 11F33, 14K10 (primary), 11F55, 11G40, 14L05, 14G35, 14G22 (secondary).}
\keywords{Eigenvarieties, $p$-adic automorphic forms, Picard modular varieties, Bloch-Kato conjecture, $p$-divisible groups}
\thanks{The author has received partial funding from the European Research Council under the European Union's Horizon 2020 research and innovation programme (grant agreement No 682152).}

\begin{abstract}
In this article we construct a p-adic three dimensional Eigenvariety for the group U(2,1)(E), where E is a quadratic imaginary field and p is inert in E. The Eigenvariety parametrizes Hecke eigensystems on the space of overconvergent, locally analytic, cuspidal Picard modular forms of finite slope. The method generalized the one developed in Andreatta-Iovita-Pilloni by interpolating the coherent automorphic sheaves when the ordinary locus is empty. As an application of this construction, we reprove a particular case of the Bloch-Kato conjecture for some Galois characters of E, extending the results of Bellaiche-Chenevier to the case of a positive sign.
\end{abstract}

\maketitle

\vspace*{6pt}\tableofcontents  % for this guide only.
% A table of contents should normally not be included

\section{Introduction}

Families of automorphic forms have been a rather fruitful area of research since their introduction by Hida in 1986 for ordinary modular forms and their generalisations, notably the 
Coleman-Mazur eigencurve, but also to other groups than $GL_2$. Among examples of applications we can for example cite some cases of the Artin conjecture, for 
many modular forms the parity conjecture, and generalisations to a bigger class of automorphic representations of instances of Langlands' philosophy (together with local-global 
compatibility).

The goal of this article is to present a new construction of what is called an "eigenvariety", i.e. a $p$-adically rigid-analytic variety which parametrises \textit{Hecke eigensystems}. More precisely, the idea is to construct families of eigenvalues for an appropriate Hecke algebra acting on certain rather complicated cohomology groups which are large $\QQ_p$-Banach spaces, into which we can identify \textit{classical} Hecke eigenvalues. For example Hida and Emerton consider for these cohomology groups some projective systems of étale cohomology on a tower of Shimura varieties, whereas Ash-Stevens and Urban instead consider cohomology of a large system of coefficients on a Shimura variety. Another construction which was introduced for $\GL_2$ by Andreatta-Iovita-Stevens and Pilloni was to construct large coherent Banach sheaves on some open neighborhood of the rigid modular curve (more precisely on strict neighborhoods of the ordinary locus at $p$) indexed by $p$-adic weights and that vary $p$-adically. Their approach was then improved in \cite{AIP,AIP2} to treat the case of Siegel and Hilbert modular forms, still interpolating classical automorphic sheaves by large (coherent) Banach sheaves. This method relies heavily on the construction of the Banach sheaves for which the theory of the canonical subgroup is central. For example in the case of $\GL_2$, the idea is to construct a fibration in open ball centered in the images through the Hodge-Tate map of generators of the (dual of the) canonical subgroup inside the line bundle associated to the conormal sheaf $\omega$ on the modular curve $X_0(p)$.
This rigid sub-bundle has then more functions but as the canonical subgroup doesn't exists on the entire modular curve this fibration in open balls only exists on a strict neighborhood of the ordinary locus. Following the strategy of \cite{AIP,AIP2}, Brasca \cite{Bra} extended this eigenvariety construction to groups that are associated to PEL Shimura varieties whose ordinary locus (at $p$) is non empty, still using the canonical subgroup theory as developed in \cite{Far2}.

As soon as the ordinary locus is empty, the canonical subgroup theory gives no information and without a generalisation of it the previous strategy seems vacuous. To my knowledge no eigenvarieties has been constructed using coherent cohomology when the ordinary locus is empty. Fortunately we developed in \cite{Her2} a generalisation of this theory, called the canonical filtration, for (unramified at $p$) PEL Shimura varieties. The first example when this happen is the case of $U(2,1)_{E/\QQ}$, where $E$ is a quadratic imaginary field, as the associated Picard modular surface has a non empty ordinary locus if and only if $p$ splits in $E$. In this article we present a construction of an eigenvariety interpolating $p$-adically (cuspidal) Picard modular forms when $p$ is inert in $E$. The strategy is then to construct new coherent Banach sheaves on strict neighborhoods of the $\mu$-ordinary locus using the ($2$-steps) canonical filtration, and we get the following result,

\begin{theor}
Let $E$ be a quadratic imaginary field and $p \neq 2$ a prime, inert in $E$. Fix a neat level $K$ outside $p$, and a type\footnote{In first approximation we can simply think of $(K_J,J)$ being $(K,1)$. Forms of type $(K,1)$ are simply forms of level $K$} $(K_J,J)$ outside $p$, i.e. a compact open $K_J$ trivial at $p$, together with $J$ a finite dimensionnal complex representation of $K_J$, and we assume moreover that $K \subset \Ker J \subset K_J$. 
Let $N$ be the places where $K$ is not hyperspecial (or very special) and $I_p$ a Iwahori subgroup at $p$. There exists two equidimensional of dimension $3$ rigid spaces,
\[ \mathcal E \overset{\kappa}{\fleche} \mathcal W,\]
with $\kappa$ locally finite, together with dense inclusions $\ZZ^3 \subset \mathcal W$ and $\mathcal Z \subset \mathcal E$ such that $\kappa(\mathcal Z) \subset \ZZ^3$, and each $z \in \mathcal Z$ coincides with Hecke eigensystem for $\mathcal H = \mathcal H^{Np}\otimes \ZZ[U_p,S_p]$  acting on cuspidal Picard modular forms of weight $\kappa(z)$, type $(K_JI_p,J)$   that are finite slope for the action of $U_p$. More precisely, we have a map $\mathcal H \fleche \mathcal O(\mathcal E)$, which induces an injection,
\[ \mathcal E(\overline{\QQ_p}) \hookrightarrow \Hom(\mathcal H,\overline{\QQ_p})\times \mathcal W(\overline{\QQ_p}),\]
and such that for all $w \in \mathcal W(\overline{\QQ_p})$, $\kappa^{-1}(w)$ is identified by the previous map with eigenvalues for $\mathcal H$ on the space of cuspidal, overconvergent and locally analytic Picard modular forms of weight $w$, type $(K_J,J)$ which are finite slope for the action of $U_p$. When $z \in \mathcal Z$ is of weight $w$ (necessarily in $\ZZ^3$), then the system of Hecke eigenvalues for $\mathcal H$ moreover coincide with one of $\mathcal H$ acting on \emph{classical} previous such forms.
\end{theor}

In order to get the previous result we need to have a control on the global sections of the Banach sheaves mentioned before the theorem. A general strategy to prove such a control result is developed in \cite{AIP} and rely on a vanishing theorem for cuspidal functions on the toroïdal and minimal compactifications of the corresponding Shimura varieties.
This kind of vanishing results have been proven in great generality in \cite{Lan} (though it doesn't apply here directly), but in the simpler case of $U(2,1)$, as the boundary of the toroidal compactification is quite simple, we managed to simplify a part of the argument of \cite{AIP}. In a forthcoming work, we will use this method together with the technics developed in \cite{Her2} to contruct eigenvarieties for more general PEL Shimura datum.

The second part of this article focuses on a very nice application of eigenvarieties to construct Galois extensions in certain Selmer groups. The method follows the strategy initiated by Ribet (\cite{Rib}) in the case of inequal characteristics to prove 
the converse to Herbrand theorem. It was then understood by Mazur-Wiles how to apply this technic in equal caracteristics  using Hida families to prove Iwasawa main conjecture. In his PhD \cite{Bellthesefull}, 
Bellaïche understood that using a certain endoscopic representation together with a generalisation of Ribet's lemma he could produce some extensions of Galois representations, and then how to delete the \textit{wrong} extensions to only keep the one predicted by the Bloch-Kato conjecture. This method was then improved using $p$-adic families and Kisin's result on triangulations of modular forms to 
construct the desired extensions in Selmer groups as in \cite{BC1} for imaginary  quadratic Hecke character and \cite{SU} for modular forms using Saito-Kurokawa lifts to $GSp_4$. 
In the previous constructions, it seemed necessary that the sign at the center of the functionnal equation is -1, in order to get an endoscopic automorphic representation for a bigger group ($U(3)$ in \cite{BC1}, $GSp_4$ in \cite{SU}) related the one we started with. In this article, we study the simplest case with a sign +1.

Let $\chi$ be an algebraic Hecke character of $E$ satisfying the following polarisation,
\[\chi^\perp := (\chi^c)^{-1} = \chi |.|^{-1}\]
and $L(\chi,s)$ its $L$-function. In particular $\chi_{\infty}(z) = z^a\overline z^{1-a}$ for all $z \in \CC^\times$ where $a\in \ZZ$. Denote $\chi_p : G_E \fleche F^\times$, where $F/\QQ_p$ is a finite extension, the associated $p$-adic Galois character, and $H^1_f(E,\chi_p)$ the Selmer group of $\chi_p$, which is the sub-$F$-vector space of $H^1(E,\chi_p)$ generated by the extensions,
\[ 0 \fleche \chi_p \fleche V \fleche 1 \fleche 0,\]
such that for $v \notdivides p$, 
\[ \dim V^{I_v} = 1 + \dim \rho^{I_v},\]
and for $v | p$,
\[ \dim D_{cris,v}(V) := \dim D_{cris}(V_{|G_v}) = 1 + \dim D_{cris,v}(\rho).\]
We will say that such an extension has \textit{good reduction everywhere}. The conjecture of Bloch-Kato predicts the equality $\ord_{s=0}L(\chi,s) = \dim_F H^1_f(E,\chi_p)$, and in particular the following result, due to Rubin \cite{Rub},

\begin{theor}[Rubin] Suppose $p \notdivides |\mathcal O_E^\times|$.
If $L(\chi,0) = 0$ then $H^1_f(E,\chi_p)\neq \{0\}$.
\end{theor}

The previous result follows from Rubin's work on Iwasawa main conjecture for CM elliptic curve and its proof uses Euler systems. In particular this construction prove that the Selmer group is non trivial but doesn't really construct a particular extension. For example it is not possible to know if the extension that exists is \textit{geometric} or \textit{automorphic} in nature. 
Another proof of this result (\cite{BC1}) uses families of  modular forms given by the corresponding eigenvarieties, a particular case of transfer as predicted by Langland's philosophy, together with a generalisation of Ribet's "change of lattice" Lemma. More precisely, if $p$ is split in $E$, $p \notdivides \Cond(\chi)$, and the order of vanishing $\ord_{s=0}L(\chi,s)$ is odd, then Bellaïche-Chenevier can construct the predicted extension in $H^1_f(E,\chi_p)$ by deformation of a non-tempered automorphic form $\pi^n(\chi)$ for $U(3)$, the compact at infinity unitary group in three variables.
It is a natural question to ask why this condition of the order of vanishing being odd is necessary. If the order of vanishing is even, following multiplicity results on automorphic representations for unitary groups on three variables of Rogawski (\cite{Rog,Rogbook}), there exists a non tempered automorphic representation $\pi^n(\chi)$ for $U(2,1)$ with Galois representation $\rho_{\pi^n(\chi),p} = 1 \oplus \chi_p \oplus \chi_p^{\perp}$. In this article we check that we can indeed deform this representation such that the associate Galois deformation is generically irreducible, and that we can control the reduction at each place, thus constructing an extension in the Selmer group. More precisely we can reprove the following case of Rubin's result,

\begin{theor}
Let $p$ a prime, unramified in $E$, such that $p \notdivides \Cond(\chi)$. Suppose moreover that $p \neq 2$ if $p$ is inert.
If $L(\chi,0) =0$ and $\ord_{s=0}L(\chi,s)$ is even, then
\[ H^1_f(E,\chi_p) \neq 0.\]
\end{theor}

In particular we can extend the result of \cite{BC1} when the order of vanishing is even, and also to the case of an inert prime $p$ using the corresponding eigenvariety (in this case the ordinary locus is empty).
An advantage of the construction of the eigenvariety presented here is that if an Hecke eigensystem appears in the classical cuspidal global sections of a coherent automorphic sheaf, then there is an associated point on the eigenvariety. This argument might be more complicated with other constructions, as the representation $\pi^n(\chi)$ is not a regular discrete series (it doesn't even \textit{appear} in the cohomology in middle degree). Another advantage of using coherent cohomology is that we can also deal with the limit case where $\pi^n(\chi)$ does not appears in the etale cohomology\footnote{i.e. when $a \in \{0,1\}$ i.e. $\chi_\infty(z) = z$ or $\chi_\infty(z) = \overline z$} (but it was known to Bellaïche \cite{BellSelmer} how to get this limit case).
Apart from this fact, the deformation when $p$ is split follows the lines of \cite{BC1}, whereas when $p$ is inert the geometry of the eigenvariety is quite different. In particular, there is less \textit{refinements} (and thus only one point on $\mathcal E$ corresponding to $\pi^n(\chi)$ instead of three) and we need a bit more care to isolate the \textit{right} extension. We also need to be slightly more careful with $p$-adic Hodge theory to understand the local-global compatibility at $p$ and a generalisation of Kisin's result on triangulation of refined families as provided by \cite{Liu}. 
Let us also remark that a consequence of this construction and Chenevier's method (\cite{CheJL}) to compare the eigenvarieties for $U(3)$ and $U(2,1)$ (say when $p$ splits) is that the point 
$\pi^n(\chi)$ when the sign at infinity is $+1$ together with its \textit{good} refinement also appears in the eigenvariety of $U(3)$, despite \textit{not} being a classical point for this group.

A lot of the ideas developed in this article extends to more general Shimura varieties, and in particular to Picard modular varieties associated to a general CM field $E$. More precisely, the result on the canonical filtration (Theorem \ref{thrfiltcan}) in this last case remains true as soon as $p \geq 3$. For more general groups the bound on the prime $p$ will grow with the group as predicted in \cite{Her2}. This is also why we can't use $p=2$ in this article. With this restriction on the prime, we should also be able to construct the associated eigenvariety, and we hope to come back on this matter in a following article. For the second part, which deals with the Bloch-Kato conjecture, it is less clear what to expect. Already for $U(2,1)_{E/F}$ with $E$ a CM extension, we can use the work of Rogawski and prove that the Galois representation passing through a particular refinement is irreducible, unfortunately, just as in \cite{BC2}, we can't get rid of the bad extension in $H^1_f(E,\omega_p)$ anymore (see subsection \ref{sect115}). For higher dimensional groups, we could probably have similar results as the one in the book \cite{BC2}, but in addition to the previous remark, results in this case would be conditional to Arthur's conjecture, as in \cite{BC2}.

\subsection*{Acknowledgements} I would like to warmly thank my PhD advisors Laurent Fargues, and even more Vincent Pilloni for explaining me carefully his previous work and suggesting this subject. I would also like to thank Fabrizio Andreatta and Adrian Iovita for their support, remarks and interesting discussions on this subject. It should be clear that this work is a continuations of theirs. I would like to particularly thank Joël Bellaïche and Gaëtan Chenevier for their perfectly written article and book on extensions in Selmer groups (from which my inspiration is easy to feel) but also for very inspiring discussions, encouragements and remarks. 
I would also like to thank Nicolas Bergeron, Stéphane Bijakowski, Jean-François Dat, Guy Henniart, Bruno Klingler, Arthur-César Le Bras, Alberto Minguez, Benoit Stroh, and Éric Urban for very interesting discussions on the subject. I would also like to thank the referee for their remarks which helped improve the exposition.

\section{Shimura datum}
\subsection{Global datum}
Let $E/\QQ$ a quadratic imaginary field and denote $\overline{\bullet}$ the complex conjugation of $E$. Let $(V = E^3,\psi)$ be the hermitian space of dimension 3 over E, of signature (2,1) at infinity given by the matrix \[
J =\left(
\begin{array}{ccc}
  &   & 1  \\
  &  1 &   \\
 1 &   &   
\end{array}
\right).
\]
Let us then denote,
\begin{eqnarray*} G = GU(V,\psi) = GU(2,1)\\ = \{ (g,c(g)) \in \GL(V)\times \mathbb G_{m,\QQ} : \forall x,y \in V, \psi(gx,gy) = c(g)\psi(x,y)\} \subset \GL_V \times \mathbb G_m,\end{eqnarray*}
the reductive group over $\QQ$ of unitary similitudes of $(V,\psi)$. 

Let $p$ be a prime number, unramified in $E$. If $p = v\overline v$ is split in $E$, then,
\[ V \otimes_\QQ \QQ_p = V\otimes_E E_v \oplus \overline{V \otimes_E E_v},\]
where the action of $E_v$ is by $\overline v$ on $\overline{V \otimes_E E_v}$. Moreover, the complex conjugation exchanges ${V \otimes_E E_v}$ and $\overline{V \otimes_E E_v}$.
In particular, $G \otimes_\QQ \QQ_p \simeq \GL(V\otimes_E E_v)\times \mathbb G_m$ (this isomorphism depends on the choice of $v$ over $p$).

We will be particularly interested in the case where $p$ is inert in $E$, the case when $p$ split has been studied before (see for example \cite{Bra}).

\begin{remaen}
We could more generally work in the setting of $(B,\star)$ a simple $E$-algebra of rank 9 with an involution of the second kind, such that $(B\otimes \QQ_p,\star)$ is isomorphic to $(\End(V\otimes \QQ_p), \star_V)$\footnote{Here $\star_V(g) = J{^t\overline g}J$}, and replace $G$ with the group,
\[ G_B = \{ g \in B^\times : g^\star g = c(g) \in \mathbb G_{m,\QQ}\}.\]
The construction of the eigenvarieties in the case where $B$ isn't split is easier as the associated Shimura varieties are compact, but some non-tempered automorphic form, for example the one constructed by Rogawski and studied in \cite{BC1} and the second part of this article will never be automorphic for such non split $B$. 
\end{remaen}

The Shimura datum we consider is given by,
\[ h :
\begin{array}{ccc}
\mathbb S  & \fleche  & G_\RR  \\
z = x+iy  &\longmapsto   &   
\left(
\begin{array}{ccc}
x  &   & iy  \\
  & z  &   \\
iy  &   &   x
\end{array}
\right)
\end{array}
\]

\subsection{Complex Picard modular forms and automorphic forms}
\label{sectaut}

Classically, Picard modular forms are introduced using the unitary group $U(2,1)$, but we can treat the case of $GU(2,1)$ similarly. Let $G(\RR) = GU(2,1)(\RR)$ the group stabilizing (up to scalar) the signature matrix 
$J$
and let,
\[ Y = \{ z = (z_1,z_2) \in \CC^2 : 2\Im(z_1) + |z_2|^2 < 0\},\]
be the symmetric space associated to $G(\RR)$, it is isomorphic to the 2-dimensional complex unit ball
\[ B = \{ (z_1,z_2) \in \CC^2 | |z_1|^2 + |z_2|^2 < 1\}.\]
On $Y$, there is an action of $G(\RR)$ through,
\[ 
\left(
\begin{array}{cc}
 A &  b   \\
 c  &  d  
\end{array}
\right) z = \frac{1}{c\cdot z + d}(Az + b) \in B, \quad A \in M_{2\times2}(\CC).
\]
\begin{remaen}
It is known that $U(2,1)(\RR)$ stabilizes $Y$, and $GU(2,1)(\RR)$ stabilizes $X$ too as if $^t\overline AJA = cJ$ with $c \in \RR^\times$, we get $\overline{\det A}\det A = |\det A|^2 = c^3$ thus $c > 0$.
\end{remaen}

This action is transitive and identifies $Y$ with $G(\RR)/K_\infty$ where $K_\infty = \Stab((i,0)) \subset \{ (A,e) \in GU(2)(\RR) \times GU(1)(\RR) \}$ can be identified with $\{ (A,e) \in GU(2)(\RR) \times \CC^\times : c(A) = N(e)\}$. We denote $(i,0) = x_0$.

The subgroup $K_\infty$ is not compact but can be written $Z(\RR)^0(U(2)(\RR)\times U(1)(\RR))$, with $Z$ the center of $GU(2,1)$. Let $L$ be the $\CC$-points of $K_\infty$. 
Then $L \simeq (\GL_2 \times \GL_1)\times \GL_1(\CC)$ is the Levi of a parabolic $P$ in $\GL_3 \times \GL_1(\CC)$.
For any $\kappa = (k_1,k_2,k_3,r) \in \ZZ^4$ such that $k_1 \geq k_2$, there is an associated (irreducible) representation $S_\kappa(\CC)$ of $P$, of highest weight
\[
\left(
\begin{array}{ccc}
t_1  &   &   \\
  &  t_2 &   \\
  &   &  t_3 
\end{array}
\right), c \in (\GL_2 \times \GL_1)\times \GL_1(\CC) \longmapsto t_1^{k_1}t_2^{k_2}t_3^{k_3}c^r.
\]
$K_\infty$  embeds in $L \subset P$ by $(A,e) \mapsto ((\overline{A},e),N(e))$.

Following \cite{Har,HarVB},\cite{Mil}, such a representation gives $\Omega^\kappa$ a locally free sheaf with $G(\CC)$-action on $G(\CC)/P$, whose structure as sheaf doesn't depend on $r$.
Restricting it to $G(\RR)/K_\infty = Y$ we get a sheaf $\Omega^\kappa$ whose sections over $X$ can be seen as holomorphic functions,
\[ f : G(\RR)/K_\infty \longmapsto S_\kappa(\CC),\]
such that $f(gk) = \rho_\kappa(k)^{-1} f(g)$, for $g \in G(\RR), k \in K_\infty$, which we call (meromorphic at infinity) modular forms of weight $\kappa$. 
In an informal way, the choice of the previous integer $r$ normalizes the action of the Hecke operators and corresponds to normalizing the (norm of the) central character of the modular forms.
We will not use this description of the sheaves, and instead introduced a modular description of these automorphic sheaves.

Fix $\tau_\infty : E \fleche \CC$ an embedding, and $\sigma \neq 1 \in \Gal(E/\QQ)$, thus $\sigma\tau_\infty = \overline{\tau_\infty}$ is the other embedding of $E$. Over $\CC$, for any sufficiently small compact open $K$ in $G(\mathbb A_f)$, the Picard variety $Y_K(\CC)$ of level $K$ can be identified with a (disjoint union of some) quotient of $B = GU(2,1)/K_\infty$, but also with the moduli space parametrizing quadruples $(A,\iota,\lambda,\eta)$ where
$A$ is an abelian scheme of genus 3, $\iota : \mathcal O_E \fleche \End(A)$ is an injection, $\lambda$ is a polarisation for which Rosati involution corresponds to the conjugation $\overline{\cdot}$ on $\mathcal O$, and $\eta$ is a type-$K$-level structure such that the action of $\mathcal O_E$ on the conormal sheaf $\omega_A = e_{A/S}^*\Omega^1_{A/S}$\footnote{$e_{A/S}$ denotes the unit section of $A \fleche S$} decomposes under the embeddings $\tau_\infty,\sigma\tau_\infty$ into two direct factors of respective dimensions 1 and 2. This is done for example explicitly in \cite{dsg}, section 1.2.2 to 1.2.4, and we will be especially interested in the description by "moving lattice" given in 1.2.4. Thus, to every $x = (z_1,z_2) \in B$, we can associate a complex abelian variety,
\[ A_x = \CC^3/L_x,\]
where $L_x$ is the $\mathcal O_E$-module given in \cite{dsg} (1.25), and the action of $\mathcal O_E$ on $A_x$ is given by
\[ a \in \mathcal O_E \longmapsto 
\left(
\begin{array}{ccc}
\overline\tau_\infty(a)  &   &   \\
  & \overline\tau_\infty(a)  &   \\
  &   & {\tau_\infty(a)}
\end{array}
\right) \in M_3(\CC).
\]
There is moreover $\eta_x$ a canonical ($K$-orbit of) level-$N$-structure (for $K(N) \subset K \subset G(\mathbb A_f)$).
Over $Y_K(\CC)$ we thus have a sheaf $\omega_A$ that can be decomposed $\omega_{\tau,A} \oplus \omega_{\sigma\tau,A}$ according to the action of $\mathcal O_E$, and we can 
consider the sheaf \[\omega^\kappa := \Sym^{k_1-k_2}\omega_{\sigma\tau,A}\otimes {\det}^{\otimes k_2}\omega_{\sigma\tau,A} \otimes {\det}^{\otimes k_3}\omega_{\tau,A},\] for $(k_1,k_2,k_3)$ 
a dominant (i.e. $k_1 \geq k_2$) weight. Using the previous description, if we denote $\zeta_1,\zeta_2,\zeta_3$ the coordinates on $\CC^3$, $\omega_{A_x,\sigma\tau}$ is generated by 
$d\zeta_1,d\zeta_2$ and $\omega_{A_x,\tau}$ by $d\zeta_3$.

There is also $X_K(\CC)$ a toroidal compactification of $Y_K(\CC)$, \cite{Lars} and \cite{BellThese}, on which $\omega^\kappa$ extends as $\omega^{\kappa}$ (the \textit{canonical} sheaf of Picard modular forms) and $\omega^{\kappa}(-D)$ (the sheaf of cuspidal forms) where $D$ here is the closed subscheme $X_K \backslash Y_K$ together with its reduced structure. 

\begin{definen}
We call the module $H^0(X_K(\CC),\omega^{\kappa})$ (respectively $H^0(X_K(\CC),\omega^{\kappa}(-D)) =: H^0_{cusp}(X_K(\CC),\omega^\kappa)$) the 
space of (respectively, cuspidal) Picard modular forms of level $K$ and weight $\kappa$. 
\end{definen}

We sometimes say 'classical' if we want to emphasis the difference with overconvergent modular forms defined 
later. Denote also $V^\kappa$ the representation of $\GL_2\times \GL_1$ given by,
\[ (A,e) \longmapsto \Sym^{k_1-k_2}(\overline A)\otimes {\det}^{k_2}\overline A \otimes {\det}^{k_3}e.\]

\begin{definen}
For all $g \in G(\RR) = GU(2,1)(\RR)$, write,
\[
g = \left(
\begin{array}{cc}
 A & b    \\
 c &  d  
\end{array}
\right), A = \left(
\begin{array}{cc}
 a_1 & a_2    \\
 a_3 &  a_4 
\end{array}
\right)
\]
and for $x = (z_1,z_2) \in B$, following \cite{Shi}, define,

\[ \kappa(g,x) = 
\left(
\begin{array}{cc}
\overline{a_1} - \overline{a_3}z_1 & \overline{c_2}z_1 - \overline{c_1}  \\
 \overline{a_3}z_2 -\overline{b_1}   & \overline{d} - \overline{c_2}z_2  
\end{array}
\right), \quad \text{and} \quad j(g,x) = ( cx +  d).\]
Finally, define,
\[ J(g,x) = (\kappa(g,x), j(g,x)) \in \GL_2\times \GL_1(\CC).\]
\end{definen}

The following proposition is well known (see \cite{Hsieh} Lemma 3.7) and probably already in \cite{Shi}, but we rewrite it to fix the notations,
\begin{propen}
There is a bijection between $H^0(Y_K(\CC),\omega^{\kappa})$ and functions $F : B \times G(\mathbb A_{\QQ,f}) \fleche V^\kappa$ such that,
\begin{enumerate}
\item For all $\gamma \in G(\ZZ), F(\gamma x, \gamma k) = J(\gamma,x) \cdot F(x,k)$,
\item For all $k' \in K, F(x,kk') = F(x,k)$,
\end{enumerate}
given by $F(x,k) = f(A_x,\eta_x\circ k^\sigma,(d\zeta_1,d\zeta_2,d\zeta_3))$
\end{propen}

\begin{proof}[Proof]
For all $\gamma \in G(\ZZ)$, there is an isomorphism between $(A_x,\eta_x)$ and $(A_{\gamma x},\eta_{\gamma x} \circ \gamma^\sigma)$, for exemple described in \cite{dsg} 1.2.2 or in \cite{Gordon} which sends $(d\zeta_1,d\zeta_2,d\zeta_3)$ to $\gamma^*(d\zeta_1,d\zeta_2,d\zeta_3) = (\gamma^*(d\zeta_1,d\zeta_2),\gamma^*d\zeta_3)$ as $\gamma$ preserve the action of $\mathcal O_E$. $\gamma^* d\zeta_3$ is calculated in \cite{dsg}, Proposition 1.15, and given by,
\[ \gamma^*d\zeta_3 = j(\gamma,x)^{-1}d\zeta_3.\]
Moreover, by the Kodaira-Spencer isomorphism $\omega_{A_x,\tau}\otimes \omega_{A_x,\sigma\tau} = \Omega^1$ (\cite{dsg} Proposition 1.22 for example), we only need to determine the action of $\gamma$ on $dz_1,dz_2$. But this is done in \cite{Shi}, 1.15 (or an explicite calculation), given by $c(\gamma){^t}\kappa^{-1}(\gamma,x)j(\gamma,x)$, and we get,
\[ \gamma^*(d\zeta_1,d\zeta_2) =  {^t\kappa(\gamma,x)}^{-1}(d\zeta_1,d\zeta_2).\]
Thus, setting $F(x,k) = f(A_x,\eta_x \circ k^\sigma,(d\zeta_1,d\zeta_2,d\zeta_3))$, we get,
\[F(\gamma x, \gamma k) = \Sym^{k_1-k_2}({^t}\kappa(\gamma,x)^{-1})((\det \kappa(\gamma,x)))^{-k_2}j(\gamma,x)^{-k_3} F(x,k).\]
\end{proof}

Thus, to $f \in H^0(Y_K(\CC),\omega^\kappa)$ we can associate a function, $\Phi_f : G(\QQ)\backslash G(\mathbb A) \fleche V^\kappa$, by
\[{\Phi_f}(g) = c(g_\infty)^{-k_1-k_2-k_3}J(g_\infty,x_0)^{-1} \cdot F(g_\infty x_0, g_f),\]
where the action $\cdot$ is the one on $V^\kappa$, and we use the decomposition $g = g_\QQ g_\infty g_f \in G(\QQ)G(\RR)G(\mathbb A_f)$. We can check that this expression 
doesn't depends on the choice in the decomposition. This association commutes with Hecke operators. ${\Phi_f}$ doesn't necessarily have a unitary central character, but it is nevertheless normalized by $\kappa$. Indeed, for $z_\infty \in \CC^\times = Z(\RR)$,
\[{\Phi_f}(z_\infty g) = N(z_\infty)^{-k_1-k_2-k_3}\overline{z_\infty}^{k_1+k_2}{z_\infty}^{k_3}{\Phi_f}(g) = \overline z_\infty^{{-k_3}}{z_\infty}^{-k_1-k_2}{\Phi_f}(g).\]

Let $L : V^\kappa \fleche \CC$ a non zero linear form. Define the injective map of right-$K_\infty$-modules, 
\[ \mathcal L : 
\begin{array}{ccc}
 V^\kappa &  \fleche & Fonct(K_\infty,\CC)  \\
 v & \longmapsto  & L(J(k,x_0)^{-1}v)  
\end{array}
\]

We denote by $L^2_0(G(\QQ)\backslash G(\mathbb A),\CC)$ the set of cuspidal $L^2$-automorphic forms for $G$.
We have the following well known proposition,
\begin{propen}
\label{prop27}
\label{propautheck}
The map $f \mapsto \varphi_f = L\circ \Phi_f$ is an isometry from $H^0_{cusp}(X_K(\CC),\omega^\kappa) = H^0(X_K(\CC),\omega^\kappa(-D))$ to the subspace of $L^2_0(G(\QQ)\backslash G(\mathbb A),\CC)$ of functions $\varphi$, $C^\infty$ in the real variable, such that,
\begin{enumerate}
\item For all $g \in G(\mathbb A)$, the function $\varphi_g : k \in K_\infty \mapsto \varphi(gk)$ is in $\mathcal L(V^\kappa)$, and in particular $\varphi$ is right $K_\infty$-finite,
\item For all $k \in K, \varphi(gk) = \varphi(g)$,
\item For all $X \in \mathfrak p_\CC^-$, $X\varphi = 0$, i.e. $\varphi$ is holomorphic.
\item $\varphi$ is \textit{cuspidal}, i.e. for all unipotent radical $N$ of a proper parabolic $P$ of $G$, 
\[ \int_{N(\QQ)\backslash N(\mathbb A)} \varphi(ng)dn = 0.\] 
\end{enumerate}
This isometry is equivariant under the Hecke action of $\mathcal H^N$ ($K(N) \subset K$).

\end{propen}

Using the previous proposition, to every $f \in H^0_{cusp}(X_K,\omega^\kappa)$, an eigenvector for the Hecke algebra, we will be able to attach a automorphic form $\varphi_f$, and an automorphic representation $\Pi_f$ (with same central character).

\subsection{Local groups}

In this subsection, we describe the local group at an inert prime. Let $p$ be a prime, inert in $E$. % (in the case of a simple algebra $B$, we should require $p$ to split $B$). 
Let $E_p \supset \QQ_p$ its $p$-adic completion.
Recall that $V\otimes \QQ_p = E_p^3$ and that the hermitian form is given by the matrix,
\[J =
\left(
\begin{array}{ccc}
  &   & 1  \\
  &  1 &   \\
 1 &   &   
\end{array}
\right).
\]

The diagonal maximal torus $T$ of $G_{\QQ_p}$ is isomorphic to $E_p^\times \times E_p^\times$, 
\[ T(\QQ_p) =\left\{
\left(
\begin{array}{ccc}
a  &   &   \\
  & e  &   \\
  &   & N(e)\overline{a}^{-1}  
\end{array}
\right), \quad a,e \in E_p^\times \right\}.
\]
and contains $T^1$, isomorphic to $E_p^\times \times E^1_p$, where $E^1_p = \{ x \in E_p : x\overline x = 1\} = (\mathcal O_{E_p})^{1}$, the torus of $U(E^3,J)$,
\[ T^1(\QQ_p) =\left\{
\left(
\begin{array}{ccc}
a  &   &   \\
  & e  &   \\
  &   &\overline{a}^{-1}  
\end{array}
\right), \quad a \in E_p^\times, e \in E^1_p \right\}.
\]
We also have the Borel subgroups $B = B_{GL_3(E)} \cap G_{\QQ_p}$ of upper-triangular matrices,
\[ B(\QQ_p) = \left\{
\left(
\begin{array}{ccc}
a  & x  & y  \\
  & e  & \overline{x}\overline{a}^{-1}e  \\
  &   & N(e)\overline{a}^{-1}  
\end{array}
\right), \quad a,e,x,y \in E_p^\times \text{ and } \Tr(\overline{a}^{-1}y) =N(a^{-1}x)\right\},\]
and $B^1$ the corresponding Borel for $U(E^3,J)$,
\[ B^1(\QQ_p) = \left\{
\left(
\begin{array}{ccc}
a  & x  & y  \\
  & e  & \overline{x}\overline{a}^{-1}e  \\
  &   & \overline{a}^{-1}  
\end{array}
\right), \quad a,x,y \in E_p^\times, e \in E_p^1 \text{ and } \Tr(\overline{a}^{-1}y) =N(a^{-1}x)\right\}.\]

\section{Weight space}
\label{sectW}

Denote by $\mathcal O = \mathcal O_{E_p}$, and by $T^1(\ZZ_p)$ the torus $\mathcal O^\times \times \mathcal O^1$ over $\ZZ_p$. It is the $\ZZ_p$-points of a maximal torus of $U(2,1)$ and denote by $T = T^1\otimes_{\ZZ_p} \mathcal O$ the split torus over $\Spec(\mathcal O)$. 

\begin{definen}
The weight space $\mathcal W$ is the rigid space over $\QQ_p$ given by $\Hom_{cont}(T^1(\ZZ_p),\mathbb G_m)$. It coincides with Berthelot's rigid fiber \cite{Ber}, of the formal scheme (called the \textit{formal weight space}) given by,
\[ \mathfrak W := \Spf(\ZZ_p[[T^1(\ZZ_p)]]).\]
If $K$ is an extension of $\QQ_p$, the $K$-points  of $\mathcal W$ are given by,
\[ \mathcal W(K) = \Hom_{cont}(\mathcal O^\times \times \mathcal O^1,K^\times).\]
\end{definen}

$\mathcal W$ is isomorphic to a union of $(p+1)(p^2-1)$ open balls of dimension 3 (see Appendix \ref{AppW}, compare with \cite{Urb} section 3.4.2),
\[ \mathcal W \simeq \coprod_{(\mathcal O^\times \times \mathcal O^1)^{tors}} B_3(0,1).\]
There is also a universal character,
\[ \kappa^{un} : T^1(\ZZ_p) \fleche \ZZ_p[[T^1(\ZZ_p)]],\]
which is locally analytic and we can write $\mathcal W = \bigcup_{w>0} \mathcal W(w)$ as an increasing union of affinoids using the analycity radius (see Appendix \ref{AppW}).

\begin{definen}
To $k = (k_1,k_2,k_3) \in \ZZ^3$ is associated a character,
\[
\underline k : \begin{array}{ccc}
\mathcal O^\times \times \mathcal O^1  &  \fleche & \QQ_{p^2}  \\
(x,y)  & \longmapsto  &   (\sigma\tau)(x)^{k_1}\tau(x)^{k_3}(\sigma\tau)(y)^{k_2}.
\end{array}
\]
Characters of this form are called \textit{algebraic}, or classical if moreover $k_1 \geq k_2$. They are analytic and Zariski dense in $\mathcal W$. 
\end{definen}

\section{Induction}
\label{sectBGG}
Set $U = U(2,1)/\ZZ_p$, $T^1$ its maximal torus, $K = \QQ_{p^2}$ and $\mathcal O = \mathcal O_K$. We have $U\otimes_{\ZZ_p} \mathcal O \simeq \GL_3/\mathcal O$, and we 
denote by $T$ its torus, and $\GL_2 \times \GL_1 \subset P$ the Levi of the standard parabolic of $\GL_3/\mathcal O$. Let $T \subset B$ the upper triangular Borel of $\GL_2\times \GL_1$ and $U$ 
its unipotent radical.

\begin{definen}
Let $\kappa \in X^+(T)$, then there exists a (irreducible) algebraic representation of $\GL_2\times\GL_1$ (of highest weight $\kappa$) given by,
\[ V_\kappa = \{ f : GL_2 \times \GL_1 \fleche \mathbb A^1 : f(gtu) = \kappa(t)f(g), t \in T, u \in U\},\]
where $\GL_2\times \GL_1$ acts by translation on the left (i.e. $g f(x) = f(g^{-1}x)$). $V_\kappa$ is called the \textit{algebraic induction of highest weight} $\kappa$ of $\GL_2\times\GL_1$. 
\end{definen}

Let $I = I_1$ be the Iwahori subgroup of $\GL_2(\mathcal O)\times\GL_1(\mathcal O)$, i.e. matrices that are upper-triangular modulo $p$. Let $I_n$ be the subset of matrices in $B$ modulo $p^n$, i.e. of the form
\[
\left(
\begin{array}{ccc}
a  &  b &   \\
p^nc  & d  &   \\
  &   & e  
\end{array}
\right), \quad a,b,c,d \in\mathcal O.
\]
Denote by $N^0$ the opposite unipotent of $U$, and $N_n^0$ the subgroup of $N^0$ of elements reducing to identity modulo $p^n$. These are seen as subgroups of $\GL_2(\mathcal O)\times\GL_1(\mathcal O)$ rather than formal or rigid spaces, but could be identified with $\mathcal O$-points of rigid spaces. Precisely,
\[ N^0 = \left\{
\left(
\begin{array}{ccc}
1  &  0 &   \\
x  & 1 &   \\
  &   & 1  
\end{array}
\right), \quad x \in\mathcal O\right\} \quad \text{and} \quad N^0_n = \left\{
\left(
\begin{array}{ccc}
1  &  0 &   \\
p^nx  & 1 &   \\
  &   & 1  
\end{array}
\right), \quad x \in\mathcal O\right\}.
\]

We identify $N_n^0$ with $p^n\mathcal O \subset  (\mathbb A^1_{\mathcal O})^{an}$. For $\eps >0$; denote by $N_{n,\eps}^0$ the affinoid,
\[ N_{n,\eps}^0 = \bigcup_{x \in p^n\mathcal O} B(x,\eps) \subset (\mathbb A^1_\mathcal O)^{an}.\]
$N_{n,\eps}^0$ is a rigid space. Actually we could define a rigid space $\mathcal N^0$ (which would just be $\mathbb A^1$ here) such that $\mathcal N^0(\mathcal O) = N^0$, and $N_{n,\eps}^0$ is an affinoid in $\mathcal N^0$, and is a neighborhood of $N^0_n$.
For $L$ an extension of $\QQ_p$, denote $\mathcal F^{\eps-an}(N^0_n,L)$ the set of functions $N^0_n \fleche L$ which are restriction of analytic functions on $N^0_{n,\eps}$.
Let $\eps >0$ and $\kappa \in \mathcal W_\eps(L)$ a $\eps$-analytic character, we denote,
\[ V_{\kappa,L}^{\eps-an} = \{ f : I \fleche L : f(ib) = \kappa(b)f(i) \text{ and } f_{N^0} \in \mathcal F^{\eps-an}(N^0,L)\}.\]
We also denote, for $\eps >0$ and $k = \lfloor -\log_p(\eps) \rfloor$,
\[ V_{0,\kappa,L}^{\eps-an} = \{ f : I_{k}\fleche L : f(ib) =  \kappa(b)f(i) \text{ and } f_{N^0_{k}} \in \mathcal F^{\eps-an}(N^0_{k},L)\},\]
where $\lfloor . \rfloor$ denote the previous integer, and,
\[ V_{\kappa,L}^{l-an} = \bigcup_{\eps >0} V_{\kappa,L}^{\eps-an} \quad \text{and}\quad V_{0,\kappa,L}^{l-an} = \bigcup_{\eps>0}V_{0,\kappa,L}^{\eps-an}.\]
Concretely, $V_{0,\kappa,L}^{\eps-an}$ is identified with analytic functions on $B(0,p^{\lfloor - \log_p\eps\rfloor})$ (a ball of dimension 1) .

We can identify $V_{\kappa,L}^{l-an}$ with $\mathcal F^{l-an}(p\mathcal O,L)$ by restricting $f \in V_{\kappa,L}^{l-an}$ to $N^0$. We can also identify $V_{0,\kappa,L}^{l-an}$ to the germ of locally analytic function on 0.

Let \[ \delta = 
\left(
\begin{array}{ccc}
p^{-1}  &   &   \\
  & 1  &   \\
  &   &   1
\end{array}
\right)
\]
which acts on $\GL_2 \times \GL_1$, stabilise the Borel $B(K)$, and define an action on $V_{\kappa,L}$ for $\kappa \in X^+(T)$, via $(\delta \cdot f)(g) = f(\delta g \delta^{-1})$. The action by conjugation of $\delta$ on $I$ does not stabilise it, but it stabilises $N^0$. We can thus set, for $j \in I$, write $j = nb$ the Iwahori decomposition of $j$, and set,
\[ \delta \cdot f(j) = f(\delta n \delta^{-1}b).\]
We can thus make $\delta$ act on $V_{\kappa,L}^{\eps-an}, V_{\kappa,L}^{l-an}, V_{0,\kappa,L}^{\eps-an},V_{0,\kappa,L}^{l-an}$.
Via the identification $V_{\kappa,L}^{\eps-an} \simeq \mathcal F^{l-an}(p\mathcal O,L)$, $\delta \cdot f (z) = f(pz)$. Thus $\delta$ improves the analycity radius. Moreover, its supremum norm is negative.

\begin{propen}
Let $f \in V_{0,\kappa,L}^{l-an}$. Suppose $f$ is of finite slope under the action of $\delta$, i.e. $\delta \cdot f = \lambda f, \lambda \in L^\times$. Then $f$ comes (by restriction) from a (unique) $f \in V_{\kappa,L}^{an}$.
\end{propen}

\begin{proof}[Proof]
$f \in V_{0,\kappa,L}^{p^n-an}$ for a certain $n$, in particular, it defines a function,
\[ f : 
\left(
\begin{array}{ccc}
a  & u &   \\
p^n\mathcal O  &  b &   \\
  &   &   c
\end{array}
\right) = I_n \fleche L
\] which is identified to a function in $\mathcal F^{an}(p^n\mathcal O,L)$. But $f$ is an eigenform for $\delta$ with eigenvalue $\lambda \neq 0 \in L$, thus $f = \lambda^{-1}\delta(f)$.
But if $f = f(z)$, with the identification  to $\mathcal F^{an}(p^n\mathcal O,L)$, $\delta f$ is identified with $f(pz)$, thus $f \in \mathcal F^{an}(p^{n-1}\mathcal O,L)$ i.e. $\delta^{-1}$ strictly increases the analyticity radius, and by iterating, $f \in \mathcal F^{an}(p\mathcal O,L)$, thus $f \in V_{\kappa,L}^{an}$. 
\end{proof}

\begin{propen}
For $\kappa = (k_1,k_2,r) \in X^+(T)$, there is an inclusion,
\[ V_{\kappa,L} \subset V_{\kappa,L}^{an},\]
which under the identification of $V_{\kappa,L}^{an}$ with $\mathcal F^{an}(p\mathcal O,L)$ identifies $V_\kappa$ with polynomial functions of degree less or equal than $k_1-k_2$.
\end{propen}

\begin{propen}
Let $\kappa = (k_1,k_2,r) \in X^+(T)$. The following sequence is exact,
\[ 0 \fleche V_{\kappa,L} \fleche V_{\kappa,L}^{an} \overset{d_\kappa}{\fleche} V_{(k_2-1,k_1+1,r),L}^{an},\]
where $d_\kappa$ is given by,
\[ f \in V_{\kappa,L}^{an} \longmapsto X^{k_1-k_2+1}f,\]
and \[Xf(g) = \left(\frac{d}{dt}f\left(g
\left(
\begin{array}{ccc}
1  &    &  \\
  -t&   1 & \\
  & & 1
\end{array}
\right)\right)
\right)_{t=0}\]
\end{propen}

\begin{proof}[Proof]
Let us first check that $d_\kappa$ is well defined.
Indeed, using $(k_1-k_2+1)$-times the formula,
\begin{eqnarray*} (Xf)(g\left(
\begin{array}{ccc}
t_1 &   &   \\
  &  t_2& \\
  & & 1
\end{array}
\right)) =  \\
\left(\frac{d}{dt}f\left(g\left(
\begin{array}{ccc}
t_1 &      &\\
  &  t_2& \\
  & & 1
\end{array}
\right)
\left(
\begin{array}{ccc}
1  &   &   \\
  -t&   1& \\
  & & 1
\end{array}
\right)
\left(
\begin{array}{ccc}
t_1^{-1} &    &  \\
  &  t_2^{-1}& \\
  & & 1
\end{array}
\right)\left(
\begin{array}{ccc}
t_1 &  &    \\
  &  t_2&\\
  &  & 1
\end{array}
\right)\right)\right)_{t=0}\\
= \left(\frac{d}{dt}f\left(g
\left(
\begin{array}{ccc}
1  &   &   \\
  -t_2t_1^{-1}t&   1&\\
  & & 1
\end{array}
\right)\right)t_1^{k_1}t_2^{k_2}
\right)_{t=0} = t_2t_1^{-1} (Xf)(g)t_1^{k_1}t_2^{k_2} = t_1^{k_1-1}t_2^{k_2+1} (Xf)(g)
\end{eqnarray*}
(and the corresponding formula for the action of $\left(
\begin{array}{ccc}
1  & u   &  \\
  &   1&\\
  & & 1
\end{array}
\right)$) we deduce that $d_\kappa f$ has the right weight.

We can check (evaluating on $\left(
\begin{array}{ccc}
1  &   &   \\
  u &   1&\\
  & & 1
\end{array}
\right)$) that on $\mathcal F^{an}(p\mathcal O,L)$ $d_\kappa$ correspond to $(\frac{d}{dz})^{k_1-k_2+1}$, where $z$ is the variable on $p\mathcal O$. Thus, using the previous identification with $V_\kappa$ and polynomials of degree less or equal than $k_1-k_2$, we deduce can check that $V_\kappa$ is exactly the kernel of $d_\kappa$.
\end{proof}

\begin{remaen}
A more general version of the previous proposition has been developed by Urban, \cite{Urb} Proposition 3.2.12, or Jones \cite{Jones}.
\end{remaen}

\section{Hasse Invariants and the canonical subgroups}

Let $p$ be a prime. Fix $E \subset \overline{E} \subset \CC$ an algebraic closure of $E$ and fix an isomorphism $\CC \simeq \overline{\QQ_p}$. Call $\tau,\sigma\tau$ the two places of $\overline{\QQ_p}$ that corresponds respectively to $\tau_\infty,\sigma\tau_\infty$ through the previous isomorphism (sometimes if $p$ splits in $E$ we will instead right $v = \tau$ and $\overline v = \sigma\tau$ following the notation of \cite{BC1}).
Suppose now $p$ is inert in $E$. Let us take $K=K^pK_p \subset G(\mathbb A_f)$ a sufficiently small compact open, such that $K_p$ is hyperspecial, and denote $Y = Y_K/\Spec(\mathcal O)$ the integral model of Kottwitz (\cite{Kotjams}) of the Picard Variety associated to the Shimura datum of the first section and the level $K$ (recall $\mathcal O = \mathcal O_{E,p}$). Denote $\mathcal I = \Hom(\mathcal O,\CC_p) = \{\tau,\sigma\tau\}$ the set of embeddings of $\mathcal O$ into an algebraic closure of $\QQ_p$ ($\CC_p = \widehat{\overline{\QQ_p}}$), where $\sigma$ is the Frobenius of $\mathcal O$, which acts transitively on $\mathcal I$, and $\Gal(E/\QQ) = \{\id,\sigma\}$.

Recall the (toroidal compactification of the) Picard modular surface $X = X_K$ is the (compactified) moduli space of principally polarized abelian varieties $\mathcal A \fleche S$ of genus 3, endowed with an action of $\mathcal O_E$, and a certain level structure $K^p$, and such that, up to extending scalars of $S$, we can decompose the conormal sheaf of $\mathcal A$ under the action of $\mathcal O = \mathcal O_{E,p}$,
\[ \omega_\mathcal A = \omega_{\mathcal A,\tau} \oplus \omega_{\mathcal A,\sigma\tau},\]
and we assume $\dim_{\mathcal O_S} \omega_{\mathcal A,\tau} =1$ (and thus $\dim_{\mathcal O_S} \omega_{\mathcal A,\sigma\tau} = 2$). These two sheaves extend from $Y$ to $X$. One reason is that there is a semi-abelian scheme on $X$ together with an action of $\mathcal O_E$, thus its conormal sheaf extend $\omega_A$ to the boundary, and the $\mathcal O_E$ action allows the splitting to make sense on the boundary too (see for example \cite{Lan}, Theorem 6.4.1.1).

\begin{remaen}
If $p$ splits in $E$, there is also a integral model of the Picard Surface, which is above $\Spec(\ZZ_p)$, and it has a similar description (of course in this case $\mathcal O_{E,p} \simeq \ZZ_p\times\ZZ_p$).
\end{remaen}

\begin{remaen}
The case of toroïdal compactifications of the Picard modular surface is particularly simple. Contrary to the case of other (PEL) Shimura varieties, there is a unique and thus canonical choice of a rational polyhedral cone decomposition $\Sigma$, and thus a unique toroïdal compactification. This cone decomposition is the one of $\RR_+$, it is smooth and projective. As a result, the toroïdal compactification is smooth and projective.
\end{remaen}

\subsection{Classical modular sheaves and geometric modular forms}

On $X$, there is a sheaf $\omega$, the conormal sheaf of $\mathcal A$, the universal (semi-)abelian scheme, along its unit section, and $\omega = \omega_\tau \oplus \omega_{\sigma\tau}$.

For any $\kappa = (k_1,k_2,k_3) \in \ZZ^3$ such that $k_1 \geq k_2$, is associated a "classical" modular sheaf, 
\[ \omega^\kappa = \Sym^{k_1-k_2}\omega_{\sigma\tau} \otimes (\det \omega_{\sigma\tau})^{k_2} \otimes \omega_{\tau}^{k_3}.\]
Denote by $\kappa' = (-k_2,-k_1,-k_3)$, this is still a dominant weight, and $\kappa \mapsto \kappa'$ is an involution. There is another way to see the classical modular sheaves.

Denote by $\mathcal T = \Hom_{X,\mathcal O}(\mathcal O_X^2\otimes\mathcal O_X,\omega)$ where $\mathcal O$ acts by $\sigma\tau$ on the first 2-dimensional factor and $\tau$ on the other one. Denote $\mathcal T^\times =  \Isom_{X,\mathcal O}(\mathcal O_X^2\otimes\mathcal O_X,\omega)$, the $\GL_2\times\GL_1$-torsor of trivialisations of $\omega$ as a $\mathcal O$-module. There is an action on $\mathcal T$ of $\GL_2\times\GL_1$ by $g \cdot w = w \circ g^{-1}$.

Denote by $\pi : \mathcal T^\times \fleche X$ the projection. For any dominant $\kappa$ as before, define,
\[ \omega^\kappa = \pi_*\mathcal O_{\mathcal T^\times}[\kappa'],\]
the subsheaf of $\kappa'$-equivariant functions for the action of the upper triangular Borel $B \subset \GL_2\times \GL_1$. As the notation suggests, there is an isomorphism, if $\kappa = (k_1,k_2,k_3)$,
\[ \pi_*\mathcal O_{\mathcal T^\times}[\kappa'] \simeq \Sym^{k_1-k_2}\omega_{\sigma\tau} \otimes (\det \omega_{\sigma\tau})^{k_2} \otimes \omega_{\tau}^{k_3}.\]

\begin{definen}
Recall that $X$ is the (compactified) Picard variety of level $K = K_pK^p$. The global sections $H^0(X,\omega^\kappa)$ is the module of Picard modular forms of level $K$ and weight $\kappa$. If $D$ denotes the (reduced) boundary of $X$, the submodule  $H^0(X,\omega^\kappa(-D))$ is the submodule of Picard cusp-forms.
\end{definen}

In the sequel we will be interested in the case $K_p = I$, the Iwahori subgroup when we speak about modular forms on the rigid space associated to $X$. But for moduli-theoretic reason we will not directly consider an integral model of $X_I^{rig}$, but only an integral model of a particular open subset (the one given by the canonical filtration). Indeed, in this particular setting ($U(2,1)$ with $p$ inert), the canonical filtration identifies a strict neighborhood of the $\mu$-ordinary locus in $X$ with an open in the level $I$-Picard variety.

\begin{remaen}
There is a more general construction of automorphic sheaves $\omega^{k_1,k_2,k_3,r}$ given in \cite{HarVB}, they are independant of $r$ as sheaves on the Picard Variety, only the $G$-equivariant action (and thus the action of the Hecke operators) depends on $r$. We will only use the previous definition of the sheaves. We could get more automorphic forms by twisting by the norm character (which would be equivalent to twist the action of the Hecke operators). 
\end{remaen}

\subsection{Local constructions}

Let $G$ be the $p$-divisible group of the universal abelian scheme over $Y \subset X$. Later we will explain how to extend our construction to all $X$.
$G$ is endowed with an action of $\mathcal O$, and we have that his signature is given by,
\[
\left\{
\begin{array}{cc}
p_\tau = 1  & q_\tau = 2   \\
 p_{\sigma\tau} = 2  & q_{\sigma\tau} = 1   
\end{array}
\right.
\]
which means that if we denote $\omega_G = \omega_{G,\sigma\tau}\oplus \omega_{G,\tau}$, the two pieces have respective dimensions $p_{\sigma\tau} = 2$ and $p_\tau =1$. Moreover $G$ carries a polarisation $\lambda$, such that $\lambda : G \overset{\sim}{\fleche} G^{D,(\sigma)}$ is $\mathcal O$-equivariant.

The main result of \cite{Her1}, see also \cite{GN}, is the following,
\begin{definen}
There exists sections,
\[ \widetilde{\ha_{\sigma\tau}} \in H^0(Y\otimes \mathcal O/p,\det(\omega_{G,\sigma\tau})^{\otimes(p^2-1)}) \quad \text{and} \quad \widetilde{\ha_{\tau}} \in H^0(Y\otimes \mathcal O/p,(\omega_{G,\tau})^{\otimes(p^2-1)}),\]
such that $\widetilde{\ha_{\tau}}$ is given by (the determinant of) $V^2$,
\[ \omega_{G,\tau} \overset{V}{\fleche} \omega_{G,\sigma\tau}^{(p)} \overset{V}{\fleche} \omega_{G,\tau}^{(p^2)},\]
and $\widetilde{\ha_\sigma\tau}$ is given by a division by $p$ on the Dieudonne crystal of $G$ of $V^2$, restricted to a lift of the Hodge Filtration $\omega_{G^D,\sigma\tau}$.
\end{definen}

On $Y \otimes \mathcal O/p$, the sheaves $\omega_{G,\tau}$ and $\omega_\tau$ (respectively $\omega_{G,\sigma\tau}$ and $\omega_{\sigma\tau}$) are isomorphic. In fact, the previous sections extends to all $X$, for example by Koecher's principle, see \cite{LanIMRN} Theorem 8.7.

\begin{remaen}
\begin{enumerate}
\item These sections are Cartier divisors on $X$, i.e. they are invertible on an open and dense subset (cf. \cite{Her1} Proposition 3.22 and \cite{Wed1}).
\item Because of the $\mathcal O$-equivariant isomorphism $\lambda : G \simeq G^{D,(\sigma)}$, and the compatibility of $\widetilde{\ha_\tau}$ with duality (see \cite{Her1}, section 1.10), we deduce that,
\[\widetilde{\ha_\tau}(G) = \widetilde{\ha_\tau}(G^D) = \widetilde{\ha_\tau}(G^{(\sigma)}) = \widetilde{\ha_{\sigma\tau}}(G).\]
Thus, we could use only $\widetilde{\ha_{\sigma\tau}}$ or $\widetilde{\ha_{\tau}}$ and define it in this case without using any crystalline construction. We usually denote by $\widetilde{^{\mu}{\ha}} =  \widetilde{\ha_\tau} \otimes\widetilde{\ha_{\sigma\tau}}$, but because of the this remark, we will only use $\widetilde{\ha_\tau}$ in this article (which is then reduced, see the appendix).
\item We use the notation $\widetilde{\bullet}$ to denote the previous global sections, but if we have $G/\Spec(\mathcal O_{\CC_p}/p)$ a $p$-divisible $\mathcal O$-module of signature $(2,1)$, we will also use the notation $\ha_\tau(G) = v(\widetilde{\ha_\tau(G)})$, where the valuation $v$ on $\mathcal O_{\CC_p}$ is normalized such that $v(p) = 1$ and truncated by 1.
\end{enumerate}
\end{remaen}

\begin{definen}
We denote by $\overline{X}^{ord}$ the $\mu$-ordinary locus of $\overline X = X\otimes_{\mathcal O} \mathcal O/p$, which is $\{ x \in \overline X : \widetilde{\ha_\tau} \text{ is invertible}\}$. It is open and dense (see \cite{Wed1}).
\end{definen}

Let us recall the main theorem of \cite{Her2} in the simple case of Picard varieties. Recall that $p$ still denotes a prime, inert in $E$, and suppose $p > 2$.

\begin{theoren}
\label{thrfiltcan}
Let $n \in \NN^\times$. Let $H/\Spec(\mathcal O_L)$, where $L$ is a valued extension of $\QQ_p$, a truncated $p$-divisible $\mathcal O$-module of level $n+1$ and signature $(p_{\tau} = 1,p_{\sigma\tau} =2)$. Suppose,
\[ \ha_\tau(H) < \frac{1}{4p^{n-1}}.\]
Then there exists a unique filtration (socalled "canonical" of height $n$) of $H[p^n]$,
\[ 0 \subset H_{\tau}^n \subset H_{\sigma\tau}^n \subset H[p^n],\]
by finite flat sub-$\mathcal O$-modules of $H[p^n]$, of $\mathcal O$-heights $n$ and $2n$ respectively.
Moreover,
\[ \deg_{\sigma\tau}(H_{\sigma\tau}^n) + p\deg_{\tau}(H_{\sigma\tau}^n) \geq n(p +2) - \frac{p^{2n}-1}{p^2-1}\ha_\tau(H),\]
and
\[ \deg_{\tau}(H_{\tau}^n) + p\deg_{\sigma\tau}(H_{\tau}^n) \geq n(2p +1) - \frac{p^{2n}-1}{p^2-1}\ha_\tau(H).\]
In particular, the groups $H_\tau^n$ and $H_{\sigma\tau}^n$ are of high degree. In addition, points of $H_\tau^n$ coïncide with the kernel of the Hodge-Tate map 
$\alpha_{H[p^n],\tau,n - \frac{p^{2n}-1}{p^2-1}\ha_\tau(H)}$ and $H_{\sigma\tau}^n$ with the one of $\alpha_{H[p^n],\sigma\tau,n - \frac{p^{2n}-1}{p^2-1}\ha_\tau(H)}$.
They also coïncide with steps of the Harder-Narasihman filtration % (associated respectively to $\tau$ and $\sigma\tau$) 
and are compatible with $p^s$-torsion ($s \leq n$) 
and quotients.
\end{theoren}

\begin{proof}
This is exactly \cite{Her2} if $p \geq 5$. Here the bounds on $p$ is slightly better than the one of \cite{Her2}. The reason is that we construct $H_{\sigma\tau}^n$ using Théorème 8.3 and Remarque 8.4 of \cite{Her2} (with our bound for $\ha_\tau$). We can then construct $D_\tau^n \subset G[p^n]^D$ of $\mathcal O$-height $2n$ by the same reason. Then setting,
\[ H_\tau^n := (D_\tau^n)^\perp = (G[p^n]^D/D_\tau^n)^D \subset G[p^n],\]
we can have the asserted bound on $\deg_\tau H_\tau^n$. For the assertion on the Kernel of the Hodge-Tate map, we get an inclusion with the bound on the degree. We can then use \cite{Her2} proposition 7.6. Then by remarque 8.4 of \cite{Her2}, these groups are steps of the Harder-Narasihman filtration (the classical one), and thus we get the assertion of the inclusion $H_\tau^n \subset H_{\sigma\tau}^n$. By unicity of the Harder-Narasihman filtration, we get that $(H_\tau^n)^\perp = H_{\sigma\tau}^n$.
\end{proof}

\begin{definen}
Let $H/\Spec(\mathcal O_L)$ as before, with $n = 2m$. Then we can consider inside $H[2m]$ the finite flat subgroup,
\[ K_m = H_{\tau}^{2m} + H_{\sigma\tau}^m.\]
It coïncides, after reduction to $\Spec(\mathcal O_L/\pi_L)$ (the residue field of $L$) with the kernel of $F^{2m}$ of $H[p^{2m}]$ (see \cite{Her2}, section 2.9.1).
\end{definen}

Recall that we denoted $X/\Spec(\mathcal O)$ the (schematic) Picard surface. Denote by $X^{rig}$ the associated rigid space over $E_p$, there is a specialisation map,
\[ \Sp : X^{rig} \fleche \overline{X},\]
and we denote by $X^{ord} \subset X^{rig}$ the open subspace defined by $\Sp^{-1}(\overline{X}^{ord})$.

Let us denote, for $v \in (0,1]$,
\[ X(v) = \{ x \in X^{rig} : \ha_\tau(x) = v(\widetilde{\ha_\tau}(G_x)) < v\} \quad \text{and} \quad X(0) = X^{ord},\]
the strict neighbourhoods of $X^{ord}$. The previous theorem and technics introduced in \cite{FarHN} (see \cite{Her2} section 2.9) implies, if $v \leq \frac{1}{4p^{n-1}}$, that we
have a filtration in families over the rigid space $X(v)$,
\[ 0 \subset H_{\tau}^n \subset H_{\sigma\tau}^n \subset G[p^n].\]
Let us explain how to get this result. On $Y^{rig}(v)$ this is simply \cite{Her2} Théorème 9.1 which is essentially \cite{FarHN} Théorème 4 (again, see proof of Theorem \ref{thrfiltcan} about the bound). The problem is that the $p$-divisible group $G$ doesn't extend to the boundary. But by results of Stroh \cite{Stroh} section 3.1, there exists 
\[  \overline U \fleche X^{rig},\]
which is an etale covering of $X^{rig}$ (actually $\overline U$ is algebraic and exists also integrally) together with $\overline R \overset{p_1}{\underset{p_2}{\rightrightarrows}} \overline U$ etale maps, such that
\[ X^{rig} \simeq [\overline{U}/\overline R].\]
Over $\overline U$ there is a Mumford 1-motive $M = [L \fleche \widetilde G]$ such that $M[p^n] = A[p^n]$ ($A$ is the semi-abelian scheme), and thus there is a canonical $\widetilde G$ semi abelian scheme with an action of $\mathcal O_E$ of (locally) constant toric rank and thus $\widetilde G[p^n]$ is finite flat. Thus applying to $\widetilde G[p^n]$ the results of \cite{Her2}, there exists on $\overline{U}(v) = \overline U \times_X X(v)$ the two groups $H_\tau^n \subset H_{\sigma\tau}^n$. 
Moreover, over $R(v) \overset{p_1}{\underset{p_2}{\rightrightarrows}} U(v)$, we have $pr_1^*H_\star^n = pr_2^*H_\star^n$ as $H_\star^n$ descend to $Y(v)$. Over $\overline R$ we have an isomorphism $pr_1^*G = pr_2^*G$ (\cite{Stroh} 3.1.5), and thus we have over $R(v)$ an isomorphism
\[ G/pr_1^*H_\star^n \overset{\sim}{\fleche} G/pr_2^*H_\star^n.\]
As $\overline R$ is normal (because smooth) and $R$ is dense in $\overline R$, by \cite{Stroh} Théorème 1.1.2 (due to Faltings-Chai) there is an isomorphism over $\overline R$
\[G/pr_1^*H_\star^n \overset{\sim}{\fleche} G/pr_2^*H_\star^n.\]
Thus, by faithfully flat descent, $H_{\tau}^n \subset H_{\sigma\tau}^n$ exists on $X(v)$. In particular, $K_m = H_{\sigma\tau}^m + H_\tau^{2m}$ also exists on $X(v)$.

\textit{A priori}, this filtration does not extend to a formal model of $X(v)$, but as $X/\Spec(\mathcal O)$ is a normal scheme, we will be able to use the following proposition.

\begin{definen}
For $K/\QQ_p$ an extension, define the category $\mathfrak{Adm}$ of admissible $\mathcal O_K$-algebra, i.e. flat quotient of power series ring $\mathcal O_K<<X_1,\dots,X_r>>$ for some $r \in \NN$. Define $\mathfrak{NAdm}$ the sub-category of normal admissible $\mathcal O_K$-algebra.
\end{definen}

\begin{propen}
\label{propKernormal}
Let $m$ be an integer, $S = \Spf R$ a normal formal scheme over $\mathcal O$, and $G \fleche S$ a truncated $p$-divisible $\mathcal O$-module of level $2m+1$ and signature 
$(p_\tau = 1,p_{\sigma\tau} =2)$. Suppose that for all $x \in S^{rig}$, $\ha_\tau(x) < \frac{1}{4p^{2m-1}}$. Then the subgroup 
$K_m := H_{\tau}^{2m} + H_{\sigma\tau}^m \subset G[p^{2m}]$ of $S^{rig}$ entends to $S$.
\end{propen}

\begin{proof}[Proof]
As we know that $K_m$ coincide with the Kernel of Frobenius on points, this is exactly as \cite{AIP}, proposition 4.1.3.
\end{proof}

\section{Construction of torsors}

\subsection{Hodge-Tate map and image sheaves}

Let $p$ a prime, inert in $E$, and $\mathcal O = \mathcal O_{E,p}$, a degree 2 unramified extension of $\ZZ_p$. Let $K$ be a valued extension of $E_p$. Let $m \in \NN^\times$ and $v < \frac{1}{4p^{2m-1}}$.
Let $S = \Spec(R)$ where $R$ is an object of $\mathfrak{NAdm}/\mathcal O_K$, and $G \fleche S$ a truncated $p$-divisible $\mathcal O$-module of level $2m$ and signature,
\[
\left\{
\begin{array}{cc}
p_\tau = 1  & q_\tau = 2   \\
 p_{\sigma\tau} = 2  & q_{\sigma\tau} = 1   
\end{array}
\right.
\]
where $\tau : \mathcal O \fleche \mathcal O_S$ is the fixed embedding. Suppose moreover that for all $x \in S^{rig}$, $\ha_\tau(x) \leq v$. According to the previous section, there 
exists on $S^{rig}$ a filtration of $G[p^{2m}]$ by finite flat $\mathcal O$-modules,
\[ 0 \subset H_{\tau}^{2m} \subset H_{\sigma\tau}^{2m} \subset G[p^{2m}],\]
of $\mathcal O$-heights $2m$ and $4m$ respectively.
Moreover, we have on $S$ a subgroup $K_m \subset G[p^{2m}]$, finite flat of $\mathcal O$-height $3m$, etale-locally isomorphic (on $S^{rig}$) to 
$\mathcal O/p^{2m}\mathcal O \oplus\mathcal O/p^m\mathcal O$, and on $S^{rig}$, $K_m = H_{\tau}^{2m} + H_{\sigma\tau}^{2m}[p^m]$.

\begin{propen}
\label{prop61}
Let $w_\tau,w_{\sigma\tau} \in v(\mathcal O_K)$ such that $w_{\sigma\tau} < m-\frac{p^{2m}-1}{p^2-1}v$ and $w_{\tau} < 2m - \frac{p^{4m}-1}{p^2-1}v$.
Then, the morphism of sheaves on $S$ $\pi : \omega_G \fleche \omega_{K_m}$, induce by the inclusion $K_m \subset G$, induces isomorphisms,
\[ \pi_\tau : \omega_{G,\tau,w_\tau} \overset{\sim}{\fleche} \omega_{K_m,\tau,w_\tau} \quad \text{and}\quad \pi_{\sigma\tau} : \omega_{G,\sigma\tau,w_{\sigma\tau}} \overset{\sim}{\fleche} \omega_{K_m,\sigma\tau,w_{\sigma\tau}}.\]
\end{propen}

\begin{proof}[Proof]
If $G/\Spec(\mathcal O_C)$ ($C$ a complete algebraically closed extension of $\QQ_p$), the degrees of the canonical filtration of $G$ assure that,
\[ \deg_{\sigma\tau}(G[p^{m}]/H_{\sigma\tau}^{m}) \geq \frac{p^{2m}-1}{p^2-1}v \quad \text{and} \quad \deg_{\tau}(G[p^{2m}]/H_{\tau}^{2m}) \geq \frac{p^{4m}-1}{p^2-1}v,\]
and there is thus an isomorphism,
\[ \omega_{G[p^m],\tau,w_\tau} \overset{\sim}{\fleche} \omega_{H_\tau^n,\tau,w_\tau},\]
and also for $\sigma\tau$ and $G[p^{2m}]$. But there are inclusions $H_{\sigma\tau}^m=H_{\sigma\tau}^{2m}[p^m] \subset K_m \subset G$ and $H_{\tau}^{2m} \subset K_m \subset G$ such that the composite,
\[ \omega_{G,\sigma\tau,w_{\sigma\tau}}\fleche \omega_{K_m,\sigma\tau,w_{\sigma\tau}} \fleche \omega_{H_{\sigma\tau}^m,\sigma\tau,w_{\sigma\tau}},\]
is an isomorphism, which implies that the first one is. The same reasoning applies for $\tau$. We can thus conclude for $S$ as in \cite{AIP} proposition 4.2.1 : Up to reducing $R$ 
we can suppose $\omega_G$ is a free $R/p^{2m+1}$-module, and look at the surjection 
$\alpha_{\sigma\tau} : R^2 \twoheadrightarrow \omega_{G,\sigma\tau} \twoheadrightarrow \omega_{K_m,\sigma\tau,w_{\sigma\tau}}$, 
it is enough to prove that for any $(x_1,x_2)$ in $\ker \alpha_{\sigma\tau}$ we have $x_i \in p^{w_\tau}R$, but as $R$ is normal, it suffice to do it for $R_\mathfrak p$, and even for 
$\widehat{R_\mathfrak p}$, for all codimension 1 prime ideal $\mathfrak p$ that contains $(p)$. But now $\widehat{R_\mathfrak p}$ is a complete, discrete valuation ring of 
mixed characteristic, and this reduce to the preceding assertion.
\end{proof}

\begin{propen}
\label{prop62}
Suppose there is an isomorphism $K_m^D(R) \simeq \mathcal O/p^m\mathcal O \oplus \mathcal O/p^{2m}\mathcal O$. Then the cokernel of the $\sigma\tau$-Hodge-Tate map,
\[ \HT_{K_m^D,\sigma\tau} \otimes 1 : K_m^D(R)[p^m]\otimes_\mathcal O R \fleche \omega_{K_m,\sigma\tau},\]
is killed by $p^{\frac{p+v}{p^2-1}}$, and the cokernel of the $\tau$-Hodge-Tate map,
\[ \HT_{K_m^D,\tau} \otimes 1 : K_m^D(R)[p^m]\otimes_\mathcal O R \fleche \omega_{K_m,\tau},\]
is killed by $p^{\frac{v}{p^2-1}}$.
\end{propen}

\begin{proof}[Proof]
This is true for $G/\Spec(\mathcal O_C)$ by, \cite{Her2} Théorème 6.10 (2),% applied to $G^D$ (to exchange $\tau$ and $\sigma\tau$), 
with the previous proposition (because $\frac{p+v}{p^2-1} < 1 - v$, already for $m=1$).
For a general normal $R$, we can reduce to previous case (see also \cite{AIP} proposition 4.2.2) : up to reduce $\Spec(R)$, we have a diagram,
\begin{center}
\begin{tikzpicture}[description/.style={fill=white,inner sep=2pt}] 
\matrix (m) [matrix of math nodes, row sep=3em, column sep=2.5em, text height=1.5ex, text depth=0.25ex] at (0,0)
{ 
R^2 & & R^2 \\
K_n^D(R)[p^n]\otimes_\mathcal O R & &  \omega_{K_n,\sigma\tau}\\
};
\path[->,font=\scriptsize] 
(m-1-1) edge node[auto] {$\gamma$} (m-1-3)
(m-2-1) edge node[auto] {$\HT_{\sigma\tau} \otimes 1$} (m-2-3)
;
\path[->>,font=\scriptsize] 
(m-1-1) edge node[auto] {$$} (m-2-1)
(m-1-3) edge node[auto] {$$} (m-2-3)
;
\end{tikzpicture}
\end{center}
and $\Fitt^1(\gamma)$ (which is just a determinant here) annihilates the cokernel of $\gamma$, and it suffices to prove that $p^{\frac{p+v}{p^2-1}} \in \Fitt^1(\gamma)$. But as $R$ is normal, it suffice to prove that 
$p^{\frac{p+v}{p^2-1}} \in \Fitt^1(\gamma)R_\mathfrak p$ for every codimension 1 prime ideal $\mathfrak p$ that contains $(p)$. But by the previous case, we can conclude. The same works for $\tau$.
\end{proof}

\begin{propen}
\label{prop63}
Suppose we have an isomorphism $K_m^D(R) \simeq \mathcal O/p^m\mathcal O \oplus \mathcal O/p^{2m}\mathcal O$. Then there exists on $S = \Spec R$ locally free subsheaves 
$\mathcal F_{\sigma\tau}, \mathcal F_{\tau}$ of $\omega_{G,\sigma\tau}$ and $\omega_{G,\tau}$ respectively, of ranks 2 and 1, which contains 
$p^{\frac{p+v}{p^2-1}}\omega_{G,\sigma\tau}$ and $p^{\frac{v}{p^2-1}}\omega_{G,\tau}$, and which are equipped, for all $w_{\sigma\tau} < m - \frac{p^{2m}-1}{p^2-1}v$ and $w_{\tau} < 2m - \frac{p^{4m}-1}{p^2-1}v$, with maps,
\[ \HT_{\sigma\tau,w_{\sigma\tau}} : K_m^D(R) \fleche \mathcal F_{\sigma\tau} \otimes_R R_{w_{\sigma\tau}}, \quad \text{and} \quad \HT_{\tau,w_{\tau}} : K_m^D(R) \fleche \mathcal F_{\tau}\otimes_R R_{w_{\tau}},\]
which are surjective after tensoring $K_m^D(R)$ with $R$ over $\mathcal O$.

More precisely, via the projection,
\[ K_m^D(R) \twoheadrightarrow (H_{\tau}^{2m})^D(R_K),\]
we have induced isomorphisms,
\[ \HT_{\sigma\tau,w_{\sigma\tau}} : K_m^D(R_K) \otimes_\mathcal O R_{w_{\sigma\tau}} \fleche \mathcal F_{\sigma\tau} \otimes_R R_{w_{\sigma\tau}},\]
and 
\[ \HT_{\tau,w_{\tau}} : (H_{\tau}^{2m})^D(R_K) \otimes_\mathcal O R_{w_{\tau}} \fleche \mathcal F_{\tau}\otimes_R R_{w_{\tau}}.\]

\end{propen}

\begin{proof}[Proof]
This is the same construction as \cite{AIP} proposition 4.3.1. To check the assertion about the isomorphism with $H_{\tau}^{2m}$, it suffices to show that the map $\HT_{\tau,w_{\tau}}$ factors, but it is true over $R_\mathfrak p$ (as the canonical filtration is given by kernels of Hodge-Tate maps) for every codimension 1 ideal $\mathfrak p$, and it is moreover surjective, so it globally factors and is globally surjective, but the two free $R_{w_{\tau}}$-modules are free of the same rank 1, so it is an isomorphism.
\end{proof}

Moreover the construction of the sheaves $\mathcal F$ is functorial in the following sense,

\begin{propen}
\label{prop64}
Suppose given $G,G'$ two truncated $p$-divisible $\mathcal O$-module such that for all $x \in S^{rig}$, \[\ha_\tau(G_x), \ha_\tau(G'_x) < v,\] and an $\mathcal O_E$-isogeny,
\[ \phi : G \fleche G'.\]
Assume moreover that we are given trivialisations of the points of $K_m^D(G)$ and $K_m^D(G')$. Then $\phi^*$ induces maps
\[ \phi^*_\tau : \mathcal F_\tau' \fleche \mathcal F_{\tau}'\quad\text{and}\quad \phi^*_{\sigma\tau} : \mathcal F_{\sigma\tau}' \fleche \mathcal F_{\sigma\tau},\]
that are compatible with inclusion in $\omega$, reduction modulo $p^w$ and the Hodge-Tate maps of $K_m^D$.
\end{propen}

\begin{proof}[Proof]
Once we know that $\phi$ will send $K_m^D(G)$ inside $K_m^D(G')$  
this is straightforward as $\mathcal F_?$ corresponds to sections of $\omega_{G,?}$ that are modulo $p^{w_?}$ generated by the image of $\HT_?$.
But $K_m$ is generated by the subgroup $H_{\sigma\tau}^{m}$ and $H_{\tau}^{2m}$ each being a breakpoint of the Harder-Narasihman filtration $\HN_{\sigma\tau}(G[p^{m}])$ and 
$\HN_\tau(G[p^{2m}])$ respectively, and thus by functoriality of these filtrations, $\phi$ sends each subgroup for $G$ inside the one for $G'$ and thus sends $K_m(G)$ inside 
$K_m(G')$.
\end{proof}

\begin{remaen}
Strictly speaking, we can't apply this section for $S^{rig} \subset X(v)$ an affinoid. The reason is that even if $H_{\sigma\tau}^n,H_\tau^n$ descend to $X(v)$, it is not the case of $\widetilde G[p^n]$. But we can apply the results of this section with $\widetilde G[p^n]$ for $S^{rig}$ an affinoid of $\overline U$. As $\omega_{\widetilde G} \simeq  \omega \times_{X(v)} \overline{U}(v)$, it is enough to check proposition \ref{prop61} and \ref{prop62} over $\overline{U}/\Spec(\mathcal O)$. Thus proposition \ref{prop63} and \ref{prop64} remain true over $S^{rig}$ and affinoid of $X(v)$ (for proposition \ref{prop64} we will restrict anyway to the open Picard variety, but it would remain true looking at an isogeny of semi-abelian schemes).
\end{remaen}

\subsection{The torsors}
To simplifiy the notations, fix $w = w_\tau = w_{\sigma\tau} < m - \frac{p^{2m}-1}{p^2-1}v$ to use the previous propositions.
Let $R \in \mathcal O_K-\mathfrak{NAdm}$ and $S = \Spf(R)$. In rigid fiber, we have a subgroup of $K_m[p^m]/S^{rig}$, $H_{\tau}^m \subset K_m[p^m]$ which induces a filtration,
\[ 0 \subset (H_{\sigma\tau}^m/H_{\tau}^m)^D(R_K) \subset K_m^D(R),\]
of cokernel isomorphic to $(H_{\tau}^m)^D(R_K)$.

Suppose we are given a trivialisation,
\[ \psi : \mathcal O/p^m\mathcal O \oplus \mathcal O/p^{2m}\mathcal O \simeq K_m^D(R),\]
which induces trivialisations (first coordinate and quotient),
\[ \psi_{\sigma\tau} : (H_{\sigma\tau}^m/H_{\tau}^m)^D(R_K) \simeq \mathcal O/p^m\mathcal O \quad \text{and} \quad 
\psi_{\tau} : (H_{\tau}^{2m})^D(R_K) \simeq \mathcal O/p^{2m}\mathcal O.\]

Let $\mathcal Gr_{\sigma\tau} \fleche S$ be the Grassmanian of locally direct factor sheaves of rank 1, $\Fil^1\mathcal F_{\sigma\tau} \subset \mathcal F_{\sigma\tau}$. 
Let $\mathcal Gr_{\sigma\tau}^+ \fleche \mathcal Gr_{\sigma\tau}$ the $\mathbb G_m^2$-torsor of trivialisations of $\Fil^1\mathcal F_{\sigma\tau}$ and $\mathcal F_{\sigma\tau}/\Fil^1\mathcal F_{\sigma\tau}$. Let also $\mathcal Gr_{\tau}^+ \fleche S$ the $\mathbb G_m$-torsor of trivialisations of $\mathcal F_{\tau}$.

\begin{definen}
We say that a $A$-point of $\mathcal Gr_{\sigma\tau}$ (respectively $\mathcal Gr_{\sigma\tau}^+$ or $\mathcal Gr_{\tau}^+$), 
\[ \Fil^1(\mathcal F_{\sigma\tau}\otimes A) \quad (\text{respectively } (\Fil^1(\mathcal F_{\sigma\tau}\otimes A),P_1^{\sigma\tau},P_2^{\sigma\tau}) \text{ or } P^{\tau})\] is $w$-compatible with $\psi_\tau, \psi_{\sigma\tau}$ if respectively,
\begin{enumerate}
\item $\Fil^1(\mathcal F_{\sigma\tau}\otimes A) \otimes_R R_w = \HT_{\sigma\tau,w}((H_{\sigma\tau}^m/H_{\tau}^m)^D(R_K)\otimes_\mathcal O R_w)\otimes_R A$,
\item $P_1^{\sigma\tau} \otimes_R R_w = \HT_{\sigma\tau,w} \circ(\psi_{\sigma\tau} \otimes_\mathcal O R_w)\otimes_R A$,
\item $P_2^{\sigma\tau} \otimes_R R_w = \HT_{\sigma\tau,w} \circ(\psi_{\tau} \otimes_\mathcal O R_w)\otimes_R A$,
\item $P^{\tau} \otimes_R R_w = \HT_{\tau,w} \circ (\psi_{\tau} \otimes_\mathcal O R_w)\otimes_R A.$
\end{enumerate}
\end{definen}

We can define the functors,
\[
\mathfrak{IW}_{\sigma\tau,w} :
\begin{array}{ccc}
R-\mathfrak{Adm} & \fleche  & SET  \\
 A & \longmapsto  & \{ w-\text{compatible } \Fil^1 (\mathcal F_{\sigma\tau} \otimes_R A) \in \mathcal Gr_{\sigma\tau}(A)\},
\end{array}
\]
\[
\mathfrak{IW}_{\sigma\tau,w}^+ :
\begin{array}{ccc}
R-\mathfrak{Adm} & \fleche  & SET  \\
 A & \longmapsto  & \{ w-\text{compatible } (\Fil^1 (\mathcal F_{\sigma\tau} \otimes_R A),P_1^\tau,P_2^\tau) \in \mathcal Gr_{\sigma\tau}^+(A)\},
\end{array}
\]
\[
\mathfrak{IW}_{\tau,w}^+ :
\begin{array}{ccc}
R-\mathfrak{Adm} & \fleche  & SET  \\
 A & \longmapsto  & \{ w-\text{compatible } P^{\tau} \in \mathcal Gr_{\tau}^+(A)\}.
\end{array}
\]
The previous functors are representable by formal schemes, affine over $S = \Spf(R)$, and locally isomorphic to,
\[
\left(
\begin{array}{cc}
  1 &      \\
  p^w \mathfrak B(0,1) &     1 
\end{array}
\right) \times_{\Spf(\mathcal O_K)}\Spf(R) \quad \text{for } \mathfrak{IW}_{\sigma\tau,w}, \quad  1+ p^w \mathfrak B(0,1)  \quad \text{for }  \mathfrak{IW}_{\tau,w}^+
\] where $\mathfrak B(0,1) = \Spf(\mathcal O_K<T>)$ is the formal unit 1-dimensionnal ball, and
\[\left(
\begin{array}{cc}
  1 +  p^w \mathfrak B(0,1)  &     \\
  p^w \mathfrak B(0,1) &  1+  p^w \mathfrak B(0,1) 
\end{array}
\right) \times_{\Spf(O_K)}\Spf(R) \quad \text{for } \mathfrak{IW}_{\sigma\tau,w}^+ 
\]
We also define $\mathfrak{IW}_w^+ = \mathfrak{IW}_{\tau,w}^+ \times_S \mathfrak{IW}_{\sigma\tau,w}^+$. The previous constructions are independent of $n = 2m$ (because $\mathcal F_\tau,\mathcal F_{\sigma\tau}$ are).

Let $T^1 = \Res_{\mathcal O/\ZZ_p} \mathbb G_m \times U(1)$ the torus of $U(2,1)$ over $\ZZ_p$ whose $\ZZ_p$-points are $\mathcal O^\times \times \mathcal O^1$. Its scalar extension $T = T^1 \otimes_{\ZZ_p} \mathcal O$ is isomorphic to $\mathbb G_m^3$, and $\mathcal Gr^+ = \mathcal Gr^+_\tau \times \mathcal Gr_{\sigma\tau}^+ \fleche \mathcal Gr =  \mathcal Gr_\tau$ is a $T$-torsor.
Denote $\mathfrak T \fleche \Spf(\mathcal O)$ the formal completion of $T$ along its special fiber, and $\mathfrak T_w$ the torus defined by,
\[ \mathfrak T_w(A) = \Ker(\mathfrak T(A) \fleche \mathfrak T(A/p^wA).\]
Then $\mathfrak{IW}_w^+ \fleche \mathfrak{IW}_{\sigma\tau,w}$ is a $\mathfrak T_w$-torsor.

Denote by $\mathcal{IW}_{\sigma\tau,w},\mathcal{IW}_{\tau,w}^+,\mathcal{IW}_{\sigma\tau,w}^+,\mathcal{IW}_w^+,\mathcal T$ the generic fibers of the previous formal schemes.

\section{The Picard surface and overconvergent automorphic sheaves}

\subsection{Constructing automorphic sheaves}

Let us consider the datum $(E,V,\psi,\mathcal O_E,\Lambda = \mathcal O_E^3,h)$ the PEL datum introduced in section 2. Let $p$ be a prime, inert in $E$ and $G$ the reductive group associated over $\ZZ_p$. We fix $K^p$ a compact open subgroup of $G(\mathbb A_f^p)$ sufficiently small and $\mathfrak C = G(\ZZ_p)$ an hyperspecial subgroup at $p$.
Let $X = X_{K^p\mathfrak C}/\Spec(\mathcal O)$ the (integral) Picard variety associated to the previous datum (cf. \cite{Kotjams},\cite{Lan},\cite{LRZ}).

Let $K/\mathcal O[1/p]$ be a finite extension (that we will choose sufficiently large) and still denote $X = X_{\mathcal O_K} = X\times_{\Spec\mathcal O} \Spec(\mathcal O_K)$.

Denote by $A$ the universal semi-abelian scheme, $X^{rig}$ the rigid fiber of $X$, $X^{ord}$ the ordinary locus and for $v \in v(K)$, $X(v)$ the rigid-analytic open $\{x \in X^{rig} : \ha_\tau(x) < v\}$. Denote also $\mathfrak X \fleche \Spf(\mathcal O_K)$ the formal completion of $X$ along its special fiber, $\widetilde{\mathfrak X(v)}$ the admissible blow up of $\mathfrak X$ along the ideal $(\widetilde{\ha_\tau},p^v)$ and $\mathfrak X(v)$ its open subscheme where $(\widetilde{\ha_\tau},p^v)$ is generated by $\widetilde{\ha_\tau}$.

\begin{lemmen}
The formal scheme $\mathfrak X(v)$ is normal.
\end{lemmen} 

\begin{proof}[Proof]
As $X(v)$ is smooth, thus normal, and $\widetilde{\ha_\tau}$ is reduced, this follow from the 
\begin{lemmen}
Let $A\in \mathcal O_K-\mathfrak{Adm}$ such that $A_K$ is normal and $A/\pi_K$ is reduced. Then $A$ is normal.
\end{lemmen}

\begin{proof}[Proof]
For $R$ a (noetherian) ring denote $R^{norm}$ its integral closure in its total ring of fractions. Denote, for all $x \in A_K$,
\[ v_A(x) = \sup \{ n \in \ZZ : \pi_K^{-n}x \in A\} \quad \text{and} \quad |x|_A = \pi_K^{-v_A(x)}.\]
Then we can check that $|f + g|_A \leq \sup(|f|_A,|g|_A)$ and that $|f^n|_A = |f|_A^n$. Indeed, $\pi_K^{-v_A(f)}f \in A\backslash \pi_K A$. Thus, as $A/\pi_KA$ is reduced, 
$(\pi_K^{-v_A(f)}f)^n \in A \backslash \pi_KA$ and thus $v_A(f^n) = n v_A(f)$. For this norm, we have that \[A = \{x \in A_K : |x|_A \leq1\}.\]
Now let us verify that $A$ is normal. Let $x \in A^{norm}$, in particular, $x\in A_K^{norm}$ but as $A_K$ is normal, $x \in A_K$.
Now write, for $a_i \in A, 0 \leq i \leq n$ and $a_n =1$
\[ x^n + a_{n-1}x^{n-1} +\dots+ a_1x = -a_0.\]
Then $\sup_{1 \leq i \leq n} |x^ia_{n-i}|_A = |a_0|_A \leq 1$, thus $|x^n|_A = |x|_A^n \leq 1$, and $|x|_A \leq 1$. Thus $x \in A$.
\end{proof}

\end{proof}

By the previous sections, we have on $X(v)$ a filtration of $G$ by finite flat $\mathcal O$-modules,
\[ 0 \subset H_{\tau}^{2m} \subset H_{\sigma\tau}^{2m} \subset G[p^{2m}],\]
locally isomorphic to $\mathcal O/p^{2m}\mathcal O$ and $(\mathcal O/p^{2m}\mathcal O)^2$.
Moreover, the subgroup $K_m = H_{\tau}^{2m} + H_{\sigma\tau}^{2m}[p^m]$ extend to $\mathfrak X(v)$ by proposition \ref{propKernormal}, and over $X(v)$ is locally isomorphic to
\[ \mathcal O/p^{2m}\mathcal O \oplus \mathcal O/p^{m}\mathcal O.\]

\begin{definen}
We denote,
\[ X_1(p^{2m})(v) = \Isom_{X(v),pol}(K_m^D,\mathcal O/p^{m}\mathcal O\oplus \mathcal O/p^{2m}\mathcal O),\]
where the condition \textit{pol} means that we are looking at isomorphisms $\psi = (\psi_1,\psi_2)$ which induces an isomorphism "in first coordinate",
\[\psi_{1,1} = (\psi_1)_{|(H_{\sigma\tau}^{m}/H_{\tau}^{m})^D} : (H_{\sigma\tau}^{m}/H_{\tau}^{m})^D \simeq  \mathcal O/p^{m}\mathcal O,\]
such that $(\psi_{1,1})^D = ((\psi_{1,1})^{(\sigma)})^{-1}$, and such that the quotient morphism,
\[ \psi_{2,1} = \psi_1/(\psi_1)_{|(H_{\sigma\tau}^{m}/H_{\tau}^{m})^D} : (H_{\tau}^{2m})^D \fleche \mathcal O/p^{m}\mathcal O,\]
is zero.
\end{definen}

\begin{remaen}
The map $\psi_{1,1}$ is automatically an isomorphism. Moreover,
\[(\psi_{1,1})^D : \mathcal O/p^{m}\mathcal O \fleche H_{\sigma\tau}^{m}/H_{\tau}^{m} \overset{\lambda}{\simeq} (H_{\sigma\tau}^{m}/H_{\tau}^{m})^{D,(\sigma)},\]
where the last morphism is induced by $\lambda$, the polarisation of $A$.
\end{remaen}

Denote by $B_n$ the subgroup of $\GL(\mathcal O/p^{m}\mathcal O \oplus \mathcal O/p^{2m}\mathcal O))$ of matrices,
\[
\left(
\begin{array}{cc}
a  & p^mb   \\
0 & d   \\
\end{array}
\right)
\]
such that $a^{-1} = a^{(\sigma)}$ i.e. $a \in (\mathcal O/p^m\mathcal O)^1$. We can map $\mathcal O^\times \times \mathcal O^1$ (diagonally) to $B_n$.
\[ B_n \simeq 
\left(
\begin{array}{cc}
(\mathcal O/p^m\mathcal O)^1  &  \mathcal O/p^m\mathcal O    \\  
  &  (\mathcal O/p^{2m}\mathcal O)^\times
\end{array}
\right).
\]
Denote also, 
\[B_\infty(\ZZ_p) = \left(
\begin{array}{cc}
\mathcal O^1  &  \mathcal O \\  
 0 &  \mathcal O^\times
\end{array}
\right)
\]
which surjects to $B_n$ and that we can embed into $GL_2\times GL_1$ (even in its upper triangular borel) via,
\[
\left(
\begin{array}{cc}
a &  b \\  
 0 & d
\end{array}
\right)
\longmapsto \left(
\begin{array}{ccc}
\sigma\tau(a)  & \sigma\tau(b) &   \\
 &  \sigma\tau(d) &   \\
  &   &   \tau(d)
\end{array}
\right).\]
We denote by $\psi_\tau$ and $\psi_{\sigma\tau}$ the inverses of the induced isomorphisms,
\[ \psi_{\sigma\tau} : \mathcal O/p^m\mathcal O \simeq (H_{\sigma\tau}^m/H_{\tau}^m)^D,\]
and the quotient,
\[ \psi_{\tau} = \psi^{-1}/\psi_{\sigma\tau} : \mathcal O/p^{2n}\mathcal O \simeq (H_{\tau}^{2m})^D.\]
We also denote $\mathfrak X_1(p^{2m})(v)$ the normalisation of $\mathfrak X(v)$ in $X_1(p^{2m})(v)$.
Over $\mathfrak X_1(p^{2m})(v)$, we have by the previous section locally free subsheaves of $\mathcal O_{\mathfrak X_1(p^{2m})(v)}$-modules $\mathcal F_\tau,\mathcal F_{\sigma\tau}$ of $\omega_{G,\tau}$ and $\omega_{G,\sigma\tau}$ together with morphisms,
\[ \HT_{\tau,w}\circ \psi_{\tau}[p^m] : (\mathcal O/p^{m}\mathcal O)\otimes_{\mathcal O}\mathcal O_{\mathfrak X_{1}(p^{2m})(v)} \overset{\sim}{\fleche} \mathcal F_{\tau} \otimes_{\mathcal O_K}\mathcal O_K/p^w,\]
\[\HT_{\sigma\tau,w}\circ \psi_{\sigma\tau} : (\mathcal O/p^{m}\mathcal O)\otimes_{\mathcal O}\mathcal O_{\mathfrak X_{1}(p^{2m})(v)} \hookrightarrow \mathcal F_{\sigma\tau} \otimes_{\mathcal O_K}\mathcal O_K/p^w,\]
and denote by $\mathcal F_{\sigma\tau,w}^{can}$ the image of the second morphism. It is a locally direct factor of $\mathcal F_{\sigma\tau}\otimes\mathcal O_K/p^w$, and passing through the quotient, we get a map,
\[ \overline{\HT_{\sigma\tau,w}} \circ \psi_{\tau}[p^m] : (\mathcal O/p^m\mathcal O) \otimes_{\mathcal O}\mathcal O_{\mathfrak X_{1}(p^m)(v)} \overset{\sim}{\fleche} (\mathcal F_{\sigma\tau} \otimes_{\mathcal O_K}\mathcal O_K/p^w)/(\mathcal F_{\sigma\tau,w}^{can}),\]
Using the construction of torsors of the previous section, we get a chain of maps,
\[ \mathfrak{IW}^+_w \overset{\pi_1}{\fleche} \mathfrak{IW}_w \overset{\pi_2}{\fleche} \mathfrak X_{1}(p^{2m})(v) \overset{\pi_3}{\fleche} \mathfrak X(v).\]
Moreover, $\pi_1$ is a torsor over the formal torus $\mathfrak T_w$, $\pi_2$ is affine, and we have an action of $\mathcal O^\times \times \mathcal O^1$ and $B_n$ on $\mathfrak X_1(p^{2m})$ over $\mathfrak X(v)$. Denote by $B$ the Borel of $\GL_2\times\GL_1$, $\mathfrak B$ its formal completion along its special fiber, and $\mathfrak B_w$,
\[ \mathfrak B_w(A) = \Ker(\mathfrak B(A) \fleche \mathfrak B(A/p^wA)).\]
We can embed $\mathfrak T$ in $\mathfrak B$ (which induce an embedding $\mathfrak T_w \subset \mathfrak B_w$) and $\mathcal O^\times \times \mathcal O^1$ in $\mathfrak T$, via  
\[(a,b) \in \mathcal O^\times \times \mathcal O^1 \longmapsto 
\left(
\begin{array}{ccc}
 \sigma\tau(b) &   &   \\
  & \sigma\tau(a)  &   \\
  &   &   \tau(a)
\end{array}
\right) \in\mathfrak T.
\]
such that the action of $\mathcal O^\times \times \mathcal O^1$ on $X_1(p^{2m})(v)$ and via $\mathfrak T$ on $\mathcal Gr^+$ preserves $\mathfrak{IW}_w^+$ (over $\mathfrak X(v)$). More generally, the action of $B_\infty(\ZZ_p)$ on $X_1(p^{2m})(v)$ (and thus $\mathfrak X_1(p^{2m})(v)$) and via $\mathfrak B$ on $\mathcal Gr^+$ preserves 
$\mathfrak{IW}_w^+$.

Let $\kappa \in \mathcal W_w(L)$. The character $\kappa : \mathcal O^\times \times \mathcal O^1 \fleche \mathcal O_L^\times$ extends to a caracter,
\[ \kappa : (\mathcal O^\times \times \mathcal O^1)\mathfrak T_w \fleche  \widehat{\mathbb G_m},\]
which can be extended as a character of,
\[ \kappa : (\mathcal O^\times\times \mathcal O^1)\mathfrak B_w \fleche  \widehat{\mathbb G_m},\]
where $\mathfrak U_w \subset \mathfrak B_w$ acts trivially, and even as a character,
\[ \kappa : B(\ZZ_p)\mathfrak B_w \fleche \widehat{\mathbb G_m},\]
where $U(\ZZ_p)\mathfrak U_w$ acts trivially.
Let us denote $\pi = \pi_3\circ\pi_2\circ \pi_1$.

\begin{propen}
The sheaf $\pi_*\mathcal O_{\mathfrak{IW}_w^+}[\kappa]$ is a formal Banach sheaf, in the sense of \cite{AIP} definition A.1.2.1.
\end{propen}

\begin{proof}[Proof]
We can use the same devissage as presented in \cite{AIP} : denote $\kappa^{0}$ the restriction of $\kappa$ to $\mathfrak T_w$. Then $(\pi_1)_*\mathcal O_{\mathfrak{IW}_w^+}[\kappa^{0}]$ is an invertible sheaf on $\mathfrak{IW}_w$. Its pushforward via $\pi_2$ is then a formal Banach sheaf because $\pi_2$ is affine, and pushing through $\pi_3$ and taking invariants over $B_\infty(\ZZ_p)/p^n = B_n, \pi_*\mathcal O_{\mathfrak{IW}_w^+}[\kappa]$ is a formal Banach sheaf.
\end{proof}

\begin{definen}
We call $\mathfrak w_w^{\kappa\dag} := \pi_*\mathcal O_{\mathfrak{IW}_w^+}[\kappa]$ the sheaf of $v$-overconvergent $w$-analytic modular forms of weight $\kappa$.
The space of integral $v$-overconvergent, $w$-analytic modular forms of weight $\kappa$ and level (outside $p$) $K^p$, for the group $G$ is,
\[ M_w^{\kappa\dag}(\mathfrak X(v)) = H^0(\mathfrak X(v),\mathfrak w_w^{\kappa\dag}).\]
\end{definen}

\begin{remaen}
Unfortunately it doesn't seem clear how to define an involution $\kappa \mapsto \kappa'$ on all $\mathcal W$ which extends the one on classical weights, and thus we only get that classical modular forms of (classical, integral) weight $\kappa$ embed in overconvergent forms of weight $\kappa'$...
\end{remaen}

\subsection{Changing the analytic radius}

Let $m - \frac{p^{2m}-1}{p^2-1}v > w' > w$ and $\kappa \in \mathfrak W_w(L)$, and thus $\kappa \in \mathfrak W_{w'}(L)$. There is a natural inclusion,
\[ \mathfrak{IW}_{w'}^+ \hookrightarrow \mathfrak{IW}_w^+,\]
compatible with the action of $(\mathcal O^\times \times \mathcal O^1)\mathfrak B_w$. This induces a map $\mathfrak w_w^{\kappa\dag} \fleche \mathfrak w_{w'}^{\kappa\dag}$ and thus a map,
\[ M^{\dag\kappa}_{w}(\mathfrak X(v)) \fleche M^{\dag\kappa}_{w'}(\mathfrak X(v)).\]
 
\begin{definen}
The space of integral overconvergent locally analytic Picard modular forms of weight $\kappa$, and level (outside $p$) $K^p$, is,
\[ M_\kappa^\dag(\mathfrak X) = \varinjlim_{v \rightarrow 0, w\rightarrow \infty} M_w^{\kappa\dag}(\mathfrak X(v)).\]
\end{definen}

\subsection{Classical and Overconvergent forms in rigid fiber.}

Denote  $\mathfrak X_{Iw^+(p^{2m})}(v)$ the quotient of $\mathfrak X_1(p^{2m})(v)$ by $\tilde{U_m} \subset B_m$, which is isomorphic to,
\[
\left(
\begin{array}{cc}
1   &  \mathcal O/p^m\mathcal O  \\
  &   1 + p^m\mathcal O/p^{2m}\mathcal O
\end{array}
\right)\subset \left(
\begin{array}{cc}
(\mathcal O/p^m\mathcal O)^\times   &  \mathcal O/p^m\mathcal O  \\
  &   (\mathcal O/p^{2m}\mathcal O)^\times
\end{array}
\right).
\]
Let also denote $X_{Iw^+(p^{2m})}(v)$ the corresponding rigid space. Over the scheme $X$, we have the locally free sheaf $\omega_A = \omega_{A,\tau} \oplus \omega_{A,\sigma\tau}$, which is locally isomorphic to $\mathcal O_X \oplus \mathcal O_X^2$, with the corresponding action of $\mathcal O$.
Denote by $\mathcal T$ the scheme $Hom_{X,\mathcal O}(\mathcal O_X^2\oplus\mathcal O_X,\omega_G)$ of trivialisation of $\omega_G$ as a $\mathcal O_X \otimes_{\ZZ_p}\mathcal O = \mathcal O_X \oplus \mathcal O_X$-sheaf, denote $\mathcal T^\times$ its subsheaf of isomorphisms, it is a $\GL_2\times\GL_1$-torsor over $X$, where $g \in \GL_2\times \GL_1$ acts on $\mathcal T^\times$ by $g \cdot \phi = \phi \circ g^{-1}$. For $\kappa \in X^+(T)$ a classical weight, denote by $\omega^\kappa$ the sheaf $\pi_*\mathcal O_{\mathcal T^\times}[\kappa']$, where $\pi : \mathcal T^\times \fleche X$ is the projection and $\kappa \fleche \kappa'$ the involution on classical weights. In down-to-earth terms, $\kappa = (k_1,k_2,l)$ where $k_1 \geq k_2$ and
\[ \omega^\kappa = \Sym^{k_1-k_2}\omega_{G,\sigma\tau} \otimes (\det \omega_{G,\sigma\tau})^{\otimes k_2}\otimes (\det \omega_{G,\tau})^{\otimes l}.\]

We have defined $X(v)$, which is the rigid fiber of $\mathfrak X(v)$. Denote by $\mathcal T_{an},\mathcal T_{an}^\times, (\GL_2\times\GL_1)_{an}$ the analytification of the schemes $\mathcal T,\mathcal T^\times, \GL_g$, and $\mathcal T_{rig},\mathcal T_{rig}^\times, (\GL_2\times\GL_1)_{rig}$ Raynaud's rigid fiber of the completion along the special fibers of the same schemes.
As $\mathcal T^\times/B$ is complete, $\mathcal T_{an}^\times/B_{an} = \mathcal T_{rig}^\times/B_{rig}$, over which there is the diagram,
\begin{center}
\begin{tikzpicture}[description/.style={fill=white,inner sep=2pt}] 
\matrix (m) [matrix of math nodes, row sep=3em, column sep=2.5em, text height=1.5ex, text depth=0.25ex] at (0,0)
{ 
\mathcal T_{rig}^\times/U_{rig} & & \mathcal T_{an}^\times/U_{an} \\
 & \mathcal T_{rig}^\times/B_{rig} & \\
 };
\path[->,font=\scriptsize] 
(m-1-1) edge node[auto] {$$} (m-1-3)
(m-1-1) edge node[auto,left] {$f$} (m-2-2)
(m-1-3) edge node[auto] {$g$} (m-2-2)
;
\end{tikzpicture}
\end{center}
where $f$ is a torsor over $U_{rig}/B_{rig} = T_{rig}$ (the torus, not to be mistaken with $\mathcal T_{rig}$) and $g$ a torsor over $T_{an}$ (same remark).

\begin{definen}
Let $\kappa \in \mathcal W_w(K)$. We denote by $\omega_w^{\kappa\dag}$ the rigid fiber of $\mathfrak w_w^{\kappa\dag}$ on $X(v)$. It exists by \cite{AIP} Proposition A.2.2.4. It is called the sheaf of $w$-analytic overconvergent modular forms of weight $\kappa$. The space of $v$-overconvergent, $w$-analytic modular forms of weight $\kappa$ is the space,
\[ H^0(X(v),\omega_w^{\kappa\dag}).\]
The space of locally analytic overconvergent Picard modular forms of weight $\kappa$ (and level $K^p$) is the space,
\[ M_\kappa^\dag(X) = \varinjlim_{v\rightarrow 0,w\rightarrow \infty} H^0(X(v),\omega_w^{\kappa\dag}).\]
\end{definen}

The injection of $\mathcal O_{\mathfrak X(v)}$-modules $\mathcal F_\tau \oplus \mathcal F_{\sigma\tau} \subset \omega_{A} = \omega_{A,\tau}\oplus \omega_{A,\sigma\tau}$ is an isomorphism in generic fiber, and this induces an open immersion,
\[ \mathcal{IW}_w \hookrightarrow \mathcal T_{rig}^\times/B_{rig} \times_{X(v)} X_1(p^{2m})(v).\]
We also have an open immersion,
\[ \mathcal{IW}_w^+ \hookrightarrow \mathcal T_{an}^\times/U_{an} \times_{X(v)} X_1(p^{2m})(v).\]

The action of $B_n$ on $\mathfrak X_1(p^{2m})(v)$ (or $X_1(p^{2m})(v)$) lift to an action on $\mathfrak{IW}_w$ (or $\mathcal{IW}_w$) because being $w$-compatible for $\Fil^1\mathcal F_\tau$ only depend on the trivialisation of $K_n^D$ modulo $B_n$. Similarly the action of $\tilde{U_n}$ lifts to $\mathfrak{IW}_w^+$ and $\mathcal{IW}_w^+$.
We can thus define $\mathcal{IW}_w^0$ and $\mathcal{IW}_w^{+,0}$ the respective quotients of $\mathcal{IW}_w$ and $\mathcal{IW}_w^+$ by $B_n$ and $\tilde{U_n}$, which induces open immersions,
\[ \mathcal{IW}_w^0\hookrightarrow \mathcal T_{rig}^\times/B_{rig} \times_{X(v)} X(v) \quad \text{and} \quad  \mathcal{IW}_w^{+,0}\hookrightarrow \mathcal T_{an}^\times/U_{an} \times_{X(v)} X_{Iw^+(p^{2m})}(v).\]

\begin{propen}
Suppose $w > m-1 + {\frac{p+v}{p^2-1}}$. Then there are embeddings 
\[ \mathcal{IW}_w^{0} \subset (\mathcal T^{an}/B)_{X(v)}\quad \text{and} \quad h :\mathcal{IW}_w^{0,+} \subset (\mathcal T^{an}/U)_{X(v)}.\]
\end{propen}

\begin{proof}[Proof]
Let $S$ be a set of representatives in $I_\infty \simeq 
\left(
\begin{array}{cc}
\mathcal O^1  & \mathcal O   \\
 p\mathcal O & \mathcal O^\times     
\end{array}
\right)
$ of $I_n/\widetilde{U_n}$ which we can suppose of the form,
\[ 
\left(
\begin{array}{cc}
 [b] &     \\
p[c] & [a]     
\end{array}
\right), \quad a \in (\mathcal O/p^m)^\times, b \in (\mathcal O^1/p^m), c \in \mathcal O/p^{m-1}.\]
Here, $[.]$ denote any lift.
Then $h$ is locally (over $X(v)$) isomorphic to,
\[ h : \coprod_{\gamma \in S} 
\left(
\begin{array}{ccc}
1+p^wB(0,1)  &   &   \\
p^wB(0,1)  &1+p^wB(0,1)    &   \\
  &   &   1+p^wB(0,1) 
\end{array}
\right) M \widetilde{\gamma} \fleche (\GL_2\times\GL_1/U)_{an}
\]
where $M$ is the matrix which is locally given by the Hodge-Tate map, and thus correspond to the inclusion $\mathcal F_\tau \oplus \mathcal F_{\sigma\tau} \subset \omega_\tau \oplus \omega_{\sigma\tau}$, and if $\gamma \in S$, then $\widetilde{\gamma}$ is given by,
\[ \widetilde \gamma = 
\left(
\begin{array}{ccc}
\sigma\tau(b)  &   &   \\
p \sigma\tau(c) &  \sigma\tau(a) &   \\
  &   &   \tau(a)
\end{array}
\right), \quad \text{if} \quad \gamma =\left(
\begin{array}{cc}
 b &     \\
pc & a     
\end{array}
\right)
\]
But there exists $M'$ with integral coefficients such that $M'M = 
\left(
\begin{array}{cc}
 p^{\frac{p+v}{p^2-1}}I_2 &      \\
  &  p^{\frac{v}{p^2-1}}  
\end{array}
\right)$, and it is easily checked that $M' \circ h$ is then injective if $w > m-1 +\frac{p+v}{p^2-1}$. The proof for the other embedding is similar (and easier).\end{proof}

\begin{remaen}
Of course not every $w$ satisfies $m-1+\frac{p+v}{p^2-1} < w < m-\frac{p^{2m}-1}{p^2-1}v$ for some $m$. But every $w$-analytic character $\kappa$ is $w'$-analytic for all $w' > w$, and in particular we can choose a $w'$ which satisfies the previous inequalities, a priori for another $m$ (for each choice of $w'$ there is a unique such $m$).
\end{remaen}

From now on, we suppose that we have fixed $m$, and we suppose that $m-1+\frac{p+v}{p^2-1} < w < m-\frac{p^{2m}-1}{p^2-1}v$.

We could have defined $\omega_w^{\kappa\dag}$ directly, by $g_*\mathcal O_{\mathcal{IW}_w^{0,+}}[\kappa]$ where $g$ is the composite,
\[ \mathcal{IW}_w^{0,+} \fleche \mathcal{IW}_w^0 \fleche X(v) \]

as shown by the next proposition. Remark that $X(v) \subset X_{Iw(p)}(v)$ via the canonical filtration of level 1.

\begin{propen}
The sheaf $\omega_w^{\kappa\dag}$ (defined as the rigid fiber of $\mathfrak w_w^{\kappa\dag}$) is isomorphic to $g_*\mathcal O_{\mathcal{IW}_w^{0,+}}[\kappa]$.
\end{propen}

\begin{proof}[Proof]
In the rigid setting, we did a quotient by $\widetilde{U_n}$ to get $\mathcal{IW}^{0,+}_w$. But $\omega_w^{\kappa\dag}$ is constructed as 
$((\pi_2\circ\pi_1)_*\mathcal O_{\mathcal{IW}^+}[\kappa^{0}])(-\kappa)^{B_n}$, and the action of $\widetilde{U_n}$ on $(\pi_2\circ\pi_1)_*\mathcal O_{\mathcal{IW}^+}[\kappa^{0}]$ is trivial and it thus descends to $X_{Iw^+(p^{2m})}(v)$ and is isomorphic to the $\kappa^0$-variant vectors in the pushforward of $\mathcal O_{\mathcal{IW}_w^{0,+}}$.

\end{proof}

\begin{propen}
For $\kappa \in X_+(T^0)$ and $\omega >0$, there is a restriction map,
\[ \omega_{X(v)}^\kappa \hookrightarrow \omega_w^{\kappa'\dag},\]
induced by the inclusion $\mathcal{IW}_w^{0,+} \subset (\mathcal T^{an}/U)$. Moreover, locally for the etale topology, this inclusion is isomorphic to the following composition,
\[ V_{\kappa'} \hookrightarrow V_{\kappa'}^{w-an} \overset{res_0}{\fleche} V_{0,\kappa'}^{w-an}.\]
\end{propen}

\begin{proof}[Proof]
Locally for the etale topology, $\omega^{\kappa}$ is identified with algebraic function on $\GL_2\times\GL_1$ which are invariant by $U$ and varies as $\kappa'$ under the action of $T$, i.e. to $V_{\kappa'}$. But a function $f \in \omega_w^{\kappa'\dag}$ is locally identified with a function, 
\[ f : 
\{ \left(
\begin{array}{ccc}
\tau(a)(1+p^wB(0,1) )  &  &   \\
p^w B(0,1)  &  \tau(b)(1+p^wB(0,1))  &   \\
  &   &   \sigma\tau(b)(1+p^wB(0,1) )
\end{array}
\right) ,  a \in \mathcal O^1, b \in \mathcal O^\times \}\fleche L,
\]
which varies as $\kappa'$ under the action on the right of $T(\ZZ_p)\mathfrak T_w$. As $\kappa' = (k_1,k_2,k_3) \in X_+(T)$ we can extend $f$ to a $\kappa'$-varying function on
 \[I_{p^w} = \{\left(
\begin{array}{ccc}
\mathbb G_m  & B(0,1) &   \\
p^w B(0,1)  &  \mathbb G_m  &   \\
  &   &   \mathbb G_m 
\end{array}
\right)\},\]
extending it "trivially" ; i.e.
\begin{eqnarray*} f \left(\left(
\begin{array}{ccc}
x  & u &   \\
p^w z  &  y  &   \\
  &   &   t 
\end{array}
\right)\right) &= f\left(\left(
\begin{array}{ccc}
1  & 0 &   \\
p^w zx^{-1}  &1  &   \\
  &   &   1
  \end{array}
\right)\left(
\begin{array}{ccc}
x & u &   \\
0  &  y - p^w z u  &   \\
  &   &   t 
\end{array}
\right)\right) \\ & = x^{k_1}( y - p^w z u )^{k_2}t^{k_3}f\left(\left(
\begin{array}{ccc}
1  & 0 &   \\
p^w zx^{-1}  &1  &   \\
  &   &   1
  \end{array}
\right)\right).\end{eqnarray*}
Under this identification, locally for the etale topology $\omega_w^{\kappa'\dag}$ is identified with $V_{0,\kappa'}^{w-an}$.
\end{proof}

\section{Hecke Operators, Classicity}
\label{sectHecke} 
 As explained in \cite{AIP} and \cite{Bra}, it is not possible to find a toroidal compactification for more general PEL Shimura varieties (already for $\GSp_4$) that is preserved with all the Hecke correspondances, but this can be overcome by looking at \textit{bounded} sections on the open variety. For the Picard modular variety, there is only one choice of a toroïdal compactification, and thus this problem doesn't appear, but we will keep the general strategy (and thus we won't have to check that the correspondances extend to the boundary).
 Thus, instead we will define Hecke operators on the open Picard Variety $\mathcal Y_{Iw(p)}$ of Iwahori level, and as bounded sections on the open variety extend automatically to the compatification (see \cite{AIP} theorem 5.5.1, proposition 5.5.2, which follows from a theorem of Lutkebohmert), we show that Hecke operators send bounded functions to bounded functions, and thus induces operators on overconvergent locally analytic modular forms.

\subsection{Hecke operators outside $p$}

These operators have been defined already in \cite{Bra}, section 4. We explain their definition quickly, and refer to \cite{Bra} (see also \cite{AIP} section 6.1) for the details.
Let $\ell \neq p$ be an integer, and suppose $\ell\notdivides N$, the set of places where $K_v$ is not maximal.
Let $\gamma \in G(\QQ_\ell) \cap \End{\mathcal O_{E,\ell}}(\mathcal O_{E,\ell}^3)\times\QQ_\ell^\times$, and consider,
\[ C_\gamma  \rightrightarrows \mathcal Y_{Iw(p)},\]
the moduli space of isogeny $f : A_1 \fleche A_2$ such that,
\begin{enumerate}
\item $f$ is $\mathcal O_E$-linear, and of degree a power of $\ell$.
\item $f$ is compatible with polarisations, i.e. $f^*\lambda_2$ is a multiple of $\lambda_1$.
\item $f$ is compatible with the $K^p$-level structure (at places that divides $N$) (remark that $f$ is an isomorphism on $T_q(A_i)$ when $q \neq \ell$ is a prime).
\item $f$ is compatible with the filtration given by the Iwahori structure at $p$.
\item The type of $f$ is given by the double class $G(\ZZ_\ell)\gamma G(\ZZ_\ell)$.
\end{enumerate}

\begin{remaen}
The space $C_\gamma$ doesn't depends on $\gamma$, only on the double class $G(\ZZ_\ell) \gamma G(\ZZ_\ell)$.

We could similarly define $C_\gamma$ without Iwahori level at $p$ (i.e. for $\mathcal Y_{G(\ZZ_p)K^p}$) without the condition that $f$ is compatible with the filtration given by the Iwahori structure at $p$. In our case, this Iwahori structure at $p$ will always be the canonical one, and thus $f$ is automatically compatible as it sends the canonical filtration of $A_1$ in the one of $A_2$. 
\end{remaen}

Denote by $p_i : C_\gamma \fleche \mathcal Y$ the two (finite) maps that sends $f$ to $A_i$. Denote by $C_\gamma(p^n)$ the fiber product with $p_1$ of $C_\gamma$ 
with $\mathcal Y_1(p^n)(v) \fleche \mathcal Y(v) \subset_{s} \mathcal Y_{Iw(p)}(v)$,
where $s$ is the canonical filtration of $A[p]$. Denote by $f$ the universal isogeny over $G_\gamma(p^n)$. It induces an $\mathcal O$-linear isomorphism,
\[ p_2^*(\mathcal F_\tau \oplus \mathcal F_{\sigma\tau}) \fleche p_1^*(\mathcal F_\tau \oplus \mathcal F_{\sigma\tau}).\]
In particular, we get a $B(\ZZ_p)\mathfrak B_w$-equivariant isomorphism,
\[ f^* : p_2^*\widetilde{\mathcal{IW}}^+_{w,|\mathcal Y_1(p^n)(v)} \overset{\sim}{\fleche} p_1^*\widetilde{\mathcal{IW}}^+_{w,|\mathcal Y_1(p^n)(v)}.\]
We can thus form the composite morphism,
\begin{eqnarray*}
H^0(\mathcal Y_1(p^n)(v),\mathcal O_{\widetilde{\mathcal{IW}}^+_w}) \overset{p_2^*}{\fleche} H^0(C_\gamma(p^n)(v),p^*_2\mathcal O_{\widetilde{\mathcal{IW}}^+_w}) \overset{(f^*)^{-1}}{\fleche} H^0(C_\gamma(p^n)(v),p^*_1\mathcal O_{\widetilde{\mathcal{IW}}^+_w}) \\ \overset{Tr(p_1)}{\fleche} H^0(\mathcal Y_1(p^n)(v),\mathcal O_{\widetilde{\mathcal{IW}}^+_w}).\end{eqnarray*}
 As $f^*$ is an isomorphism, it sends bounded functions to bounded functions, and we can make the following definition,
 
\begin{definen}
Let $\kappa \in \mathcal W_w(K)$ a weight. We define the Hecke operator,
\[ T_\gamma : M_{v,w}^{\kappa\dag} \fleche M_{v,w}^{\kappa\dag},\]
as the restriction of the previous operator to the bounded, $\kappa$-equivariant sections under the action of $B(\ZZ_p)\mathfrak B_w$. It induces an operator,
\[ T_\gamma : M^{\kappa\dag} \fleche M^{\kappa\dag}.\]
\end{definen}

\begin{definen}
Define $\mathcal H$ to be the commutative $\ZZ$ -algebra generated by all operators $T_\gamma$ for all $\ell \notdivides Np$ and all double classes $\gamma$. These operators commute on overconvergent forms, and thus $\mathcal H$ acts on them.
\end{definen}

\subsection{Hecke operator at $p$}

Here we stress that as in the previous sections, the prime $p$ is inert in $E$. We will  define a first Hecke operator at $p$, $U_p$.
Define $C$ the moduli space over $K$ which parametrizes data $(A,\lambda,i,\eta,L)$ where $(A,\lambda,i,\eta) \in X_K(v)$ and $L\subset A[p^2]$ is a totally isotropic $\mathcal O$-module for $\lambda$ (i.e. $L \subset L^\perp := (G[p^2]/L)^D \subset G[p^2]^D \overset{\lambda}{\simeq} G[p^2]$) of $\mathcal O$-height $p^3$ (and thus $L = L^\perp$) such that
\[ %L \cap H_{\sigma\tau}^2 = \{0\} (\text{or } 
L[p] \oplus H_{\tau}^1 = A[p] \quad \text{and}\quad pL \oplus H_{\sigma\tau}^1 = A[p].\]
As remarked by Bijakowski in \cite{Bijmu}, the second condition is implied by the first one and the isotropy condition.
We then define two projections,
\[ p_1,p_2 : C \fleche X_K,\]
where $p_1$ is the forgetful map which sends $(A,\lambda,i,\eta,L)$ to $(A,\lambda,i,\eta,L)$ and $p_2$ sends $(A,\lambda,i,\eta,L)$ to $(A/L,\lambda',i',\eta')$. To compare the correspondance with the canonical filtration we will need the following lemma.

\begin{lemmen}
Let $p > 2$ and $G$ be a $p$-divisible $\mathcal O$-module of unitary type and signature $(2,1)$. Let $H$ be a sub-$\mathcal O$-module of $p$-torsion and of $\mathcal O$-height 1.
Then the two following assertions are equivalent,
\begin{enumerate}
\item $\Deg_{\tau}(H) > 1+p - \frac{1}{2}$,
\item $\ha_\tau(G) < \frac{1}{2}$ and $H$ is the canonical subgroup of $G[p]$ associated to $\tau$.
\end{enumerate}
Let $H$ be a sub-$\mathcal O$-module of $p$-torsion and of $\mathcal O$-height 2.
Then the two following assertions are equivalent,
\begin{enumerate}
\item $\Deg_{\sigma\tau}(H) > p + 2 - \frac{1}{2}$,
\item $\ha_\tau(G) < \frac{1}{2}$ and $H$ is the canonical subgroup of $G[p]$ associated to $\sigma\tau$.
\end{enumerate}
In both cases we can be more precise : if $v = 1 + p - \Deg_{\tau}(H)$ (respectively $2 + p - \Deg_{\sigma\tau}(H)$) then $\ha_\tau(G) \leq v$.
\end{lemmen}

\begin{proof}[Proof]
In both cases we only need to prove that the first assumption implies the second, by the existence theorem of the canonical filtration $(\ha_\tau(G) = \ha_{\sigma\tau}(G))$.
Moreover, we only have to prove that $\ha_\tau(G) < \frac{1}{2}$, because then $G[p]$ will have a canonical filtration, and in both cases, $H$ will be a group of this filtration because it will correspond to a break-point of the Harder-Narasihman filtration (for the classical degree, as we only care about filtration in $G[p]$). Let us do the second case, as it is the most difficult one (the first case can be treated similarly, even using only technics introduced in \cite{Far2}). Let $v = 2 + p - \Deg_{\sigma\tau}(H)$. We can check that $\deg_{\sigma\tau}(H) > 2-v$, thus $\deg_{\sigma\tau}(H^D) < v$, and thus for all 
$\eps > 1-v$, if $E = G[p]/H$,
\[ \omega_{G^D,\sigma\tau,\eps} \simeq \omega_{E^D,\sigma\tau,\eps}.\]
But then the cokernel of $\alpha_{E,\sigma\tau,\eps}\otimes 1$ is of degree $\frac{1}{p^2-1}\Deg_{\sigma\tau}(E)$ ($E$ is a Raynaud subgroup of type (p...p)), and the following square is commutative,
\begin{center}
\begin{tikzpicture}[description/.style={fill=white,inner sep=2pt}] 
\matrix (m) [matrix of math nodes, row sep=3em, column sep=2.5em, text height=1.5ex, text depth=0.25ex] at (0,0)
{ 
G[p](\mathcal O_K) & & E(\mathcal O_K) \\
\omega_{G[p]^D,\sigma\tau,\eps} &  &\omega_{E^D,\sigma\tau,\eps} \\
 };

%\draw[double,double distance=5pt] (m-1-1) – (m-1-3);
\path[->,font=\scriptsize] 
(m-1-1) edge node[auto] {$$} (m-1-3)
(m-1-3) edge node[auto] {$\alpha_{E,\sigma\tau,\eps}$} (m-2-3)
(m-1-1) edge node[auto] {$\alpha_{G,\sigma\tau,\eps}$} (m-2-1)
(m-2-1) edge node[auto] {$\sim$} (m-2-3);
\end{tikzpicture}
\end{center}
thus in particular $\deg \Coker(\alpha_{E,\sigma\tau,\eps}\otimes1) = \deg \Coker(\alpha_{G[p],\sigma\tau,\eps}\otimes1)$. But according to proposition 5.25 of \cite{Her2}, we can check that the image of $\alpha_{G[p],\sigma\tau}$ is always included inside $up^{\frac{\ha_{\sigma\tau}(G)}{p^2-1}}\mathbb F_{p^2} + p^{\frac{1}{p^2}}\mathcal O_C/p \subset \omega_{G[p]^D,\sigma\tau} \simeq \mathcal O_C/p$ for some $u \in \mathcal O_C^\times$. Rewriting the inequality with $\Deg_{\sigma\tau}(E) = \Deg_{\sigma\tau}(G[p]) - \Deg_{\sigma\tau}(H)$ we get,
\[\min(\ha_\tau(G),\frac{p^2-1}{p^2}) \leq 2 + p - \Deg_{\sigma\tau}(H) = v,\] but as $v < \frac{1}{2} <1 - \frac{1}{p^2}$, we get $\ha_\tau(G) \leq v$.
\end{proof}

Thus, we can deduce the following,

\begin{lemmen}
Let $(A,\lambda,i,\eta,L)$ as before with corresponding $(A,\lambda,i,\eta) \in \mathfrak X(v)$ and $v < \frac{1}{2(p^2+1)}$. Then $A/L \in \mathfrak X(v)$, and $A[p^2]/L$ coincides with 
the group $K_1(A/L)$.  
\end{lemmen}

\begin{proof}[Proof]
By hypothesis on $L$, the map
\[ H_{\tau}^1 \fleche A[p]/L[p],\]
is an isomorphism on generic fiber, thus $\Deg_{\tau}(A[p]/L[p]) \geq \Deg_{\tau}(H^1_{\tau}) > 1 + p - v$. Thus by the previous lemma, we get that,
$\ha_\tau(A/L) \leq v$ and moreover $A[p]/L[p]$ coincide with the first canonical subgroup associated to $\tau$. Moreover, we deduce that $\deg A[p]/L[p] \geq 2 - v$.
Now consider the composite map,
\[ H^2_{\sigma\tau} \fleche A[p^2]/L \fleche (A[p^2]/L)/(A[p]/L[p]) = Q.\]
Because $H^1_{\sigma\tau}$ is sent inside $A[p]/L[p]$, we get the factorisation,
\[ H^2_{\sigma\tau}/H^1_{\sigma\tau} \fleche Q.\]
This is a generic isomorphism by the second hypothesis on $L$, and thus $\Deg_{\sigma\tau}(Q) \geq \Deg_{\sigma\tau}(H_{\sigma\tau}^2/H_{\sigma\tau}^1)$. But by construction, $H_{\sigma\tau}^2/H_{\sigma\tau}^1$ is the canonical subgroup (for $\sigma\tau$) of $A/H_{\sigma\tau}^1$ and thus $\Deg_{\sigma\tau}(H_{\sigma\tau}^2/H_{\sigma\tau}^1) \geq p+2 - \ha_{\sigma\tau}(A/H_{\sigma\tau}^1)$, and $\ha_{\sigma\tau}(A/H_{\sigma\tau}^1) \leq p^2\ha_{\sigma\tau}(A)$ (this is \cite{Her2} proposition 8.1), and this implies that $\deg Q > 3 - p^2v$. Using the exact sequence,
\[ 0 \fleche A[p]/L[p] \fleche A[p^2]/L \fleche Q \fleche 0,\]
we get that $\deg A[p^2]/L > 5 - p^2v - v$. A similar argument also shows that $\deg K_1(A/L) \geq 5 - (p^2+1)\ha_\tau(A/L)$. But using Bijakowski's proposition recalled in \cite{Her2} Proposition A.2, we get that if $2(p^2+1)v \leq 1$, then $A[p^2]/L = K_1(A/L)$.
\end{proof}

\begin{lemmen}
Suppose $v < \frac{1}{2p^4}$. Let $G/\Spec(\mathcal O_C)$ be a $p$-divisible group such that $\ha_\tau(G) < v$, then $K_1 \subset G[p^2]$ coincides over $\mathcal O_C/p^{\frac{1}{2p^2}}$ with $\Ker F^2 \subset G[p^2]$. In particular, \[\ha_\tau(G/K_1) = p^2\ha_\tau(G).\]
\end{lemmen}

\begin{proof}[Proof]
This is Appendix \ref{AppA}.
\end{proof}

\begin{propen}
Let $v < \frac{1}{2p^4}$. The Hecke correspondance $U_p$ defined by the two previous maps preserves $X(v)$. More precisely, if $y \in p_2(p_1^{-1}(\{x\}))$ where $x \in X(v)$, then $y \in X(v/p^2)$.
\end{propen}

\begin{proof}[Proof]
This is the two previous lemmas as $(A/L)/(A[p^2]/L) = A$.
\end{proof}

Denote the universal isogeny over $C$ by,
\[ \pi : \mathcal A \fleche \mathcal A/L,\]
which induces maps $\omega_{A/L,\tau} \overset{\pi^\star_\tau}{\fleche} \omega_{A,\tau}$ and $\omega_{A/L,\sigma\tau} \overset{\pi^\star_{\sigma\tau}}{\fleche} \omega_{A,\sigma\tau}.$
We define \[\widetilde{\pi^\star} : p_2^*\mathcal T^\times_{an} \fleche p_1^*\mathcal T^\times_{an}\] with $\widetilde{\pi^\star}  = \widetilde{\pi^\star}_\tau \oplus\widetilde{\pi^\star}_{\sigma\tau}$ by $\widetilde{\pi^\star}_{\tau}={\pi^\star}_{\tau}$ and $\widetilde{\pi^\star}_{\sigma\tau}$ sends a basis $(e_1,e_2)$ of $\omega_{A/L,\sigma\tau}$ to $(\frac{1}{p}\pi^\star e_1,\pi^\star e_2)$. This is an isomorphism.

We will need to slightly change the objects as in \cite{AIP}, Proposition 6.2.2.2.
\begin{defin}
Denote by $w_0 = m - \frac{p^{2m}-1}{p^2-1}v$ and for $\underline w = (w_{1,1},w_{2,1},w_{2,2},w_\sigma)$, define $\mathcal{IW}^{0,+}_{\underline w}$ as  the subspace of $\mathcal T^\times/U_{an}$ (over $X_1(p^n)(v)$) of points for a finite extention $L$ of $K$ consisting of $(A,\psi_N,\Fil_{\sigma\tau},P_1^{\sigma\tau},P_2^{\sigma\tau},P^{\tau})$ such that there exists a polarised trivialisation $\psi$ of $K_m^D$ satisfying,
\begin{enumerate}
\item $\Fil_{\sigma\tau}$ is $(w_0,\psi)$-compatible with $H_{\tau}^m$,
\item $P_1^{\sigma\tau} = a_{1,1}\HT_{\sigma\tau,w_0}(\psi(e_1)) + a_{2,1}\HT_{\sigma\tau,w_0}(\psi_{e_2}) \pmod{p^{w_0}\mathcal F_{\sigma\tau}}$,
\item $P_2^{\sigma\tau} = a_{2,2}\HT_{\sigma\tau,w_0}(\psi(e_2)) \pmod{p^{w_0}\mathcal F_{\sigma\tau} + \Fil_{\sigma\tau}}$,
\item $P_\tau = t\HT_{\tau,w_0}(\psi(e_2)) \pmod{p^{w_0}\mathcal F_{\tau}}$,
\end{enumerate}
where $a_{1,1} \in B(1,p^{w_{1,1}}), a_{2,2} \in B(1,p^{w_{2,2}}), t \in B(1,p^{w_{\sigma}}),a_{2,1} \in B(0,p^{w_{2,1}})$.
\end{defin}

Let $\underline w$ as before, with $w_{2,1},w_{2,2} < m-1 - \frac{p^{2m}-1}{p-1}v$. Denote $\underline w' = (w_{1,1},w_{2,1}+1,w_{2,2},w_\sigma)$.

\begin{propen}
The quotient map,
\[\widetilde{\pi^\star}^{-1} : p_1^*\mathcal T^\times_{an}/U_{an} \fleche p_2^*\mathcal T^\times_{an}/U_{an}\]
sends $p_1^*\mathcal{IW}_{\underline w}^{0,+}$ to $p_2^*\mathcal{IW}_{\underline w'}^{0,+}$ (i.e. improves the analycity radius)
\end{propen}

\begin{proof}[Proof]
Let $x = (A,\psi_N,L)$ be a point of $C$. Let $(e_1,e_2)$ be a basis of $K_m^D$ ($p^m\mathcal O/p^{2m}\mathcal O e_1 \oplus \mathcal O/p^{2m}\mathcal Oe_2 = K_n^D$) and denote by $(e_1',e_2')$ a similar basis for $A/L$ such that if $x_1,x_2$ and $x_1',x_2'$ denote the dual basis then $\pi^D : K_m^{',D} \fleche K_m^D$ in these basis is given by
\[
\left(
\begin{array}{cc}
 p  &   \\
  &  1  
\end{array}
\right).
\]
Let $(\Fil',w') \in p_2^*\mathcal T_{an}^\times/U_{an}$. As $\pi^D$ is a generic isomorphism on the multiplicative part, it in enough to check the proposition on $\mathcal F_{\sigma\tau}$.
Suppose $\widetilde{\pi}^*(\Fil',w') = (\Fil,w) \in p_1^*\mathcal{IW}^{0,+}_{\underline w}$, which means (on the $\sigma\tau$-factor),
\[ \frac{1}{p}\pi^*w_1' \in a_{1,1}\HT_{\sigma\tau,w}(e_1) + a_{2,1}\HT_{\sigma\tau,w}(e_2) + p^{w_0}\mathcal F_{\sigma\tau},\]
\[ \pi^*w_2' \in a_{2,2}\HT_{\sigma\tau,w}(e_2) + p^{w_0}\mathcal F_{\sigma\tau} + \Fil^1.\]
But then,
\[ \pi^*w_1' \in pa_{1,1}\HT_{\sigma\tau,w}(e_1) + pa_{2,1}\HT_{\sigma\tau,w}(e_2) + p^{w_0+1}\mathcal F_{\sigma\tau} = a_{1,1}\HT_{\sigma\tau,w}(\pi^De_1') + pa_{2,1}\HT_{\sigma\tau,w}(\pi^De_2') + p^{w_0+1}\mathcal F_{\sigma\tau},\]
and thus, as $p\mathcal F \subset\pi^*\mathcal F'$,
\[w_1' \in a_{1,1}\HT_{w,\sigma\tau}(e_1') + pa_{2,2}\HT_{w,\sigma\tau}(e_2') + p^{w_0}\mathcal F',\]
\[w_2' \in a_{2,2}\HT_{w,\sigma\tau}(e_2') + \Fil' + p^{w_0-1}\mathcal F'.\]
As $w_{2,2} \leq w_0 - 1$ we get the result.
\end{proof}

Suppose $v < \frac{1}{2(p^2+1)}$, and define ${\omega}^{\kappa\dag}_{\underline w'} = g_*\mathcal{IW}_{\underline w'}^{0,+}[\kappa]$. Suppose $w < m-1-\frac{p^{2m}-1}{p^2-1}v$.
We can then look at the following composition,
\[ H^0(\mathcal Y(v/p^2),{\omega}^{\kappa\dag}_{\underline w'}) \overset{p_2^\star}{\fleche} H^0(C,p_2^*\omega^{\kappa\dag}_w) \overset{(\widetilde{\pi}^\star)^{-1}}{\fleche} H^0(C,p_1^*\omega^{\kappa\dag}_w) \overset{\frac{1}{p^3}\Tr_{p_1}}{\fleche} H^0(\mathcal Y(v),\omega^{\kappa\dag}_w),\]
where $w'_{2,1} = w +1$  (remark that if $\kappa$ is $w+1$-analytic, there is an isomorphism between $g_*\mathcal{IW}_{\underline w'}^{0,+}[\kappa]$ and 
$\omega^{\kappa\dag}_{w_{2,1}}$).

\begin{remaen}
The normalisation of the Trace map is the same as in \cite{Bij}, the renormalisation of $\pi^\star$ giving a factor $p^{-k_2}$ on algebraic weights.
\end{remaen}

\begin{definen}
Suppose $v < \frac{1}{2(p^2+1)}$.
The operator $U_p$ is defined as the previous composition on bounded functions precomposed by $H^0(\mathcal Y(v),\omega_w^{\kappa\dag})\hookrightarrow H^0(\mathcal Y(v/p^2),\omega_{\underline w'}^{\kappa\dag})$. In particular it is compact as 
$H^0(\mathcal Y(v),\omega_w^{\kappa\dag})\hookrightarrow H^0(\mathcal Y(v/p^2),\omega_{\underline w'}^{\kappa\dag})$ is.
\end{definen}

\begin{propen}
Let $L$ be a finite extension of $K$, and $x,y \in X(v)(L)$ such that $y \in p_2(p_1^{-1}(x))$, and let $\kappa$ be a $\omega$-analytic character.
Then $U_p$ is identified with $\delta$, i.e. there is a commutative diagram,
\begin{center}
\begin{tikzpicture}[description/.style={fill=white,inner sep=2pt}] 
\matrix (m) [matrix of math nodes, row sep=3em, column sep=2.5em, text height=1.5ex, text depth=0.25ex] at (0,0)
{ 
(\omega_w^{\kappa'\dag})_y & &(\omega_w^{\kappa'\dag})_x\\
V_{0,\kappa',L}^{w-an} & & V_{0,\kappa',L}^{w-an} \\
 };
\path[->,font=\scriptsize] 
(m-1-1) edge node[auto] {$(\widetilde{\pi}^\star)^{-1}$} (m-1-3)
(m-2-1) edge node[auto] {$\delta$} (m-2-3)
(m-1-1) edge node[auto] {$$} (m-2-1)
(m-1-3) edge node[auto] {$$} (m-2-3)
;
\end{tikzpicture}
\end{center}
\end{propen}

We can also define an operator $S_p$, by considering the two maps,
\[ p_1,p_2 : X_{Iw(p)} \fleche X_{Iw(p)}\]
defined by $p_1 = \id$ and $p_2(A,\Fil(A[p])) = (A/A[p],p^{-1}\Fil(A[p])/A[p])$.
The map $p_2$ correspond to multiplication by $p$ on $\omega$ and the universal map $\pi : A \fleche A/A[p]$ over $X_{Iw(p)}$ induces a map,
\[\pi^* = p_2^*\mathcal T_{an}^\times \fleche \mathcal T_{an}^\times.\]
Define $\widetilde{\pi^*} = \frac{1}{p}\pi^*$ and consider the operator,
\[ H^0(X_{Iw(p)},\omega_w^{\kappa\dag}) \overset{p_2^*}{\fleche} H^0(X_{Iw(p)},p_2^*\omega_w^{\kappa\dag}) \overset{\widetilde{\pi^*}^{-1}}{\fleche} H^0(X_{Iw(p)},\omega_w^{\kappa\dag}).\]
The map $\pi^*$ preserves the Hasse invariant (as $A[p^\infty]/A[p] \simeq A[p^\infty]$) and sends the canonical filtration if it exists to itself (if $v < \frac{1}{4p}$).
In concrete terms, on the classical sheaf $\omega^{\kappa'} \subset \omega^{\kappa\dag}$, if we write $\kappa' = (k_1,k_2,k_3) \in \ZZ^3$ (and thus $\kappa = (-k_2,-k_1,-k_3)$)
the previous composition corresponds to a normalisation by $p^{-k_1 - k_2 - k_3}$ of the map that sends $f(A,dz_i) \mapsto f(A,pdz_i)$.

\begin{definen}
Define also the Hecke operator $S_p$ to be the previous composition. $S_p$ is invertible as $p$ is invertible in $\mathcal T^\times$.

We define the Atkin-Lehner algebra at $p$ as $\mathcal A(p) = \ZZ[1/p][U_p,S_p^{\pm 1}]$. It acts on the space of (classical) modular forms too.
\end{definen}

Classically it is also possible to define geometrically operators $U_p$ and $S_p$ at $p$ on classical modular forms of Iwahori level at $p$, and they obviously coïncide with ours through the inclusion of classical forms to overconvergent ones.
It is actually proven in \cite{Bijmu} that these Hecke operators preserve a strict neighborhood of the canonical-$\mu$-ordinary locus of $X_I$, given in terms of the degree.

\begin{remaen}
Because of the normalisation, the definition of the Hecke operator slightly differs with the one by convolution on automorphic forms. The reason is that the Hodge-Tate or automorphic weights do not vary continuously in families. This is already the case in other constructions. Let us be more specific. Let $f \in H^0(X_{Iw,K},\omega^{\kappa})$ a classical automorphic form of weight $\kappa = (k_1,k_2,k_3)$ and Iwahori level at $p$. To $f$, as explained in proposition \ref{propautheck} is associated a (non-scalar) automorphic form 
$\Phi_f$ (and a scalar one $\varphi_f$ whose Hecke eigenvalues are the same as the one of $\Phi_f$).
The Hecke action on $f$ and ${\Phi_f}$ is equivariant for the classical (i.e. non renormalised) action at $p$, more precisely at $p$ if we denote $S_p$ and $U_p$ 
the previous (normalised operators) the classical one are $S_p^{class} = p^{|k|}S_p$ and $U_p^{class} = p^{k_2}U_p$. The operators $S_p^{class}$ and $U_p^{class}$ 
corresponds to the two matrices,
\[
\left(
\begin{array}{ccc}
p  &   &   \\
  & p  &   \\
  &   &   p
\end{array}
\right) \quad \text{and} \quad \left(
\begin{array}{ccc}
p^2  &   &   \\
  & p  &   \\
  &   &   1
\end{array}
\right) \in GU(2,1)(\QQ_p).
\]
Their similitude factor is in both cases $p^2 = N(p)$.
 Let $f \in H^0(X_{Iw,K},\omega^\kappa)$ be a classical eigenform that is proper for the Hecke operator $U_p$ and $S_p$, of respective eigenvalues $\mu,\lambda$, then $\varphi_f$ has eigenvalues for the corresponding (non-normalised) Hecke operators at $p$, $p^{k_2} \mu$ and $p^{k_1+k_2+k_3}\lambda$.

\end{remaen}

\subsection{Remarks on the operators on the split case}
\label{Heckenormsplit}

When $p$ splits in $E$, the eigenvariety for $U(2,1)_{E}$ is a particular case of Brasca's construction (see \cite{Bra}). Unfortunately as noted by Brasca, there is a slight issue with 
the normalisation of the Hecke operators at $p$ constructed in section 4.2.2. of \cite{Bra}, where there should be a normalisation in families that depends on the weights, as in \cite{Bijmu} section 
2.3.1 for classical sheaves (without this normalisation Hecke operators do not vary in family). More explicitely on the split Picard case, we have 4 Hecke operators at $p$ 
(Bijakowski only consider two of them, which are relevant for classicity), denoted $U_i, i = 0,\dots,3$, following \cite{Bijmu}, section 2.3.1 (allowing $i =0$ and $i = 3$).
 The normalisations are the following on classical weights,
 \begin{eqnarray*}
  U_0 &= &\frac{1}{p^{k_3}}{U_0^c} \\
 U_1 &= &{U_1^c} \\
  U_2 &= &\frac{1}{p^{k_2}}{U_2^c} \\
 U_3 &= &\frac{1}{p^{k_1+k_2}}{U_3^c} \\
 \end{eqnarray*}
where $U_i^c$ denotes the classical Hecke operators, and we choose a splitting of the universal $p$-divisible group $A[p^\infty] = A[v^\infty] \times A[\overline{v}^\infty]$ and $A[\overline{v}^\infty] = A[{v}^\infty]^D$, where $v$ coincide with $\tau_\infty$ through the fixed isomorphism $\CC \simeq \overline{\QQ_p}$. Thus  $G=A[v^\infty]$ has height 3 and dimension 1, and modular forms of weight $\kappa = (k_1 \geq k_2,k_3) \in \ZZ^3_{dom}$ are sections of,
\[ \Sym^{k_1-k_2}\omega_{G^D} \otimes (\det \omega_{G^D})^{k_2} \otimes \omega_{G}^{\otimes k_3}.\]

\subsection{Classicity results}
In this section, we will prove a classicity result. As in \cite{AIP}, this is realised in two steps. First show that a section in $M_\kappa^\dag$ is actually a section of $\omega^{\kappa'}$ over $X(v)$ (this is called a result of classicity at the level of sheaves), then we show that this section extends to all $X_{Iw}$, but this is done in \cite{Bijmu}.

If $n$ is big enough, there is an action of $I_{p^n} \subset \GL_2 \times \GL_1$ on $\mathcal{IW}_w^{0,+}$ which can be derived as an action of $U(\mathfrak g)$ on $\mathcal O_{\mathcal{IW}^{0,+}_w}$ denoted by $\star$. As in section \ref{sectBGG}, let $\kappa = (k_1,k_2,r)$ be a classical weight, and we denote by $d_\kappa$ the map,
\[ f \in \mathcal O_{\mathcal{IW}^{0,+}_w} \mapsto X^{k_1-k_2+1} \star f,\]
which sends $\omega_w^{\kappa\dag}$ to $\omega_w^{(k_2-1,k_1+1,r)\dag}$.

\begin{propen}
Let $\kappa = (k_1,k_2,r)$ be a classical weight. There is an exact sequence of sheaves on $X(v)$,
\[ 0 \fleche \omega^{\kappa'} \fleche \omega_w^{\kappa\dag} \overset{d_\kappa}{\fleche} \omega_w^{(k_2-1,k_1+1,r) \dag}.\]
\end{propen}

\begin{proof}[Proof]
This is exactly as in \cite{AIP} Proposition 7.2.1 (we don't need assumption on $w$ as $V_{\kappa,L}^{0,w-an}$ is isomorphic to analytic functions on 1 ball only and Jones theorem applies, \cite{Jones}).
\end{proof}

\begin{propen}
\label{proclasssheaf}
\label{thrclassfaisc}
On $\omega_w^{\kappa\dag}$ we have the following commutativity,
\[ U_p\circ d_\kappa = p^{-k_1+k_2 -1}d_\kappa \circ U_p.\]
In particular, if  $H^0(X(v),\omega_w^{\kappa\dag})^{<k_1-k_2+1}$ denotes the union of generalised eigenspaces for eigenvalues of slope smaller than $k_1-k_2+1$, and $f \in H^0(X(v),\omega_w^{\kappa\dag})^{<k_1-k_2+1}$, then $f \in H^0(X(v),\omega^{\kappa'})$.
\end{propen}

\begin{proof}[Proof]
We can work etale-locally in which case by the previous results on $\omega_w^{\kappa\dag}$ locally the first part reduces to section \ref{sectBGG}.
Now, if $f$ is a generalised eigenvector for $U_p$ of eigenvalue $\lambda$ of slope (strictly) smaller than $k_1-k_2+1$, then $d_\kappa f$ is generalised eigenvector of slope 
 $\lambda p^{-k_1+k_2-1}$ 
which is of negative valuation, but this is impossible as $U_p$ (and etale-locally $\delta$) if of norm strictly less than 1. Thus $d_\kappa f = 0$ and $f$ is a section 
of $\omega^\kappa$.%\todo{eigenspace}
\end{proof}

The previous result is sometimes referred to as a classicity at the level of sheaves. Moreover, we have the following classicity result of S.Bijakowski, \cite{Bijmu}

\begin{theoren}[Bijakowski]
\label{thrclassbij}
\label{thrbij}
Let $f$ be an overconvergent section of the sheaf $\omega^\kappa$, $\kappa = (k_1 \geq k_2,k_3)$, which is proper for $U_p$ of eigenvalue $\alpha$. Then if,
\[ 3 + v(\alpha) < k_2 + k_3,\]
then $f$ is a classical form of weight $\kappa$ and level $K^pI$.
\end{theoren}

\section{Constructing the eigenvariety}

In this section we will construct the eigenvariety associated to the algebra $\mathcal H\otimes \mathcal A(p)$ and the spaces of overconvergent modular forms $M_\kappa^\dag$. In order to do this, we will use Buzzard's construction of eigenvarieties, and we need to show that the spaces $M_\kappa^\dag$ (and a bit more) are projective.
The method of proof follows closely the lines of \cite{AIP}, but as this case is simpler (because the toroïdal compactification is) we chose to write the argument in details.

\subsection{Projection to the minimal compactification}

\begin{definen}
Let $X^*$ be the minimal compactification of $Y$ as a (projective) scheme over $\Spec(\mathcal O)$. There is a map 
\[ \eta : X \fleche X^*,\]
from the toroidal to the minimal compactification. Denote $X^*_{rig}$ the rigid fiber and $X^*(v)$ the image of $X(v)$ in $X^*_{rig}$. If $v \in \QQ$ this is an affinoid as $X^*_{ord}$ is $(\det \omega$ is ample on the minimal compactification).
Denote also by $D$ the boundary in the toroïdal compactification $X$, and by abuse of notation in $X_1(p^{2m})$ and $X_1(p^{2m})(v)$.
\end{definen}

The idea to check that our spaces of cuspidal overconvergent modular forms are projective, is to push the sheaves to $X^*(v)$ which is affinoid and use the dévissage of \cite{AIP} Proposition A.1.2.2. But we need to show that the pushforward of the family of sheaves $\mathfrak w_w^{\kappa^{0,un}}(-D)$ is a small Banach sheaf.
In order to do this, we will do as in \cite{AIP} and prove that the pushforward of the trivial sheaf has no higher cohomology, and we will need to calculate this locally.

\subsubsection{Description of the toroïdal compactification}

Let $V = \mathcal O_E^3$ with the hermitian form $<,>$ chosen in the datum. For all totally isotropic factor $V'$, we denote $C(V/(V')^\perp)$ the cone of symmetric hermitian semi-definite positive forms on $(V/(V')^\perp)  \otimes_\ZZ \RR$ with rational radical. Denote by $\mathfrak C$ the set of such $V'$, and 
\[ \mathcal C = \coprod_{V' \in \mathfrak C \text{ non zero}} C(V/(V')^\perp).\]

\begin{remaen}
The subspaces $V'$ are of dimension 1 (if non zero), and $C(V/(V')^\perp) \simeq \RR_+$.
\end{remaen}

Fix $\psi_N$ level $N$ structure,
\[ \psi_N : (\mathcal O_E/N\mathcal O_E)^3\simeq V/NV\]
 and $\psi$ of level $p^{2m}$,
\[ \psi : \mathcal O_E/p^{2m}e_1 \oplus p^{m}\mathcal O_E/p^{2m}e_2 \subset V/p^{2m}V.\]

Let $\Gamma \subset G(\ZZ)$ be the congruence subgroup fixing the level outside $p$, and $\Gamma_1(p^{2m})$ fixing $\psi_N$ and $\psi$. Suppose that $N$ is big enough so that $\Gamma$ is neat.
Fix $\mathcal S$ a polyedral decomposition of $\mathcal C$ which is $\Gamma$-admissible : on each $C(V/(V')^\perp) = \RR_+$ there is a unique polyedral decomposition and thus there is a unique decomposition $\mathcal S$ and it is automatically $\Gamma$ (or $\Gamma_1(p^{2m})$)-admissible, smooth, and projective.

Recall the local charts of the toroïdal compactification $X$.
For each $V' \in \mathfrak C$ non zero, we have a diagram,
\begin{center}
\begin{tikzpicture}[description/.style={fill=white,inner sep=2pt}] 
\matrix (m) [matrix of math nodes, row sep=3em, column sep=2.5em, text height=1.5ex, text depth=0.25ex] at (0,0)
{ 
\mathcal M_{V'}&  & \mathcal M_{V',\sigma}  \\
&\mathcal B_{V'} & \\
& Y_{E} & \\
 };

%\draw[double,double distance=5pt] (m-1-1) – (m-1-3);
\path[->,font=\scriptsize] 
(m-1-1) edge node[auto] {$$} (m-1-3)
(m-1-1) edge node[auto] {$$} (m-2-2)
(m-1-3) edge node[auto] {$$} (m-2-2)
(m-2-2) edge node[auto] {$$} (m-3-2);
\end{tikzpicture}
\end{center}
where $Y_{E}$ is the moduli space of elliptic curves with complex multiplication by $\mathcal O_E$ of principal level $N$ structure, 
denote by $\mathcal E$ the universal elliptic curve, then $\mathcal B_{V'} = Ext^1(\mathcal E,\mathbb G_m \otimes O_E)$ is isogenous to $^t\mathcal E$, and is a $\mathbb G_m$-torsor, $\mathcal M_{V'} \fleche \mathcal B_{V'}$ is a $\mathbb G_m$-torsor, it is the moduli space of principally polarised 1-motives, with $\psi_N$-level structure, and $\mathcal M_{V'} \fleche \mathcal M_{V',\sigma}$ is an affine toroidal embedding associated to the cone decomposition of $C(V/(V')^\perp)$, locally isomorphic over $\mathcal B_{V'}$ to $\mathbb G_m \subset \mathbb G_a$.

Over $\mathcal B_{V'}$ we have a semi-abelian scheme of constant toric rank,
\[ 0 \fleche \mathbb G_m \otimes_\ZZ \mathcal O_E \fleche \widetilde G_{V'} \fleche \mathcal E \fleche 0.\]
Denote by $Z_{V'}$ the closed stratum of $\mathcal M_{V',\sigma}$.

Recall that $X$ is the toroïdal compactification of our moduli space $Y$ (it is unique as the polyedral decomposition $\mathcal S$ is), as defined in \cite{Lars} or \cite{Lan} in full generality, and $X^*$ is the minimal compactification. $X$ is proper and smooth and $X^*$ is proper. Moreover we have a (proper) map,
\[ \eta : X \fleche X^*.\]
Moreover, $\eta$ is the identity on $Y$. As sets, $X^*$ is a union of $Y$ to which we glue points corresponding to elliptic curves with complex multiplication, one for each component of $D$, the boundary of $X$, and over each $x \in X^* \backslash Y$, $\eta^{-1}(x)$ is a CM elliptic curve.

Denote by $\widehat{\mathcal M_{V',\sigma}}$ the completion of $\mathcal M_{V',\sigma}$ along the closed stratum $Z_{V'}$. On $X$ there is a stratification indexed
by $\mathfrak C/\Gamma$ (the open subset $Y$ corresponds to $V' = \{0\}$). For all non zero $V'$, the completion of $X$ along the $V'$ stratum is isomorphic to 
$\widehat{\mathcal M_{V',\sigma}}$, as $\Gamma_{V'}$, the stabilizer of $V'$, is trivial : $V' \simeq \mathcal O_E$ so $\Gamma_{V'} \subset \mathcal O_E^\times$, which is finite as $E$ is quadratic imaginary, and thus because $\Gamma$ is neat, $\Gamma_{V'} = \{1\}$.

As the Hasse invariant on the special fiber of $X$ is defined as the one of the abelian part of the sem-abelian scheme, we can identify it with the same one on $\mathcal M_{V',\sigma}$, which comes from the special fiber of $Y_{V'} \simeq Y_E$.
Denote by $\mathcal Y$, $\mathcal X$ the formal completions of $Y,X$ along the special fiber. We have defined $\mathfrak X(v) \fleche \mathfrak X$ as an open subset of a blow up, denote by $\mathfrak Y(v)$ the inverse image of $\mathfrak Y$, and we will describe its boundary locally. Denote
\begin{itemize}
\item $\mathfrak Y_{E}$ the formal completion along $p$ of $Y_{V'}$.
\item $\mathfrak Y_E(v)$ the open subset of $\mathcal Y_E(v)$ along $I = (p^v,\Ha_\tau)$ where $I$ is generated by $\Ha_\tau$, but as every $CM$ elliptic curve is $\mu$-ordinary, $\mathfrak Y_E(v) = \mathfrak Y_E$.
\item $\mathfrak B_{V'}$ the formal completion of $\mathcal B_{V'}$.
\item We define similarly, $\mathfrak M_{V'},\mathfrak M_{V',\sigma},\mathfrak Z_{V'}$. 
\end{itemize}

\begin{propen}
The formal scheme $\mathfrak X(v)$ has a stratification indexed by $\mathfrak C/\Gamma$, and the stratum corresponding to $V'$ is isomorphic to 
$\mathfrak Z_{V'}$ if $V'$ is non zero, and $\mathfrak Y(v)$ if $V' = \{0\}$. For all non zero $V' \in \mathfrak C$ the completion of $\mathfrak X(v)$ along the $V'$ stratum is isomorphic to $\widehat{\mathfrak M_{V',\sigma}}$ (completion along $\mathfrak Z_{V'}$).
\end{propen}

\begin{proof}[Proof]
We complete and pullback the stratification of $X$. The analogous result on $\mathfrak X$ is true since we can invert the completion along $p$ and the stratum.
If $V' \neq 0$ it is simply that the boundary of $\mathfrak X$ is inside the $\mu$-ordinary locus. For $V'= 0$  the stratum is the pull back of $\mathfrak Y$ inside $\mathfrak X(v)$, i.e. $\mathfrak Y(v)$.
\end{proof}

We used the space $\mathfrak X_1(p^{2m})$ in the previous sections, we would like to describe its boundary.

Let $\mathfrak C'$ be the subset of $V' \in \mathfrak C$ such that $\im (\psi) \subset (V')^\perp/p^{2m}(V')^\perp$. The (unique) polyhedral decomposition previously considered induces also a (unique) polyhedral decomposition on
\[ \coprod_{V' \in \mathfrak C' \text{ non zero}} C(V/(V')^\perp),\]
which is $\Gamma_1(p^{2m})$ admissible.

For $V' \in \mathfrak C'$ non zero, decompose,
\[ 0 \fleche V'/p^{2m} \fleche (V')^\perp/p^{2m} \fleche (V')^\perp/(V' + p^{2m} (V')^\perp)\fleche 0,\]
and denote $W$ the image in $ (V')^\perp/(V' + p^{2m} (V')^\perp)$ of $\psi(\mathcal O/p^{2m} \oplus p^m\mathcal O/p^m)$.
This is isomorphic to $\mathcal O/p^m$. Indeed, as $(V')^\perp$ contains $e_1,p^me_2$ modulo $p^{2m}$, $(V')^\perp/p^{2m} = \mathcal O/p^{2m}(e_1,e_2)$.
Then $\overline{V'} = V'/p^{2m}$ is totally isotropic inside, i.e. generated by $ae_1 + be_2$ where $p^m | b$ (totally isotropic) and $a \in \mathcal O^\times$ (direct factor).
Thus the image of $\psi$ is generated in $(V')^\perp/(V' + p^{2m} (V')^\perp)$ by the image of $e_1 = a^{-1}be_2$ which is $p^m$-torsion.

We denote by,
\begin{enumerate}
\item $\mathcal Y_{V'}$ the rigid fiber of $\mathfrak Y_{V'}$,
\item $H_{m,V'}$ the canonical subgroup of level  $m$ of the universal elliptic scheme $\mathfrak E_{V'}$ over $\mathfrak Y_{V'}$,
\item $\mathcal Y_1(p^m)_{V'}$ the torsor $\Isom_{\mathcal Y_{V'}}((H_{m,V'})^D,W^\vee)$, and $\psi_{V'}$ the universal isomorphism,
\item $\mathfrak Y_1(p^m)$ the normalisation of $\mathfrak Y_{V'}$ in $\mathcal Y_1(p^m)_{V'}$, 
\item There is an isogeny $i : \mathfrak B_{V'} \fleche {\mathfrak E}_{V'}$, and if we denote $i' : {\mathfrak E_{V'}} \fleche  {\mathfrak E_{V'}}/H_{m,V'}$, set,
\[ \mathfrak B_1(p^m)_{V'}(v) = \mathfrak B_{V'}\times_{i,{\mathfrak E},i'} {\mathfrak E_{V'}}/H_{m,V'}.\]
\item Denote $\mathfrak M_1(p^m)_{V'}, \mathfrak M_1(p^m)_{V',\sigma}, \mathfrak Z_1(p^m)_{V',\sigma}$ the fibered products of the corresponding formal schemes with $\mathfrak B_1(p^m)_{V'}$ over $\mathfrak B_{V'}$.
\end{enumerate}

\begin{propen}
The formal scheme $\mathfrak X_1(p^{2m})(v)$ has a stratification indexed by $\mathfrak C'/\Gamma_1(p^{2m})$, for all non zero $V'$, the completion of  $\mathfrak X_1(p^{2m})(v)$ along the $V'$-stratum is isomorphic to the completion $\widehat{\mathfrak M_1}(p^{2m})_{V',\sigma}$ along $\mathfrak Z_1(p^{2m})_{V'}$.
\end{propen}

\begin{proof}[Proof]
This is known in rigid fiber, with the same construction, but the previous local charts are normal, and thus coincide with the normalisation in their rigid fiber 
of the level $\Gamma$-charts. Thus $\mathfrak M_{1}(p^{2m})_{V',\sigma}$ is the normalisation of $\mathfrak M_{V',\sigma}$ in 
$\mathcal M_{1}(p^{2m})_{V',\sigma}$. But the completion of $\mathfrak M_{1}(p^{2m})_{V',\sigma}$ along $\mathfrak Z_{1}(p^{2m})_{V',\sigma}$ 
coincides with the normalisation of $\widehat{\mathfrak M_{V',\sigma}}^{V'} = \widehat{\mathfrak X(v)}^{V'}$ inside 
$\mathcal M_1(p^{2m})_{V',\sigma} = \widehat{X_1(p^n)(v)}^{V'}$.
 
\end{proof}

\subsection{Minimal compactification}

Let $X^*$ be the minimal compactification of $Y$ of level $\Gamma$. As a topological space, it corresponds to adding a finite set of points to $Y$, corresponding 
to CM elliptic curves. $X^*$ is also stratified by $\mathfrak C/\Gamma$. Let $\overline x \in X^*\backslash Y$ a geometric point of the boundary, it corresponds to a point $x \in Y_E$.

Using the previous description of $X$, we can describe the local rings of $X^*$.
Let $V' \in \mathfrak C$ non zero. Over $\mathcal B_{V'}$, $\mathcal M_{V'}$ is an affine $\mathbb G_m$-torsor, and we can thus write,
\[ \mathcal M_{V'} = \Spec_{B_V}\mathcal L,\]
where $\mathcal L$ is a quasi-coherent $\mathcal O_{\mathcal B_{V'}}$-algebra endowed with an action of $\mathbb G_m$, that can be decomposed,
\[ \mathcal L = \bigoplus_{k \in \ZZ} \mathcal L(k).\]
For all $k$, $\mathcal L(k)$ is locally free of rank 1 over $\mathcal B_{V'}$.
Denote $\widehat{B_{V';\overline x}}$ the completion of $\mathcal B_{V'}$ along the fiber over $\overline x$. We have the
\begin{propen}
$X^*$ is stratified by $\mathfrak C/\Gamma$ and $\eta : X \fleche X^*$ is compatible with stratifications. Moreover, for all $V'$, $X_{V'}^*$ is isomorphic to $Y_{V'}$ and for all $\overline{x} \in X_{V'}^\star$ a geometric point,
\[ \widehat{\mathcal O_{X^*,\overline x}} = \prod_{k \in \ZZ} H^0(\widehat{B_{V',\overline x}},\mathcal L(k)),\]
where $\widehat{\mathcal O_{X^*,\overline x}}$ is the completion of the strict henselisation of $\mathcal O_{X^*}$ at $\overline x$.
\end{propen}

\begin{proof}[Proof]
This is Serre's theorem on global sections of the structure sheaf on proper schemes (as $\eta$ is proper and $X^*$ is normal), the theorem of formal functions and the previous description of $X$.
\end{proof}

The Hasse invariant $\ha_\tau$ descends to the special fiber of $X^*$ and we can thus define $\mathfrak X^*$ the formal completion of $X^*$ along its special fiber and $\mathfrak X^*(v)$ the normalisation of the open subspace of the blow up of $X^*$ along $(p^v,\ha_\tau)$ where this ideal is generated by $\ha_\tau$.

\begin{propen}
For all $V' \in \mathfrak C$ the $V'$-strata of $\mathfrak X^*(v)$ is $\mathfrak Y_{V'}(v)$ (and $Y_E$ if $V'$ is non zero).
\end{propen}

\begin{proof}[Proof]
%\todo{???}
This is known before the blow up, and thus for $V'$ non zero as the boundary is contained in the $\mu$-ordinary locus. But for $V' = 0$ this is tautologic.
\end{proof}

\subsection{Higher Cohomology and projectivity of the space of overconvergent automorphic forms}

We will look at the following diagram,

\begin{center}
\begin{tikzpicture}[description/.style={fill=white,inner sep=2pt}] 
\matrix (m) [matrix of math nodes, row sep=3em, column sep=2.5em, text height=1.5ex, text depth=0.25ex] at (0,0)
{ 
\mathfrak X_{1}(p^n)(v) &  & \mathfrak X(v)  \\
&\mathfrak X^*(v) & \\
 };

%\draw[double,double distance=5pt] (m-1-1) – (m-1-3);
\path[->,font=\scriptsize] 
(m-1-1) edge node[auto] {$\pi_4$} (m-1-3)
(m-1-1) edge node[auto] {$\eta$} (m-2-2)
(m-1-3) edge node[auto] {$\pi$} (m-2-2);
\end{tikzpicture}
\end{center}

\begin{propen}
Let $D$ be the boundary of $\mathfrak X_1(p^{2m})(v)$. Then for all $q>1$,
\[ R^q \eta_*\mathcal O_{\mathfrak X_1(p^{2m})(v)}(-D) = 0.\]
\end{propen}

\begin{proof}[Proof]
It is enough to work locally at $\overline x$ a geometric point of the boundary of $\mathfrak X^*(v)$, and by the theorem of formal functions, 
\[ \widehat{(\eta_*\mathcal O_{\mathfrak X_1(p^{2m})(v)}(-D))_{\overline x}} = H^q(\widehat{\mathfrak X_1(p^{2m})(v)}^{\eta^{-1}(x)},\mathcal O_{\widehat{\mathfrak X_1(p^{2m})(v)}^{\eta^{-1}(x)}}(-D)).\]
We will thus show that the right hand side is zero.
But the completion $\widehat{\mathfrak X_1(p^{2m})(v)}^{\eta^{-1}(x)}$ is isomorphic to a finite disjoint union of  spaces of the form $\widehat{\mathfrak M_1(p^{2m})_{V',\sigma}}^{\overline y}$ for $\overline y$ a geometric point in $Y_E$. Denote by $M_\sigma$ this completed space.
As \[M_\sigma = \Spf_{\widehat{\mathfrak B_1(p^m)_{V'}}}(\widehat{\bigoplus_{k \geq 0}}\mathcal L(k)),\]
and thus the morphism,
\[ M_\sigma \fleche \widehat{\mathfrak B_1(p^m)_{V'}},\]
is affine, we have the equality, \[ H^q(M_\sigma,\mathcal O(-D)) = \prod_{k >0}H^q(\widehat{\mathfrak B_1(p^m)_{V'}},\mathcal L(k)),\]
(the product is over $k >0$ as we take the cohomology in $\mathcal O(-D)$). But for $k >0$, $\mathcal L(k)$ is very ample on the elliptic curve $\mathfrak B_1(p^m)_{V'}$, and thus $H^q(\widehat{\mathfrak B_1(p^m)_{V'}},\mathcal L(k)) = 0$ for all $q >0$.
\end{proof}

\begin{theoren}
For $m \geq l$ two integers, we have the following commutative diagram,
\begin{center}
\begin{tikzpicture}[description/.style={fill=white,inner sep=2pt}] 
\matrix (m) [matrix of math nodes, row sep=3em, column sep=2.5em, text height=1.5ex, text depth=0.25ex] at (0,0)
{ 
X_{1}(p^n)(v)_l &  &X_{1}(p^n)(v)_m  \\
X_*(v)_l& &X_*(v)_m \\
 };

%\draw[double,double distance=5pt] (m-1-1) – (m-1-3);
\path[->,font=\scriptsize] 
(m-1-1) edge node[auto] {$i$} (m-1-3)
(m-2-1) edge node[auto] {$i'$} (m-2-3)
(m-1-1) edge node[auto] {$\eta_l$} (m-2-1)
(m-1-3) edge node[auto] {$\eta_m$} (m-2-3);
\end{tikzpicture}
\end{center}
and the following base change property is satisfied,
\[ i'^*\left((\eta_m)_*\mathfrak w_{w,m}^{\kappa^0\dag}(-D)\right) = (\eta_l)_*\mathfrak w_{w,l}^{\kappa^0\dag}(-D).\]
In particular, $\eta_*\mathfrak w_{w}^{\kappa^0\dag}(-D)$ is a small Banach sheaf on $\mathfrak X^*(v)$. The same result is true over $\mathfrak X^*(v) \times \mathfrak W(w)^0$ for
\[ (\eta\times 1)_*\mathfrak w_w^{\kappa^{0,un}\dag}(-D).\]
\end{theoren}

\begin{proof}[Proof]
We can just restrict to $l = m-1$, but as inductive limit and direct image commute, and as the kernel $\mathfrak w_{w,m}^{\kappa^{0,un}\dag} \fleche \mathfrak w_{w,m-1}^{\kappa^{0,un}\dag}$ is isomorphic to $\mathfrak w_{w,1}^{\kappa^{0,un}\dag}$ which is itself a direct limit of sheaves with graded pieces isomorphic to $\mathcal O_{X_1(p^m)}/\pi$ (see corollary \ref{cor95}) and thus by the previous proposition we have the announced equality. We can thus use \cite{AIP} Proposition A.1.3.1 which proves that $(\eta\times 1)_*\mathfrak w_w^{\kappa^{0,un}\dag}$ is a small formal Banach sheaf (Recall that $\eta$ is proper).
\end{proof}

\begin{propen}
Let $w >0$. Denote $\mathfrak W(w)^0 = \Spf(A)$. Then,
\[ M_{v,w}^{\kappa^{0,un}\dag,cusp} = H^0(X^*(v) \times \mathcal W(w)^0,(\eta\times 1)_*\omega_w^{\kappa^{0,un}\dag}(-D))\]
is a projective $A[1/p]$-Banach module. Moreover the specialisation map, for $\kappa \in \mathcal W(w)^0$,
\[ M_{v,w}^{\kappa^{0,un}\dag,cusp} \fleche H^0(X^*(v),\eta_*\omega_w^{\kappa\dag}(-D)),\]
is surjective.
\end{propen}

\begin{proof}[Proof]
This is proved exactly as in \cite{AIP}, Corollary 8.2.3.2. Let us sketch the ideas. Fix $(\mathfrak U_i)_{1\geq i \geq r}$ a (finite) affine covering of $\mathfrak X^*(v)$, and for $\underline i = (i_1,\dots,i_s) \in \{1,\dots,r\}^s$ denote $\mathfrak U_{\underline i}$ the interesection $\mathfrak U_{i_1} \cap \dots \mathfrak U_{i_s}$.
Then,
\[ M_{\underline i,\infty} = H^0(\mathfrak U_{\underline i}\times \mathfrak W(w)^0,(\eta\times1)_*\mathfrak w_w^{\kappa^{0,un}\dag}(-D)),\]
is isomorphic to the $p$-adic completion of a free $A$-module (i.e. is orthonormalisable). This is essentially Corollary \ref{cor95} and topological Nakayama's lemma.
But then, as $X^*(v)$ is affinoid, the Cech complex after inverting $p$ is exact and thus (\cite{AIP} Theorem A.1.2.2) provides a resolution of $M_{v,w}^{\kappa^{0,un}\dag,cusp}$ by the $M_{\underline i,\infty}[1/p]$, and thus $M_{v,w}^{\kappa^{0,un}\dag,cusp}$ is projective.
For the surjectivity assertion, fix $\mathfrak p_{\kappa}$ the maximal ideal of $A[1/p]$ corresponding to $\kappa$ and consider the Koszul resolution of 
$A[1/p]/\mathfrak p_{\kappa}$. Tensoring this for each $\underline i$ with $\eta_*\mathfrak w^{\kappa^{0,un}\dag}(-D)(\mathfrak U_{\underline i})$ gives a resolution of
 $\eta_*\mathfrak w_w^{\kappa\dag}(-D)(\mathfrak U_{\underline i})$. This gives a double complex where each column (for a fixed index of the Koszul complex) is 
 exact. But each line (for a fixed $\underline i$ non trivial) is also exact by the previous acyclicity, and thus we have the following bottom right square,
\begin{center}
\begin{tikzpicture}[description/.style={fill=white,inner sep=2pt}] 
\matrix (m) [matrix of math nodes, row sep=3em, column sep=2.5em, text height=1.5ex, text depth=0.25ex] at (0,0)
{ 
\prod_{1 \geq i \geq r} \omega^{\kappa^{0,un}\dag}(\mathcal U_i) &  & \prod_{1 \geq i \geq r} \omega^{\kappa\dag}(\mathcal U_i) & 0   \\
\omega^{\kappa^{0,un}\dag}(X^*(v) & &\omega^{\kappa\dag}(X^*(v)) &  \\
 0 & & 0 & \\
 };

%\draw[double,double distance=5pt] (m-1-1) – (m-1-3);
\path[->,font=\scriptsize] 
(m-1-1) edge node[auto] {$$} (m-1-3)
(m-1-3) edge node[auto] {$$} (m-1-4)
(m-2-1) edge node[auto] {$\pi_\kappa$} (m-2-3)
%(m-2-3) edge node[auto] {$$} (m-2-5)
(m-1-1) edge node[auto] {$$} (m-2-1)
(m-2-1) edge node[auto] {$$} (m-3-1)
(m-1-3) edge node[auto] {$$} (m-2-3)
(m-2-3) edge node[auto] {$$} (m-3-3)
;
\end{tikzpicture}
\end{center}
which proves that $\pi_\kappa$ is surjective.
\end{proof}

\begin{propen}
Denote $\Spm(B) = \mathcal W(w)$. Then the $B$-module
\[ H^0(\mathcal X(v) \times \mathcal W(w),\omega_w^{\kappa^{un}\dag}(-D))\]
is projective. Moreover, for every $\kappa \in \mathcal W(w)$, the specialisation  map,
\[  H^0(\mathcal X(v) \times \mathcal W(w),\omega_w^{\kappa^{un}\dag}(-D)) \fleche  H^0(\mathcal X(v) ,\omega_w^{\kappa\dag}(-D))\]
is surjective.
\end{propen}

\begin{proof}[Proof]
We can identify the $B$-module,
\[ M'  = H^0(\mathcal X(v) \times \mathcal W(w),\omega_w^{\kappa^{un}\dag}(-D)),\]
with $(M_{v,w}^{\kappa^{0,un}\dag,cusp}\otimes_{A[1/p]} B(-\kappa^{un}))^{B_n}$. But now $B_n$
is a finite group, and $B$ is of caracteristic zero, thus $M'$ is a direct factor in a projective $B$-module, and is thus surjective.
Moreover, as $B_n$ is finite, the (higher) 
group cohomology vanishes, and the specialisation map stays surjective.
\end{proof}

\subsection{Types}

Let $K_f^p$ be a compact open subgroup of $G(\mathbb A_f^p)$. Let $K_f = K_f^pI$ where $I \subset G(\QQ_p)$ is the Iwahori subgroup.
Fix $(J,V_J)$ a complex continuous irreducible representation of $K_f$, trivial at $p$ and outside a level $N$, it is of finite dimension and finite image, and 
thus defined over a number field.
Denote $K^0 \subset K_f$ its Kernel.

\begin{definen}
The space of Picard modular forms of weight $\kappa$, $v$-overconvergent, $w$-analytic, of type $(K_f,J)$ is,
\[ Hom_{K_f}(J,H^0(\mathcal X(v),\omega_w^{\kappa\dag})).\]
The space of overconvergent locally analytic Picard modular forms of weight $\kappa$ and type $(K_f,J)$ is then,
\[ M_\kappa^{\dag,(K_f,J)} = \Hom_{K_f}(J,\varinjlim_{v \rightarrow 0, w\rightarrow \infty} H^0(\mathcal X(v),\omega_w^{\kappa\dag}).\]
\end{definen}

\begin{remaen}
In the beginning of this section we made the assumption that the level $\Gamma$, outside $p$, is big enough ("neat"). But using the previous definition we can get 
rid of this assumption by taking $K_f^p$ big enough to have the neatness assumption, and take $J$ the trivial representation to descend our families for any level outside $p$, as the following proposition shows.
\end{remaen}

\begin{propen}
The space 
\[ M^{(K,J),\kappa^{un}\dag}_{cusp,v,w} := \Hom_{K_f}(J,H^0(\mathcal X(v)\times \mathcal W(w),\omega_w^{\kappa^{un}\dag}(-D))),\]
is a projective $\mathcal O_{\mathcal W(w)}$-module, and the specialisation map is surjective.
\end{propen}

\begin{proof}[Proof]
Suppose $K(N) = K \subset K^0 = \Ker(J) $ is neat (outside $p$, up to enlarging it). Then we have shown that, in level $K$,
\[ H^0(\mathcal X(v)\times \mathcal W(w),\omega_w^{\kappa^{un}\dag}(-D)),\]
is projective, and that the corresponding specialisation map is surjective. We can thus twist the $K_f/K$ action by $V_J^*$ and take the invariants over $J$; as $K_f/K$ is finite, the space if a direct factor inside $H^0(\mathcal X(v)\times \mathcal W(w),\omega_w^{\kappa^{un}\dag}(-D)) \otimes V_J^*$ and higher cohomology vanishes.
\end{proof}

\begin{remaen}
The same argument applies when $p$ splits in $E$, for the spaces of overconvergent modular forms defined in \cite{Bra}, \cite{AIP}. In particular we can construct families of Picard modular forms with fixed type when $p$ is unramified.
\end{remaen}

\subsection{Eigenvarieties}

\begin{theoren}
\label{thrHecke}
Let $p$ be a prime number, unramified in $E$. Let $\mathcal W$ be the $p$-adic weight space of $U(2,1)$, as defined in section \ref{sectW} when $p$ is inert, it is a disjoint union of 3-dimensional ball over $\QQ_p$ when $p$ splits.
Fix $(K_J,J)$ a type outside $p$, $K \subset \Ker J$ a neat level outside $p$, and let $S$ be the set of places where $K$ is not compact maximal or $p$. 
There exists an equidimensional of dimension 3 eigenvariety $\mathcal E$ and a locally finite map,
\[ w : \mathcal E \fleche \mathcal W,\]
such that for any $\kappa \in \mathcal W$, $w^{-1}(\kappa)$ is in bijection with eigensystems for $\mathbb T^S \otimes_\ZZ \mathcal A(p)$ acting on the space of overconvergent, locally analytic, modular forms of weight $\kappa$ and type-level $(K_J,J)$ (and Iwahori level at $p$), finite slope for $U_p$.
\end{theoren}

\begin{proof}
If $p$ is split this is a particular case of the main result of \cite{Bra} (taking into account the previous remark and the normalisation of the Hecke operators).
If $p$ is inert, this is a consequence of Buzzard-Coleman's machinery (\cite{Buz}) using for all compatible $v,w$ $(\mathcal W(w),M^{(K,J),\kappa^{un}\dag}_{cusp,v,w},U_p,\mathcal H^{Np}\otimes \mathcal A(p))$ and glueing along $v,w$.
\end{proof}

\subsection{Convention on weights}
\label{normweight}
As in \cite{BC1}, we set as convention the Hodge-Tate weight of the cyclotomic character to be -1. Fix an isomorphism $\CC \simeq \overline{\QQ_p}$ compatible with the inclusions 
$\overline{E} \subset \CC$ that extend $\tau_\infty$ and denote $\tau, \sigma_\tau$ the $p$-adic places at $p$ inert corresponding to $\tau_\infty, c\tau_{\infty}$. If $p$ is split, we will insted call $v,\overline v$ the places corresponding to $\tau_\infty, c\tau_{\infty}$, but in this section we focus on $p$ inert, even if a similar result hold with $v,\overline v$.

Let us recall the different parameters that are associated to an algebraic automorphic representation $\pi$ of $GU(2,1)$ that we will need, following partly \cite{Ski}. There is 
$\lambda = ((\lambda_1,\lambda_2,\lambda_3),\lambda_0)$, the Harrish-Chandra parameter, there is $c = ((c_1,c_2,c_3),c_0)$, the highest weight of the algebraic representation which has the same infinitesimal character as $\pi_\infty$ in the discrete series case and $(c_0,c_0')$ is the parameter at infinity of ${\omega_\pi}^c$ the conjugate of the central character of $\pi$. There is $\kappa = (k_1,k_2,k_3)$ the classical weight such that $\pi_f$ appears in 
$H^0(X_K,\omega^\kappa)$  (if it exists) and there are the Hodge-Tate weights $((h^\tau_1,h^\tau_2,h^\tau_3),(h^{\sigma\tau}_1,h^{\sigma\tau}_2,h^{\sigma\tau}_3))$ of 
the Galois representation of $G_E$ associated to $\pi$ by Blasius-Rogawski or Skinner. Let us explain how they are related in the case of discrete series.

First denote $\rho_n$ the half-sum of the positive non-compact roots and $\rho_c$ the half-sum of the positive compact roots (see \cite{Gold}, Section 5.3).
We have then for $i \geq 1, \lambda_i = (c + \rho_n + \rho_c)_i$, and $(c_0,c_0')$ is the infinite weight of the dual of the central character. 
The calculation of Harris and Goldring gives $(-k_3,k_1,k_2) = \lambda + \rho_n - \rho_c$ (forgetting the $\lambda_0$ factor here, it is because we only considered 3 parameters in the weight space). 
The Hodge-Tate weights of the Galois representation associated to $\pi$ depends of course on the normalisation of the correspondance, but take the one of Skinner, \cite{Ski} 
section 4.2,4.3 and after theorem 10, we get,
\begin{eqnarray*} ((h^\tau_1,h^\tau_2,h^\tau_3),(h^{\sigma\tau}_1,h^{\sigma\tau}_2,h^{\sigma\tau}_3)) = \\ ((-c_0 - c_1,-c_0 - c_2 + 1,-c_0 - c_3 +2),(- c_0' + c_3, -c_0' + c_2 + 1,-c_0' + c_1 + 2)).\end{eqnarray*}

\begin{remaen}
Let $f \in H^0(X,\omega^\kappa)$ be a classical form. To $f$ is associated $\Phi_f$ an automorphic form, with equivariant Hecke action, cf. section 4.2.2, and thus an autormorphic representation $\pi_f$.
\end{remaen}

Before going further, let us remark that a (algebraic) representation $\pi$ of $GU(2,1)$ is equivalent to a pair $(\pi^0,\psi)$ of $\pi^0$ a (algebraic) automorphic representation of $U(2,1)$ (the restriction of $\pi$) and a (algebraic) Hecke character of $GU(1) = Res_{E/\QQ}\mathbb G_m$ (the central caracter of $\pi$) which extend the central character of $\pi^0$ (see section \ref{secttransfert}). To an algebraic (nice) $\pi$ is associated a (non-necessarily polarized) Galois representation $\rho_\pi$, but also a pair $\pi^0,\psi_\pi$, and to $\pi^0$ is associated a polarized Galois representation, which is what we will need. Thus from Skinner's normalisation, removing the central character of $\pi$, we get the following proposition (we could also directly use \cite{CH}).

\begin{propen}
\label{proSki}
Let $\kappa = (k_1,k_2,k_3)  \in \ZZ^3$ and $f = H^0(X,\omega^\kappa)$ which is an eigenvector for the Hecke operators outside p. Denote $|k| = k_1 + k_2 + k_3$.
Let ${\pi}$ be the automorphic representation corresponding to $f$ (i.e. a irreducible factor in the representation generated by ${\Phi_f}$ of section \ref{sectaut}).

Suppose ${\pi}_\infty$ is a (regular) discrete series of Harrish-Chandra parameter $\lambda$, then \[\lambda = ((k_1,k_2-1,1-k_3),|k|)\] (see \cite{Gold} section 5) with $k_1 \geq k_2 > 2-k_3$.
Denote by $\rho_{{\pi},Ski}$ the $p$-adic Galois representation associated to ${\pi}$ by Skinner, \cite{Ski}. Then $\rho_{{\pi},Ski}$ satisfies the following essentially self-polarisation,
\[ \rho_{{\pi},Ski}^c \simeq \rho_{\pi,Ski}^\vee \otimes \eps_{cycl}^{-|k|-2}\otimes \rho_\psi,\]
where $\eps$ is the cyclotomic character, $\psi$ is a finite Hecke character, and if $\omega_\pi$ denote the central character of $\pi$, $\omega_\pi\omega_\pi^c$ is of the form $N^{-|k|}\psi$.

Then, the $\tau$-Hodge-Tate weights of $\rho_{\pi,Ski}$ are,
\[ (k_2+1,2 + k_1, k_3+k_1+k_2 ),\]
and the $\sigma\tau$-Hodge-Tate weights are,
\[ (2 , k_2+k_3,k_1+k_3+1).\]
To $\pi^0 = \pi_{|U(2,1)}$ is associated a polarised continuous Galois representation $\rho_\pi$ verifiying,
\[ \rho_\pi^c \simeq \rho_\pi^\vee,\]
of $\tau$-Hodge-Tate weights $(1-k_3,k_2-1,k_1)$ and (thus) $\sigma\tau$-Hodge-Tate weights $(-k_1,1-k_2,k_3-1)$.

\end{propen}

\begin{proof}
The calculation of $\lambda$ in terms of $\kappa$ is exactly \cite{Gold} Theorem 5.5.1. Remark also that we can calculate in terms of $\kappa$ which are the discrete series by Harish-Chandra Theorem (\cite{Knapp} Theorem 6.6), and we find,
$k_1 \geq k_2 > 2-k_3$. 
Thus, the calculation of the Hodge-Tate weights of the Galois representations associated to $\pi$ are \cite{Ski}, under Theorem 10, with the previous calculation of $c$ in terms of $\lambda$. The representation $\rho_\pi$ is given by $\rho_{\pi,Ski}\rho_{\overline{\omega_\pi}}^{-1}(1)$. \end{proof}

\begin{remaen}
In terms of the $\tau$-Hodge-Tate weights of $\rho_\pi$, discrete series correspond to $h_1 < h_2 < h_3$.
All of this is coherent with the BGG decomposition, see for example \cite{LanICCM} (2.3), for $\underline c = \mu = (a,b,c)$ a highest weight representation of $G$,
\[ H^2_{dR}(Y,V_\mu^\vee) = H^2(X,\omega^{(-b,-a,c)}) \oplus H^1(X,\omega^{(1-c,-a,b+1)}) \oplus H^0(X,\omega^{(1-c,1-b,a+2)}).\]
\end{remaen}

Denote $\mathcal Z \subset \mathcal E$ the set of characters corresponding to regular (i.e. $w_2(z) \leq w_1(z) < -2-w_3(z) \in \ZZ^3$) classical modular forms (recall that if $f$ is classical of weight $(k_1,k_2,k_3)$, $w(f) = (-k_2,-k_1,-k_3)$). For each $z \in \mathcal Z$, there exists $f$ a classical form, which determines $\Pi$ an automorphic representation of $GU(2,1)$ (generated by $\Phi_f$ defined in subsection \ref{sectaut}). Such a $\Pi$ correspond to a packet, to which by work of Blasius-Rogawski \cite{BR}, Theorem 1.9.1 (see also for generalisation to higher dimension unitary groups and local global compatibilities the work of many authors, in particular \cite{BelComp,CH,Ski,BLGHT}) is associated a number field $E_z$, and compatible system of Galois representations, 
\[ \rho_{z,\lambda} : G_E \fleche \GL_3(E_{z,\lambda}), \forall \lambda \in \Spm(\mathcal O_{E_z}),\]
satisfying local global compatibilities (see for example \cite{Ski}, where the association is normalized by the previous proposition (the Hodge-Tate 
weight of the cyclotomic character being $-1$) and the previous proposition for a normalisation suitable to our needs). In particular, denote $S$ the set of prime of $E$ where $Ker(J)I$ is not hyperspecial, and if $\ell$ a prime under $\lambda$, denote 
$S_\ell$ the set of places of $E$ dividing $\lambda$. Then $\rho_{z,\lambda}$, is unramified outside $SS_\ell$.

We have the classical proposition, which is one reason why eigenvarieties are so usefull (see for example \cite{BC2} proposition 7.5.4),

\begin{propen}
\label{thrT}
Let $p$ be unramified in $E$. To each $z \in \mathcal Z$, denote $\rho_z$ the ($p$-adic) polarized representation associated to $z$ by proposition \ref{proSki}. 
There exists a unique continuous pseudocharacter
\[ T : G_{E,S} \fleche \mathcal O_\mathcal E,\]
such that for all $z \in \mathcal Z$, $T_z = \tr(\rho_{z})$. Moreover the pseudocharacter $T$ satisfies $T^\perp = T$, where $T^\perp(g) = T((\tau g \tau)^{-1})$ for all $g \in G_E$.
\end{propen}

\begin{proof}[Proof]
$\mathcal Z$ is dense in $\mathcal E$ by density of very regular weights in $\mathcal W$ and the two classicity results (\ref{proclasssheaf} and \ref{thrbij}).
We only need \cite{Che1} Proposition 7.1 to conclude, the hypothesis $(H)$ there being verified by the Frobenius classes in $S$. The polarisation assumption follows from the case of $z \in \mathcal Z$ by density.
\end{proof}

\section{Application to a conjecture of Bloch-Kato}

Let $E$ be a quadratic imaginary field, and fix an algebraic Hecke character,
\[ \chi : \mathbb A^\times_E/E^\times \fleche \CC^\times,\]
such that, for all $z \in \CC^\times, \chi_{\infty}(z) = z^a\overline{z}^b$, for some $a,b \in \ZZ$. Call $w = -a-b$ the motivic weight of $\chi$. Denote by $\chi_p : G_E \fleche K$, where $K$ is a finite extension of $\QQ_p$, the $p$-adic realisation of $\chi$.

We are interested in the Selmer group $H^1_f(E,\chi_p)$, which parametrises extensions $U$,
\[ 0 \fleche \chi_p \fleche U \fleche 1 \fleche 0,\]
which have good reduction everywhere (\cite{BK,FPR}, \cite{BC2} Chapter 5 and the introduction of this article).

Associated to $\chi$ there is also an $L$-function $L(\chi,s)$, where $s$ is a complex variable, which is a meromorphic function on $\CC$, which satisfies a functional equation,
\[ \Lambda(\chi,s) = \eps(\chi,s)\Lambda(\chi^*(1),-s),\]
where $\chi^*$ is the contragredient representation, $\Lambda(V,s)$ is the completed $L$-function of $V$, a product of $L(V,s)$ by a finite number of $\Gamma$-factors.

The conjecture of Bloch-Kato (more precisely a particular case of)  in this case is the following equality,
\[ \dim H^1_f(E,\chi_p) - \dim (\chi_p)^{G_E} = \ord_{s=0} L(\chi^*(1),s).\]
The conjecture is more generally for a Galois representation $\rho$ of the Galois group $G_F$ of a number field, but in the previous case we have a special case by the 
theorem of Rubin on Iwasawa Main Conjecture for CM elliptic curves,

\begin{conjecture}
Suppose that $\chi$ is polarized and of weight $-1$, i.e.
\[ \chi^\perp = \chi |.|^{-1},\]
where $\chi^\perp(z) = \chi^{-1}(czc)$, and $c \in G_\QQ$ induces the complex conjugation in $E$. 
Then,
\[ \ord_{s = 0}L(\chi,s) \neq 0 \Rightarrow \dim H^1_f(E,\chi_p) \geq 1.\]
\end{conjecture}
 
\begin{remaen}
Under the previous polarisation assumption, we have $L(\chi^*(1),s) = L(\chi,s)$. The previous conjecture is mainly known by work of Rubin (\cite{Rub}) if $p \notdivides |\mathcal O_E^\times|$.
\end{remaen}

In the rest of the article, we will prove the following,
\begin{theoren}
Suppose that $\chi$ is polarized and of weight $-1$, i.e.
\[ \chi^\perp = \chi |.|^{-1},\]
where $\chi^\perp(z) = \chi^{-1}(czc)$, and $c \in G_\QQ$ induces the complex conjugation in $E$. 
Suppose that $p$ is unramfied in $E$ and that $p \notdivides \Cond(\chi)$. If $p$ is inert suppose moreover that $p \neq 2$.
Then,
\[ \ord_{s = 0}L(\chi,s) \neq 0 \text{ and } \ord_{s=0}L(\chi,s) \text{ is even} \Rightarrow \dim H^1_f(E,\chi_p) \geq 1.\]
\end{theoren}

\begin{remaen}
The same result but for $\ord_{s=0}L(\chi,s)$ odd (and strictly speaking for $p$ split) is proved in \cite{BC1}, and using the same method with the eigenvariety for the group $U(3)$.
\end{remaen}

\begin{definen}
To stick with notations of \cite{BC1}, denote $k$ the positive odd integer such that $\chi_\infty(z) = z^{\frac{k+1}{2}}\overline{z}^{\frac{1-k}{2}}$ (i.e. $k = 2a-1 = 1-2b$). We suppose $k \geq 1$, i.e. $a \geq 1$ (which we can always suppose up to changing $\chi$ by $\chi^c$, which doesn't change either the $L$-function nor the dimension of the Selmer group).
\end{definen}

\subsection{Endoscopic transfer, after Rogawski}

Let $\chi_0 = \chi |.|^{-1/2}$ the unitary character as in \cite{BC1}. We will define following Rogawski \cite{Rog} an automorphic representation of $U(2,1)$, by constructing it at each place.

\subsubsection{If $\ell$ is split in $E$}

Write $\ell = v\overline v$ and the choice of say $v$ induces an isomorphism $U(2,1)(\QQ_\ell) \overset{i_v}{\simeq} \GL_3(\QQ_\ell)$. Let $P = MN$ be the standard parabolique of $\GL_3(\QQ_p)$ with Levi $M = \GL_2\times \GL_1$. Define,
\[ \widetilde{\chi_{0,\ell}} : 
\left(
\begin{array}{ccc}
 A &      \\
  & b    
\end{array}
\right) \in \GL_2 \times \GL_1 \longmapsto \chi_{0,\ell}(\det A),
\]
trivially extended to $P$, and denote by $\ind-n_P^G(\widetilde{\chi_{0,\ell}})$ the normalised induction of $\widetilde{\chi_{0,\ell}}$. Then set,
\[ \pi_\ell^n(\chi) = i_v^*\ind-n_P^G(\widetilde{\chi_{0,\ell}}).\]
If $\chi_\ell$ is unramified, then so is $\pi^n_\ell(\chi)$. Fix in this case $K_\ell$ a maximal compact subgroup of $U(2,1)(\QQ_\ell)$.

\subsubsection{If $\ell$ is inert or ramified in $E$}

In this case denote $T = \mathcal O_{E_\ell}^\times \times \mathcal O_{E_\ell}^1$ the torus of $U(2,1)(\QQ_\ell)$, and consider the following character of $T$,
\[\widetilde{\chi}_{\ell} : 
\left(
\begin{array}{ccc}
 a &   &   \\
  & b  &   \\
  &   &  \overline{a}^{-1}
\end{array}
\right) \longmapsto \chi_\ell(a),
\]
trivially extended to the Borel $B$ of $U(2,1)(\QQ_\ell)$. Then the normalized induction $\ind-n_B^{U(2,1)(\QQ_\ell)}(\widetilde{\chi_\ell})$ has two Jordan-Holder factors, one which is non tempered that we denote by $\pi^n_\ell(\chi)$ and the other one, which is square integrable, that we denote by $\pi^2_\ell(\chi)$, see \cite{Rog}.

If $\ell$ is inert and $\chi_{0,\ell}$ is unramified, $\pi^n_\ell(\chi)$ is also unramified (Satake) and we can choose $K_\ell$ a maximal compact for which $\pi^n_\ell(\chi)$ has a non zero fixed vector.

If $\ell$ is ramified and $\chi_{0,\ell}$ is unramified, there is two conjugacy classes of maximal compact subgroup, but only one of them, denoted $K_\ell$ (called very special) 
verifies that 
$\pi_\ell^n(\chi)$ has a non-zero fixed vector under $K_\ell$ whereas $\pi^2_\ell(\chi)$ has none. 

\subsubsection{Construction at infinity}

As in the inert case, let $\pi^n_\infty(\chi)$ be the non-tempered Jordan-Holder factor of $\ind-n_B^{U(2,1)(\RR)}(\widetilde{\chi_\infty})$.

Then we have the following proposition, following Rogawski,

\begin{propen}[Rogawski]
\label{proprog}
Suppose $a\geq 1$. Recall that $\ord_{s=0}L(\chi,s)$ is even. Then the representation,
\[ \pi^n(\chi) = \bigotimes'_\ell \pi^n_\ell(\chi) \otimes \pi_\infty^n(\chi),\]
is an automorphic representation of $U(2,1)$. If moreover $L(\chi,0) = 0$, it is a cuspidal representation.
Its Galois representation (associated by the work of \cite{LRZ} or see also \cite{BC1} section 3.2.3 and Proposition 4.1) $\rho_{\pi^n(\chi),p} : G_E \fleche \GL_3(\overline{\QQ_p})$ is,
\[ \rho_{\pi^n(\chi),p} = (1 \oplus \chi_p \oplus \chi_p^\perp).\]
Moreover, its $\tau$-Hodge-Tate weights are $(-\frac{k+1}{2},-\frac{k-1}{2},0) = (-a, 1-a ,0)$.
\end{propen}

\subsection{Accessible refinement (at $p$) for $\pi^n(\chi)$}

In order to construct a $p$-adic family of modular forms passing through $\pi^n(\chi)$, we need to construct inside $\pi^n_p(\chi)^I$ a form which is proper for the operator $U_p$ previously defined. Strictly speaking, $U_p$ is defined for $GU(2,1)$, and when $p$ is inert, $U_p$ is associated to the operator of the double Iwahori classe $IU_p^cI$ where,
\[
U_p^c = \left(
\begin{array}{ccc}
p^2 &   &   \\
  &  p &   \\
  &   & 1  
\end{array}
\right)
\] This class is not in $U(2,1)(\QQ_p)$, but $p^{-1}U_p^c$ is.
Fix $T \subset B \subset U(2,1)(\QQ_p)$. As $p$ is unramified, we have as representation of $T/T^0$ ($T^0$ a maximal compact in $T$), for $\pi$ a representation of $U(2,1)(\QQ_p)$ an isomorphism, see \cite{BC1},
\[ \pi^I \simeq (\pi_N)^{T^0}\otimes \delta_B^{-1}.\]
Thus to understand how the double coset operator $U_p^c$ in the Iwahori-Hecke algebra acts, we only need to determine the Jacquet functor $(\pi^n_p(\chi))_N$ as a representation of $T$. If $p$ splits, this is computed in \cite{BC1} (and \cite{BC2} in greater generality), so suppose that $p$ is inert.

\begin{propen}
Let $\widetilde{\chi}$ be the (unramified) character of the torus $T$ of $U(2,1)(\QQ_p)$ defined by,
\[
\left(
\begin{array}{ccc}
 a &   &   \\
  &  e &   \\
  &   & \overline{a}^{-1}  
\end{array}
\right) \longmapsto \chi_p(a).
\]
Denote by $w \in W_{U(2,1)(\QQ_p)} \simeq \ZZ/2\ZZ$ the non trivial element and $\widetilde{\chi}^w$ the corresponding character of $T$ ($\widetilde{\chi}^w = \chi(w\cdot w))$. Then the unique admissible refinement of $\pi^n_p(\chi)$ is given by $\widetilde{\chi}^w$, i.e. $\pi^n_p(\chi)_N = \widetilde{\chi}^w \delta_B^{1/2}$.
\end{propen}

\begin{proof}
Denote for the proof $G = U(2,1)(\QQ_p)$. According to Rogawski we have $\ind-n_B^G(\widetilde \chi)^{ss} = \{ \pi^n_p,\pi^2_p\}$ and 
$(\ind-n_B^G(\widetilde \chi))_N^{ss} = \{ \widetilde \chi \delta_B^{1/2}, \widetilde\chi^w\delta_B^{1/2}\}$ by Bernstein-Zelevinski's geometric lemma.
Following \cite{BC2}, denote for $\sigma \in W_G$ $S(\widetilde{\chi}^\sigma)$ the unique subrepresentation of $\ind-n_B^G(\widetilde \chi^\sigma)$ (this induction is non split by \cite{Keys} for example). It is also the Jordan-Hölder factor that contains $\widetilde\chi^\sigma\delta_B^{1/2}$ inside its semi-simplified Jacquet functor.
Thus $S(\widetilde \chi) = \pi^2_p$ or $S(\widetilde \chi) = \pi^n_p$. And as changing $\widetilde\chi$ by $\widetilde\chi^w$ exchanges the subrepresentation and the quotient in the induced representation, $S(\widetilde\chi) \neq S(\widetilde\chi^w)$. So the proposition is equivalent to $\pi^2_p = S(\widetilde\chi)$.
Let us remark that it is announced in \cite{Rog}, as $\pi^n_p$ is said to be the Langlands quotient, but let us give an argument for that fact.
We can use Casselman's criterion for $\pi^2_p$ (\cite{Casselman} Theorem 4.4.6). For $A = T^{split} = \mathbb G_m \subset B$, 
\[ A^-\backslash A(\mathcal O)A_\delta = \{\Diag(x,1,x^{-1}) : x \in \ZZ_p \backslash \ZZ_p^\times = p\ZZ_p\},\]
and thus,
\[ \forall x \in p\ZZ_p, |\widetilde{\chi}\delta_B^{1/2}(\Diag(x,1,x^{-1}))| = |\chi(x)| = |\chi_0(x)||x|^{1/2} < 1,\]
as $\chi_0$ is unitary, and thus $\widetilde \chi\delta_B^{1/2}$ is an exposant of $r_B^G(\pi^2)$ and $\pi^2_p \subset \ind-n_B^G(\widetilde \chi)$ i.e.
\[ (\pi^n_p)_N = \widetilde \chi^w\delta_B^{1/2}.\]
\end{proof}

When $p$ is split, the calculation is done in \cite{BC1} and we get the following up to identifying an unramified character of $T(\QQ_p) \simeq \mathbb (\QQ_p^\times)^3$,
\[
\psi : 
\begin{array}{ccc}
 T(\QQ_p) & \fleche  &  \CC \\

\left(
\begin{array}{ccc}
x_1  &   &   \\
  & x_2  &   \\
  &   & x_3  
\end{array}
\right)
  & \longmapsto& \psi_1(x_1)\psi_2(x_2)\psi_3(x_3)   
\end{array}
\]
with the triple $(\psi_1(p),\psi_2(p),\psi_3(p))$.

\begin{propen}[Bellaïche-Chenevier \cite{BC1},\cite{BC2}]
If $p = \overline v v$, the accessible refinements of $\pi^n_p(\chi)$ are given (with identification with $\GL_3(\QQ_p)$ using $v$) by 
\begin{itemize}
\item $\sigma = 1, (1,\chi_{v}^\perp(p),\chi_v(p))$
\item $\sigma = (3,2), (\chi_{v}^\perp(p),1,\chi_v(p))$
\item $\sigma = (3,2,1), (\chi_{v}^\perp(p),\chi_v(p),1).$
\end{itemize}
\end{propen}

\begin{proof}
Indeed, the Langlands class associated to $\pi^n_p(\chi)$ is $(\chi_v^\perp(p),1,\chi_v(p))$ which corresponds, up to a twist of the central character by $(\chi_v^{\perp})^{-1}$, to the class $(1, (\chi_v^{\perp}(p))^{-1},|p|)$ which in turn is associated by Satake (up to twist by $\mu^{-1}|.|^{1/2}$) to the unramified induction studied in \cite{BC2} Lemma 8.2.1, $n=1,m=3$ and $\pi = \chi_0^c = \chi_0^{-1}$ (which satisfies the hypothesis of \cite{BC2} 6.9.1). Thus $L(\pi_p|.|^{1/2}) = (\chi^\perp(p))^{-1}$. The refinements are then given by the Lemma 8.2.1. of \cite{BC2}.
\end{proof}

\subsection{Coherent cohomology}

In order to associate to the automorphic representation $\pi^n(\chi)$ a point in the eigenvariety constructed in section \ref{sectHecke}, we need also to show that $\pi^n(\chi)$ appears in the global sections (over $X$) of a coherent automorphic sheaf. The full calculation is made in Appendix \ref{AppB}. Here we give an alternative proof in the case $a >1$, which corresponds to a regular weight, as the case $a=1$ will correspond to a singular weight (and for $a=1$, $\pi^2$ is a non-holomorphic limit of discrete series). Thus suppose $a > 1$.
According to Rogawski \cite{Rogbook} Proposition 15.2.1, the (regular) parameter $\varphi = \varphi(a,b,c) = \varphi(a,a-1,0)$ (see \cite{Rogbook} p176, $\chi$ corresponding to $\chi_\varphi^-$) we already know that,
\[ H^i(\mathfrak g,K,\pi^n_\infty(\chi)\otimes\mathcal F_\varphi^\vee) = \left\{
\begin{array}{cc}
 \CC & \text{if } i = 1,3   \\
  0& \text{otherwise}   
\end{array}
\right.
\]
for $\mathcal F_\varphi$ the representation of $U(2,1)(\RR)$ of highest weight $(a-1,a-1,1)$,
and 
\[ H^i(\mathfrak g,K,\pi^2_\infty(\chi)\otimes\mathcal F_\varphi^\vee) = \left\{
\begin{array}{cc}
 \CC & \text{if } i = 2   \\
  0& \text{otherwise.}   
\end{array}
\right.
\]
In particular the system of Hecke Eigenvalues of $\pi^n(\chi)$ appears in the first intersection cohomology group of a local system associated to $\mathcal F_\varphi^\vee$ (thus coming from a representation of the group $G$), $IH^1(X^{BS},\mathcal F_\varphi^\vee)$, where $X^{BS}$ is the Borel-Serre compactification, or also by a theorem of Borel in $H^1_{dR,c}(Y,(V_\varphi,\nabla))$ the de Rham cohomology with compact support (where $V_\varphi = \mathcal F_\varphi^\vee \otimes \mathcal O_X$ is the associated vector bundle with connection, see \cite{HarCor} 1.4). If our Picard surface were compact, then $IH^1(X^{BS},\mathcal F_\varphi^\vee) = H^1_{et}(X,\mathcal F_\varphi^\vee)$ is just etale cohomology and as $\mathcal F_\phi^\vee$ comes from a representation of our group, using the Hodge-decomposition for $H^1_{et}(X,\mathcal F_\varphi^\vee)$ or for the de Rham cohomology (see \cite{Fal83} or \cite{MM}), we know that there exists a coherent automorphic sheaf $V_\varphi = \mathcal F_\varphi^\vee \otimes \mathcal O_X$ such that,
\[ H_{et}^1(X,\mathcal F_\varphi^\vee) = H^1(X,\mathcal F_\varphi^\vee\otimes \mathcal O_X) \oplus H^0(X,\mathcal F_{\varphi}^\vee \otimes \Omega^1_X).\]

More generally, if $X$ is not compact (which is the case here) we still have a BGG-decomposition for the de Rham cohomology with compact support, see \cite{LanICCM} (2.4),
\[ H^1_{dR,c}(Y,(V_\varphi,\nabla)) = H^1(X^{},\omega^{(1-a,1-a,1)}(-D)) \oplus H^0(X^{},\omega^{(0,1-a,a)}(-D)).\]
We thus need to show that the system of Hecke eigenvalues appears in the last factor.
But, denote $\overline I$ the opposite induced representation $\ind-n_B^{U(2,1)(\RR)}(\chi_\infty^w)$, so that $\pi^n_\infty$ is the subrepresentation of $\overline I$ and $\pi^2$ its quotient. Writing the long exact sequence of $(\mathfrak g,K)$-cohomology associated to 
\[ 0 \fleche \pi^n_\infty \fleche \overline{I} \fleche \pi^2_\infty \fleche 0,\]
we get that $H^1(\mathfrak g,K,\pi^n_\infty \otimes \mathcal F_\varphi^\vee) = H^1(\mathfrak g,K,\overline I\otimes \mathcal F_\varphi^\vee)$.
Using Hodge decomposition for this, we get that,
\[ \Hom_K(\mathfrak p^+\otimes \mathcal F_\varphi,\pi^n_\infty) = \Hom_K(\mathfrak p^+ \otimes \mathcal F_\varphi,\overline I),\]
and using Frobenius reciprocity we can calculate the last term as,
\[ \Hom_{T\cap K}(\mathfrak p^+\otimes \mathcal F_\varphi,(\chi_\infty^w)_{T\cap K}),\]
and as we know $\mathcal F_\varphi$, we can calculate its restriction to $T\cap K$, we get,
\[ (\mathcal F_\varphi)_{T\cap K} = t^ae^{a-1} \oplus \dots \oplus t^{2a-2}e,\]
where,
\[ t^ke^l : 
\left(
\begin{array}{ccc}
t  &   &   \\
  & e  &   \\
  &   & t  
\end{array}
\right) \in T\cap K = U(1)\times U(1) \longmapsto t^ke^l.
\]
We can also explicitely calculate that by conjugacy, $T \cap K$ acts on $\mathfrak p^+$ by $1 \oplus te^{-1}$ and on $\mathfrak p^-$ by $1 \oplus t^{-1}e$. 
\begin{remaen}
\label{remchangecomplexstructure}
This is because of our choice of $h$. If we change $h$ by its conjugate, then the action on $\mathfrak p^+, \mathfrak p^-$ would have been exchanged, and $\pi^n(\chi)$ would be anti-holomorphic (but we could have used $\chi^c$ instead of $\chi$ in this case, $\pi^n(\chi^c)$ would have been holomorphic).
\end{remaen}

As $\chi^w_{T\cap K} = t^{2a-1}$, we get that,
\[\Hom_{K}(\mathfrak p^+\otimes\mathcal F_\varphi, \overline I) = \CC \quad \text{and} \quad\Hom_{K}(\mathfrak p^-\otimes\mathcal F_\varphi, \overline I) = \{0\}.\]

\begin{remaen}
Changing $\chi$ by ${\chi}^c$ inverts the previous result, as predicted by the Hodge structure, so we could have argued without explicitely calculating these spaces.
\end{remaen}

\begin{propen}
\label{propHodge}
If $a > 1$, the Hecke eigensystem corresponding to $\pi^n(\chi)$ appears in $H^0(X,F_{s\cdot(\varphi^\vee)} \otimes \mathcal O_X(-D)) = H^0(X,\omega^{(0,1-a,a)}(-D))$. More generally,  for $a \geq 1$, the Hecke eigenvalues of $\pi^n(\chi)$ appear in the global section of the coherent sheaf $\omega^{(0,1-a,a)}(-D)$.
\end{propen}

\begin{proof}
When $a > 1$ we know that the Hecke eigensystem corresponding to $\pi^n(\chi)$ appears in the de Rham cohomology associated to $\varphi$. If $Y$ were compact, the proposition would just be Matsushima's formula and the Hodge decomposition (\cite{Yosh} Theorem 4.7) of the first $(\mathfrak g,K)$-cohomology group we calculated above. 
More generally, the previous calculation still shows that the class corresponding to $\pi^n(\chi)$ can't be represented by classes in $H^1(X,\omega^{(1-a,1-a,1)})$ (whose cohomology is calculated by $(\mathfrak p^-,K)$-cohomology) thus by the BGG resolution, it appears in $H^0(X,\omega^{(0,1-a,a)}(-D))$.

For the general case, this is Appendix \ref{AppB} (which proves that the Hecke eigensystem appears in the 0-th coherent cohomology group).

\end{proof}

\subsection{Transfert to $GU(2,1)$}
\label{secttransfert}

From now on, denote by $G = GU(2,1)$ the algebraic group over $\QQ$ of unitary similitudes (relatively to $(E^3,J)$). It is endowed with a morphism $\nu$, and there is an exact sequence,
\[ 0 \fleche G_1 \fleche G \overset{\nu}{\fleche} \mathbb G_m,\]
where $G_1 = U(2,1)$ is the unitary group of $(E^3,J)$.

Let $T = \Res_{E/\QQ}\mathbb G_m$ be the center of $G$, $Nm : T \fleche \mathbb G_m$ the norm morphism, and $T^1$ its Kernel; the center of $G_1$. We have the exact sequence,
\[ 1 \fleche T^1 \fleche T \times G_1 \fleche G,\]
where the first map is given by $\lambda \mapsto (\lambda,\lambda^{-1})$.

Let $\pi_1$ be an automorphic (resp. a smooth admissible local) representation of $G_1(\mathbb A_\QQ)$ (resp. of $G_1(\QQ_p)$ or $G_1(\RR)$), of central character $\chi_1$ of $T^1$. Let $\chi$ be a character of $T$ (local or global) that extends $\chi_1$, we can thus look at the representation,
\[ (z,g) \in T \times G_1 \longmapsto \chi(z)\pi_1(g),\]
of $T\times G_1$. We can check that it factors through the action of $T^1$ and gives a representation of a subgroup of $G$.

\begin{propen}
\label{transfert}
The automorphic representation $\pi^n$ of $U(2,1)$ given in proposition \ref{proprog} has central character $\omega$ equal to
the restriction of $\chi$ to $E^1$. We can extend $\omega$ as an algebraic Hecke character $\widetilde{\omega}$ of $T$ by the algebraic character
$\widetilde\omega = N^{-1} \chi$, where $N$ is the norm of $E$. Thus, there exists an automorphic representation $\widetilde{\pi^n}$ of $G$ such that for $\ell$ a prime, unramified for $\chi_0$, $(\pi^n_\ell)^{K_\ell} = (\widetilde{\pi^n_\ell})^{K_\ell}$ (where $K_\ell \subset G(\QQ_\ell)$ is the hyperspecial – respectively special if $\ell$ ramifies in $E$– subgroup) and the Galois representation associated to $\widetilde{\pi}^n$ by \cite{BR} Theorem 1.9.1 (or  \cite{Ski}) is (with the normalisation of \cite{Ski}),
\[ (1 \oplus\chi \oplus \chi^\perp)\overline{\chi}(-2) = (\overline \chi \oplus \omega_{cycl}\oplus 1)(-2).\]
Moreover $(\pi^n_p)^I = (\widetilde{\pi^n_p})^I$.
\end{propen}

\begin{proof}
To calculate $\omega$, we only need to look at $\pi^n_p$ for all place $p$, and we can use that $\pi^n_p = \ind-n^{GL_3(\QQ_p)}_P(\chi_0)$ for split $p$'s, and $\pi^n_p \subset \ind-n_B^{U(3)(\QQ_p)}(\widetilde{\chi}^w)$ for $p = \infty$ and inert or ramified $p$'s.

The character $\widetilde\omega$ extends $\omega$. Once we have extended the central character of $\pi^n$, the existence of a $\widetilde \pi^n$ is unique and assured by \cite{CHT} Proposition 1.1.4 (as $2+1 = 3$ is odd). More precisely,
\[ \widetilde\pi^n(zg_\QQ g_1) = \widetilde\omega(z)\pi^n(g_1),\]
where $g = z g_\QQ g_1$ is written following the decomposition $G(\mathbb A) = T(\mathbb A)G(\QQ)G_1(\mathbb A)$.

Denote by $V_p$ the space of $\pi^n_p$ (and thus of $\widetilde \pi^n_p$). And $I \subset GU(2,1)(\QQ_p)$ the Iwahori subgroup, and $I_1$ its intersection with $U(2,1)(\QQ_p)$. As if $M \in I$, then $M \equiv B \pmod p$, up to multiplying by an element of $T \in T(\mathcal O)$, suppose that $TM \equiv U \pmod p$. In this case $c(TM) \equiv 1 \pmod p$, thus, as $p$ is unramified in $E$, there exist $T' \in T(\ZZ_p)$ such that $c(T'TM) =1$ and thus $M = T^{-1}(T'TM) \in T(\ZZ_p)I_1$. Thus, as we can write,
\[
\left(
\begin{array}{ccc}
a  &   &   \\
  & e  &   \\
  &   & N(e)\overline{a}^{-1}  
\end{array}
\right) = \left(
\begin{array}{ccc}
ae^{-1}  &   &   \\
  & 1  &   \\
  &   & \overline{e}\overline{a}^{-1}  
\end{array}
\right)  \left(
\begin{array}{ccc}
e  &   &   \\
  & e  &   \\
  &   & e  
\end{array}
\right) \in I_1 T(\ZZ_p),
\]
we get that,
\[ V_p^I = \{ z \in V_p^{I_1} : \forall \lambda \in T(\QQ_p) \cap I, \widetilde\omega_p(\lambda)z =z \},\]
but as $T(\QQ_p)\cap I = \mathcal O_{E_p^\times}Id$ and $\widetilde w_p$ is unramified, $V_p^I = V_p^{I_1}$. The assertion for $K_\ell$ follows the same lines and is easier.
\end{proof}

\begin{remaen}
We could have lifted the central character of $\pi^n$ simply by $\chi$, in which case the resulting representation would have been a twist of the previous one, but as we only used three variables on the weight space, which means that we don't allow families which are twists by power of the norm of the central character, only one choice of the lift of the central character gives a point in our eigenvariety. We can check that the Hecke eigenvalues of $\widetilde{\pi^n}$ appears in $H^0(X,\omega^\kappa)$, with 
\[ \kappa =   (0,1-a,a).\]
How can we find the power of the norm and the coherent weight ? First, as Hodge-Tate and coherent weights vary continuously on $\mathcal E$, and $\widetilde{\pi^n(\chi)}$ appears as a classical form of $\mathcal E$ (proposition \ref{propH0}), according to proposition \ref{proSki} and proposition \ref{thrT}, the \textit{polarised} Galois representation associated to $\widetilde{\pi^n}(\chi)$ is,
\[ 1 \oplus \chi \oplus \chi^\perp,\]
thus $(k_1,k_2-1,1-k_3) = (-a,0,1-a)$ up to order. This let us 6 possibilities for $\kappa$ :
\begin{enumerate}
\item $(0,1-a,a)$
\item $(0,2-a,a+1)$
\item $(1-a,1-a,1)$
\item $(-a,1-a,a)$
\item $(-a,2-a,1+a)$
\item $(1-a,1,1+a)$
\end{enumerate}
but as for classical points (as $\widetilde{\pi^n}(\chi)$) $k_1 \geq k_2$, and $a \geq 1$, this eliminates the three last possibilities (and the second when $a=1$).
But then, we know that the lowest $K_\infty$-type for $\pi^n(\chi)$ is of dimension $a$ by restriction to $U(2)$ and the calculation of Appendix \ref{AppB}, proposition \ref{propH0}, which makes only the first coherent weight possible when $a >1$. When $a=1$, the first and third weights are the same. 
Another possibility is also to find the infinitesimal character of $\pi^n(\chi)$ (using for example \cite{Knapp} Proposition 8.22), and that $\eta = (-k_3,k_1,k_2)$ is the highest weight character of $V_{\lambda + \rho_n-\rho_c}^\vee$ in the notations of \cite{Gold} (paying attention to the dual). Finally, if $a > 1$, then the BGG decomposition tells us which weight $\kappa$ has to be.
Then to find the corresponding power of the norm, note that $|\kappa|$ must be equal to the opposite of the power of the norm of the central character of $\widetilde{\pi^n(\chi)}$ by the calculation before proposition \ref{prop27} and conventions on weights (see section \ref{normweight} about $(c_0,c_0')$).

\end{remaen}

\subsection{Refinement of representations of $GU(2,1)(\QQ_p)$}
\label{sectrefin}
Let $G = GU(2,1)(\QQ_p)$. Consider in $C_c(I\backslash G/I,\ZZ[1/p])$ the double classes,
\[ 
U_p^c = \left(
\begin{array}{ccc}
p^2  &   &   \\
  & p  &   \\
  &   & 1  
\end{array}
\right)  \quad \text{and} \quad S_p^c = \left(
\begin{array}{ccc}
p  &   &   \\
  & p  &   \\
  &   & p  
\end{array}
\right). 
\]
The caracteristic functions of $S_p^c$ and $U_p^c$ are invertible in $C_c(I\backslash G/I,\ZZ[1/p])$
and denote by $\mathcal A(p)$ the sub-algebra generated by the characteristic functions of $U_p^c,S_p^c$ and their inverses.

\begin{propen}
For $\pi$ a smooth complex representation of $G$, we have a natural $\CC[\mathcal A(p)]$-isomorphism,
\[ \pi^I \fleche (\pi_N)^{T^0} \otimes \delta_B^{-1}.\]
\end{propen}

Let $\pi$ be a smooth admissible representation of $G$, such that $\pi$ is a subquotient of the (normalised) induction of an unramified character $\psi$ of the torus $T$ of $G$.
For example this is the case if $\pi$ is unramified, or if $\pi^I \neq \{0\}$ (by the previous equality and adjonction beetwen Jacquet functor and induction for example).
\begin{definen}
Following \cite{BC2}, an accessible refinement of $\pi$ is a $\sigma \in W$ such that $\psi^\sigma\delta_B^{1/2}$ is a subrepresentation of $\pi_N^{T^0}$ (equivalently if $\psi^\sigma\delta_B^{-1/2}$ appears in $\pi^I$). 
\end{definen}

Another way to see it is that a refinement is an ordering of the eigenvalues of the Frobenius of $LL(\pi)$, the Weil representation associated by local Langlands to $\pi$, and it is accessible if it appears in the previous sense in $\pi^I$ (or $\pi_N^{T^0}$).

For $GU(2,1)$ when $p$ is split, $GU(2,1)(\QQ_p) \simeq \GL_3(\QQ_p)\times \QQ_p^\times$ and $\psi$ is an unramified character of $\QQ_p^4$. The local Langlands representation associated to $\pi = \pi_1 \otimes \psi_4$ in this case is $LL(\pi_1)\otimes \overline{\psi_4}$ which has eigenvalues $(\psi_1(p)\overline\psi_4(p), \psi_2(p)\overline\psi_4(p),\psi_3(p)\overline\psi_4(p))$ and an ordering of these eigenvalues is given by an element of $\mathfrak S_3 = W_{\GL_3} = W_{\GL_3\times\GL_1}$. Of course, a priori not all refinements are accessible ($\pi^n_p$ will be an example).

When $p$ is inert, $W_G \simeq \ZZ/2\ZZ$, and a character of $T = (\Res_{\QQ_{p^2}/\QQ_p}\mathbb G_m)^2$ is given by two characters $(\chi_1,\chi_2)$, by,
\[
\left(
\begin{array}{ccc}
a  &   &   \\
  &  e &   \\
  &   &   N(e)\overline{a}^{-1}
\end{array}
\right), a,e \in \QQ_{p^2}^\times \longmapsto \chi_1(a)\chi_2(e).
\]
The non trivial element $w \in W_G$ acts on the character by $w \cdot (\chi_1,\chi_2) = (\chi_1^{\perp},\chi_2(\chi_1\circ N))$. Thus a refinement in this case is simply given by $1$ 
or $w$.

\begin{remaen}
In terms of Galois representation, the base change morphism from $GU(2,1)$ to $\GL_3\times \GL_1$ sends the (unramified) Satake parameter $\chi_1,\chi_2$ (if $\chi_2$ is unramified, it is trivial on $E^1$) to the parameter 
$((\chi_1,1,\overline\chi_1^{-1}),\chi_1\chi_2)$ (see \cite{BR} Theorem 1.9.1 or \cite{Ski} section 2), whose semi-simple class in $\GL_3$ 
associated by Local Langlands has Frobenius given by 
\[
\left(
\begin{array}{ccc}
\chi_2(p) &   &   \\
  & \chi_1(p)\chi_2(p)  &   \\
  &   &\chi_2(p)) \chi_1^2(p) 
\end{array}
\right)
\]

In the inert case, say $\sigma \in W_G$ is a refinement, then the action of $U_p$ on the $\sigma$-part of $\pi_N^{T^0}$ is given by $\chi_1^\sigma(p)^2\chi_2^\sigma(p)$, the action of $s$ is given by $\chi_1^\sigma(p)\chi_2^\sigma(p)$.
In particular, the action of $\mathcal A(p)$ (through $T/T^0$), and actually of $U_p$ or $u_1 = U_pS_p^{-1}$, on $\pi^I$ determines the refinement.

This is also true (and easier) if $p$ splits.

As we normalized our Galois representation $\rho_\pi$ so that they are polarized, i.e. forgetting the central character, the previous class does not directly relate to the Frobenius eigenvalues of $\rho_\pi$ but rather of the one of $\rho_{\pi,Ski}$. But as the link between both only differ through the central character of $\pi$, it is straightforward that the Frobenius eigenvalues of (a crystalline) $\rho_\pi$ are given by $(\psi_1,1,\psi_1^\perp)$, when $p$ is inert, and $\psi_1$ is given (if unramified) by the action of the Iwahori-Hecke double class
\[
\left(
\begin{array}{ccc}
p &   &   \\
  & 1 &   \\
  &   & p^{-1}
\end{array}
\right)
\]
which corresponds to $U_pS_p^{-1}$ (see next subsection). In the split case, an unramified character of the torus of $\GL_3\times \GL_1$ gives Frobenius eigenvalues $(p\psi_1(p)\psi_4(p),p\psi_2(p)\psi_4(p),p\psi_3(p)\psi_4(p))$ for (crystalline) $\rho_{\pi,Ski}$ and $(\psi_1(p),\psi_2(p),\psi_3(p))$ for $\rho_\pi$, which relates to operators $U_{i-1}/U_3$ (see next subsection). 
\end{remaen}

Thus, using the previous definition of refinement, local global compatibility at $p$, we can associate to $\Pi = \pi_\infty \otimes \bigotimes_\ell \pi_\ell$ an algebraic regular cuspidal automorphic representation of $GU(2,1)$ of level $K^{Np}I$ a representation $\rho_{\pi,p}$ together with an (accessible) ordering of its crystalline-Frobenius eigenvalues for each choice of a character in $\pi_p^{I}$ under $\mathcal A(p)$, such that

\begin{propen}
\label{rafpin}
The automorphic representation $\widetilde{\pi}^n(\chi)$ of $GU(2,1)$ constructed by proposition \ref{transfert} has only one accessible refinement at $p$ if $p$ is inert, it is
given by,
\[ \omega \neq 1 \in W_{G},\]
which correspond to the ordering $((\chi^\perp(p),1,\chi(p)),{\chi(p)})$ or $(1,{\chi(p)},|p|)$.
If $p = \overline{v}v$ is split, there are three accessible refinements, given by,
\begin{itemize}
\item $\sigma = 1, ((1,\chi_{v}^\perp(p),\chi_v(p)),\chi_v(p))$ which corresponds to $(\chi_v(p),1,|p|)$.
\item $\sigma = (3,2), ((\chi_{v}^\perp(p),1,\chi_v(p)),\chi_v(p))$ which corresponds to $(1,\chi_v(p),|p|)$.
\item $\sigma = (3,2,1), ((\chi_{v}^\perp(p),\chi_v(p),1),\chi_v(p))$ which corresponds to $(1,|p|,\chi_v(p))$.
\end{itemize}

We denote $\sigma$ the unique refinement in the inert case, and the refinement denoted $(3,2)$ in the split case.
\end{propen}

\begin{proof}
The action of $u_1$ on $\pi^n(\chi)_p$ as been calculated in a previous section.
\end{proof}

\subsection{Modular and Classical Hecke Operators}
\label{ModClassHecke}
In order to understand how the refinements vary on the Eigenvariety, we need to explicite the link between Hecke operators (at $p$) constructed in section \ref{sectHecke} and 
classical Hecke acting on automorphic forms, as above. Here we work at Iwahori level at $p$ and identify matrices with the corresponding Iwahori double classes. If $p$ is inert in 
$E$, the Atkin-Lehner algebra we consider at $p$ is generated by the two (so-called classical) operators $U_p^c$ and $S_p^c$ described above. If $p$ is split in $E$, we consider 
the Atkin-Lehner algebra $\mathcal A(p)$ of $GL_3(\QQ_p) \times \QQ_p^\times$ (see \cite{BC2} section 6.4.1.), it is generated by the Hecke operators, up to identification of $E \otimes \QQ_p \overset{i_v \times i_{\overline v}}{\simeq} \QQ_p \times \QQ_p$,
\[ (pI_3,I_3),\quad 
(\left(
\begin{array}{ccc}
p  &   &   \\
  & p  &   \\
  &   & 1  
\end{array}
\right),(\left(
\begin{array}{ccc}
p  &   &   \\
  & 1  &   \\
  &   & 1  
\end{array}
\right))
,\quad (\left(
\begin{array}{ccc}
p  &   &   \\
  & 1  &   \\
  &   & 1  
\end{array}
\right),\left(
\begin{array}{ccc}
p  &   &   \\
  & p  &   \\
  &   & 1  
\end{array}
\right)),\quad(I_3,pI_3),\]
that we denote respectively $U_0^c,U_1^c,U_2^c$ and $U_3^c$ ($c$ stands for classical in the sense "not normalized").
If we use $i_v$ to identify $GU(2,1)(\QQ_p)$ with $\GL_3\times GL_1(\QQ_p)$ then this operators are identified respectively with,
\[ (pI_3,p),\quad 
(\left(
\begin{array}{ccc}
p  &   &   \\
  & p  &   \\
  &   & 1  
\end{array}
\right),p)
,\quad (\left(
\begin{array}{ccc}
p  &   &   \\
  & 1  &   \\
  &   & 1  
\end{array}
\right),p),\quad(I_3,p),\]

In section \ref{sectHecke} we defined Hecke operators modularly, $U_p$ and $S_p$ in the inert case, and Brasca or Bijakowski defined $U_0,U_1,U_2,U_3$ in the split case (see remark in subsection \ref{Heckenormsplit}). These Hecke operators have been normalized and correspond to the above Iwahori double classes, so that we have the following result.

Let $\Pi = \pi_\infty \otimes \bigotimes_\ell \pi_\ell$ be an algebraic, regular, cuspidal automorphic representation of $GU(2,1)$ of level $K^{Np}I$ whose Hecke eigenvalues appear in the global sections of a coherent automorphic sheaf (of weight $\kappa$) and  $ f\in \Pi \cap H^0(X_I,\omega^\kappa)$ an eigenform for $\mathcal H = \mathcal H^{Np}\otimes\mathcal A(p)$, such that, if $p$ is inert,
\[ U_p f = p^{-k_2}U_p^c f \quad \text{and} \quad S_p f = p^{-k_1-k_2-k_3} S_p^c f,\]
and if $p$ splits, %\todo{verifier les doubles classes...}
\[ U_0 f = p^{-k_3}U_0^c f \quad U_1 f = U_1^c f \quad U_2 f = p^{-k_2}U_2^cf \quad U_3 f = p^{-k_1-k_2} U_3^c f,\]
where the action of the double classes $U_p^c$, $S_p^c$ and $U_i^c$ is given by convolution on $\pi_p$,

\begin{propen}
\label{prophecke}
Suppose $p$ is inert, $f$ is a classical automorphic form of classical weight $\kappa = (k_1,k_2,k_3)$ of Iwahori level at p (i.e. $f \in H^0(X_{Iw},\omega^{\kappa})$), eigen for the action of 
$\mathcal H\otimes \mathcal A(p)$, and denote $\lambda,\mu$ the eigenvalues of $f$ for $U_p,S_p$ respectively. 

Let ${\Pi}$ be a irreducible factor of the associated automorphic representation (generated by $\Phi_f$).
Then ${\Pi}_p^I \neq \{0\}$ and thus the algebra $\mathcal A(p)$ acts on ${\Pi}_p^I$ with $U_p^c$ of eigenvalue $p^{k_2}\lambda$ and $S_p^c$ of eigenvalue 
$p^{|k|}\mu$.

\end{propen}

\begin{proof}

To prove the statement, we remark that the association $f \mapsto {\Phi_f}$ is Hecke equivariant for the classical Hecke operators $U_p^{c},S_p^{c}$ acting 
on $f$.
But we defined the Hecke operators $U_p,S_p$ geometrically by $U_p = p^{-k_2}U_p^{c}$ and $S_p = p^{-|k|}S_p^{c}$ to make them vary p-adically. Thus we get the result.
\end{proof}

Using the previous refinements for representations of $GU(2,1)$, we can prove the following result on density of crystalline points on the Eigenvariety $\mathcal E$ of theorem \ref{thrHecke},

\begin{propen}
\label{cryspoint}
Suppose $p$ is inert. Let $x \in \mathcal E(F)$. There exists a neighborhood $V$ of $x$ and a constant $C >0$ such that for all classical points $y \in V$, 

if $|w_2(y) + w_3(y)| > C$,
then $\rho_y$ is cristalline and of Hodge-Tate weights $(1+w_3(y),-1-w_1(y),-w_2(y))$. 

In particular crystalline points are dense in $\mathcal E$ by classicity propostion \ref{thrclassfaisc} (as we can also assume $w_1(y) - w_2(y) > C$) and theorem \ref{thrclassbij} (as we can moreover assume $-w_1(y) - w_3(y) > C$).
\end{propen}

\begin{proof} 
Denote by $F_1,F_2$ the two invertible functions of $\mathcal E$ given by the eigenvalues under
\[
U_p = \left(
\begin{array}{ccc}
p^2  &   &   \\
  & p  &   \\
  &   & 1  
\end{array}
\right) \quad \text{and} \quad S_p = \left(
\begin{array}{ccc}
p  &   &   \\
  & p  &   \\
  &   & p  
\end{array}
\right).
\] The valuations of $F_1,F_2$ are locally constant on $\mathcal E$, and thus there exists $V$ a neighbourhood of $x$ where these valuations are constant.
As $y$ corresponds to $f$ a classical form of ($p$-adic) weight $w(y)$ and level $K$ proper under $\mathcal H \otimes \mathcal A(p)$, we can look at $\Pi$ an irreducible component of 
the representation generated by $\Phi_f$, which is thus algebraic, regular, and its associated representation $\rho_y$ doesn't depend on $\Pi$ as it only depends on the eigenvalues 
of $\mathcal H$ on $f$. As $\Pi_p$ the $p$-th component of $\Pi$ is generated by its $I$-invariants, $\Pi_p$ is a subquotient of the induction $\ind-n_B^G(\QQ_p)(\psi)$ for 
some unramified character $\psi$ (prop 6.4.3 of \cite{BC2} and the adjontion property of induction).
We need to show that $\Pi_p$ is unramified, but $\Pi_p$ appears as a subquotient of $\ind_B(\psi)$, which has a unique unramified subquotient, it suffices to prove 
that $\ind_B(\psi)$ is irreducible, which happens in particular when $|\psi_1(p)|\neq p^{\pm1}$ when $p$ is inert (cf. the key result of Keys, see \cite{Rogbook} 12.2).

In the inert case, we have that if $w = (-k_2,-k_1,-k_3)$ i.e. $f$ is of automorphic classical weight $(k_1,k_2,k_3)$,
then by proposition \ref{prophecke}
\[ \psi_1^\sigma(p) = p^{-k_1-k_3}F_1(y)/F_2(y),\]
for a certain choice $\sigma\in W_{GU}$ (see subsection \ref{sectrefin} for example), 
but as the valuations of $F_1,F_2$ are constant on $V$, there is a constant $C$ such that if $|k_1+k_3| > C$, $\Pi_p$ is unramified.
Thus, by local-global compatibility at $p$ for $\Pi$ (cf. \cite{Ski} Theorem B), $\rho_y$ is crystalline.
\end{proof}

\begin{remaen}
In the split case, the same proposition is true under the assumption $\delta(w(y)) := \min_i(|w_i(y)-w_{i+1}(y)|) > C$, as the same proof of proposition 8.2 of \cite{BC1}, together with classicity results of \cite{BPS} and \cite{Bra} Proposition 6.6, Theorem 6.7 applies.
\end{remaen}

\subsection{Types at ramified primes for $\chi$}

In order to control the ramification at $\ell | \Cond(\chi)$, Bellaïche and Chenevier introduced a particular type $(K_0,J_0)$, which we can slightly modify to suit our situation.

\begin{propen}
\label{proptypes}
Let $\ell |\Cond(\chi)$ a prime. There exists a compact subgroup $K_\ell$ of $GU(2,1)(\QQ_\ell)$ and a representation $J_\ell$ of $K_\ell$ such that,
\begin{enumerate}
\item $\Hom_{K_\ell}(J_\ell, \widetilde{\pi^n_\ell(\chi)}\otimes(\chi_{0,\ell}\circ \det)) \neq 0$,
\item For all smooth admissible representation $\pi$ of $GU(2,1)(\QQ_\ell)$ such that $\Hom_{K_\ell}(J_\ell,\pi)\neq 0$ and for all place $v | \ell$, there exist four unramified characters
$\phi_1,\phi_2,\phi_3,\phi_4 : E_v^\times \fleche \CC^\times$ such that, the Langlands semi-simple class in $\GL_3 \times \GL_1$ corresponds to,
\[ L(\pi_{E_v}) = (\phi_1 \oplus \phi_2 \oplus \phi_3\chi_0^{-1},\phi_4\chi_0^{-1})\]
or to the (unpolarized) Langlands class in $\GL_3$,
\[ L(\pi_{E_v}) = \phi_1\overline{\phi_4}\chi_0 \oplus \phi_1\overline{\phi_4}\chi_0 \oplus \phi_3\overline{\phi_4}.\]
\end{enumerate}
\end{propen}

\begin{proof}
Let $(K^0_\ell,J_\ell^0)$ be the type defined by Bellaïche and Chenevier in \cite{BC1}.
If $\ell = v_1v_2$ is split, let $K_\ell$ be the subgroup of matrices congruent to
\[
\left(
\begin{array}{ccc}
\star  &  \star & \star  \\
 \star & \star  &  \star \\
  0& 0  &   y
\end{array}
\right),e
\]
modulo $\ell^m$, the $\ell$-adic valuation of $\Cond(\chi)$. Let $J_\ell$ be the representation that sends the matrices in $K_\ell$ to $\chi_{0,v_1}^{-1}(y)\chi_{0,v_1}(\overline e)$. As every matrix in $GU(2,1)(\QQ_\ell) = \GL_3\times \GL_1(\QQ_\ell)$ can be written as $M = \lambda U$ where $U \in U(2,1)(\QQ_\ell) = \GL_3(\QQ_\ell)$ and $\lambda = (1,\lambda)$ is in the center, we can check that $\Hom_{K_\ell}(J_\ell, \widetilde \pi_\ell^n \otimes \chi_{0,v_1}\circ\det^{-1}) \neq 0$.

Now if $\Hom_{K_\ell}(J_\ell,\pi) \neq 0)$ then $\Hom_{K_\ell^0}(J_\ell^0,\pi_{|U(2,1)}) \neq 0$ when restricted to $U(2,1) = \GL_3$ thus by \cite{BC1} we have the conclusion up to a character. But as $K' = (Id\times \GL_1)\cap K_\ell^0 \simeq \mathcal \ZZ_\ell^\times$, $\pi_{|K'} = \chi_{0,v_1}^{-1}\otimes \psi$ where $\psi$ is an unramified character, and thus,
\[ L(\pi_{E,v}) = \phi_1\overline \psi \chi_0\oplus \phi_2\overline \psi \chi_0\oplus \phi_3\overline \psi.\]

If $\ell$ is prime, denote $K_\ell = \mathcal O_{E_\ell}^\times K^0_\ell$ and define $J_\ell$ by,
\[ J_\ell(\lambda M^0) = \chi_\ell(\lambda)^{-2} J_\ell^0(M^0).\]
As $\mathcal O_{E_\ell}^\times \cap K^0 \subset \mathcal O_{E_\ell}^1$, this is well define because the central character of $J^0_\ell$ is up to an unramified character equal to $\chi_\ell^{-2}$. Moreover, $\Hom_{K_\ell}(J_\ell,\widetilde{\pi}^n_\ell(\chi)\otimes(\chi_{0,\ell}\circ \det) \neq 0$ as it is the case for $(K^0_\ell,J^0_\ell)$ by \cite{BC1} which sends back to Blasco \cite{Bla}, and the central character of $\widetilde{\pi^n}_\ell(\chi)$ is equal to $\chi_\ell$ (up to a unramified character).

Conversly, if $\pi$ is a representation of $GU(3)(\QQ_\ell)$ such that $\Hom_{K_\ell}(J_\ell,\pi) \neq 0$ thus $\Hom_{K_\ell^0}(J_\ell^0,\pi_{|U(3)(\QQ_\ell)}) \neq 0$ and thus $L(\pi_{|U(3)}) = \phi_1 \oplus \phi_2 \oplus \phi_3 \chi_0^{-1}$ by \cite{BC1}, and its central character corresponds to $\chi_\ell$ up to an unramified character, and we thus get the result on the Langlands Base change of $\pi$.
\end{proof}

\section{Deformation of $\widetilde{\pi^n}$}

By proposition \ref{proptypes}, we can find for every $\ell | \Cond(\chi)$ $K_\ell$ a subgroup of $GU(2,1)(\QQ_\ell)$ and an irreducible representation $J_\ell$ such that
\[ \Hom_{K_\ell}(J_\ell,\widetilde{\pi^n_\ell}(\chi) \otimes(\chi_{0,\ell} \circ \det)) \neq 0,\]
and for all $\widetilde \pi_\ell$ of type $(K_\ell,J_\ell)$, its base change to $\GL_3(E_v)$, for all $v |\ell$, gives the representation (normalised as in proposition \ref{proSki}, Theorem \ref{thrT})
\[ L(\pi_{\ell,E_v}) = \phi_1\chi_0^{-1} \oplus \phi_2 \oplus \phi_3,\]
where $\phi_j : E_v^\times \fleche \CC^\times$ are unramified characters.

\subsection{Choosing the level}

Up to choosing compatibly places at $\infty$ and embeddings of $\QQ_{p^2}$, we can make $\chi : G_E \fleche \overline{\QQ_p}$, 
the $p$-adic realisation of $\chi$ at $p$, have $\tau$-Hodge-Tate weight $-a = -\frac{k+1}{2}$ and thus $\overline{\chi}$ $\tau$-Hodge-Tate weight $a-1 = \frac{k-1}{2}$.

Let $N = \Cond(\chi)$, suppose $p \neq 2$ if $p$ is inert, $p \not| N$, and is unramified in $E$. Define $K_f = \prod_\ell K_\ell$ by,
\begin{enumerate}
\item If $\ell$ is prime to $pN$, $K_\ell$ is the maximal compact subgroup defined previously such that $\widetilde \pi^n$ as invariants by $K_\ell$ (hyperspecial at unramified $\ell$, very special otherwise)
\item If $\ell = p$, $K_p$ is the Iwahori subgroup of $GU(2,1)(\QQ_p)$.
\item If $\ell | N$, $K_\ell$ is the type as defined before.
\end{enumerate}
We then set,
\[ J = \bigotimes_{\ell | N} J_\ell \otimes (\chi_{0,\ell}\circ\det),\]
as representation of $K_f$.

By construction of $K_f$, 
, 
there is 
$\phi \in \widetilde{\pi^n}(\chi)^{K_f}$, 
an automorphic eigenform for $\mathcal H^{Np}$ and of character under $\mathcal A(p)$ corresponding to the refinement $\sigma$ of proposition \ref{rafpin} and which is associated a classical Picard modular form $f \in H^0(X_I,\omega^\kappa)$ (by proposition \ref{propH0} or proposition \ref{propHodge} if $a >1$) which is an eigenform for $\mathcal A(p)$, whose eigenvalues for $\mathcal A(p)$ corresponds to the 
refinement $\sigma$ too (with the normalisation explained in proposition \ref{prophecke}), and $\kappa = (0,1-a,a).$%(\frac{k-1}{2},\frac{k-1}{2},1)$.

Thus, setting $w_0 = (a-1,0,-a)$ (corresponding to automorphic weight $(0,1-a,a)$)%(\frac{1-k}{2},\frac{1-k}{2},-1)$
, to $f$ is associated a point $x_0 \in \mathcal E$ such that $w(x_0) = w_0$ and 
$\rho_{x,Ski}^{ss} = \overline{\chi}\eps^{-2}(1 \oplus \chi \oplus \chi^\perp)$,
and with normalisation of proposition \ref{proSki},
\[\rho_{x}^{ss} = 1 \oplus \chi \oplus \chi^\perp,\]which is of (ordered) Hodge-Tate weights $(1-a,-a,0)$.

\subsection{A family passing through $f$}

As we have normalized the pseudocharacter $T$ of proposition \ref{thrT} in order to have the "right" representation at $x_0$ (corresponding to $1\oplus \chi \oplus \chi^\perp$),
the map $w$ from the eigenvariety gives the $p$-adic (automorphic) weight, and not the classical automorphic weight nor the Hodge-Tate weights of $T$, thus we will normalize this map accordingly.

The $\tau$-Hodge-Tate weights of $1 \oplus \chi \oplus \chi^\perp$ are given by $(1-a,-a,0) =: h_0$. 
Let $F/\QQ_p$ be a finite extension such that $f$ is defined over $F$.
\begin{propen}
If $p$ is inert, there exists,
\begin{enumerate}
\item A dimension 1 regular integral affinoid $Y$ over $F$, and $y_0 \in Y(F)$,
\item a semi-simple continuous representation,
\[ \rho_{K(Y)} : \Gal(\overline E/E)_{Np} \fleche \GL_3(K(Y)),\]
satisfying $\rho^\perp_{K(Y)} \simeq \rho_{K(Y)},$ the property (ABS) of \cite{BC1}, and $\tr(\rho_{K(Y)})(\Gal(\overline E/E)) \subset \mathcal O_Y$,
\item A $F$-morphism, $h = (h_1,h_2,h_3) : Y \fleche \mathbb A^3$ such that $h_3 = 0$, $h(y_0) =h_0$.
\item A subset $Z \subset Y(F)$ such that $h(Z) \subset h_0  + (p-1)(p+1)^2\ZZ^{3}_{dom}$ (i.e. the weight are regular)
\item A function $F_1$ in $\mathcal O(Y)^\times$ of constant valuation.
\end{enumerate}
such that,
\begin{enumerate}
\item For every affinoid $\Omega$ containing $y_0$, $\Omega \cap Z$ is Zariski dense in $\Omega$.
\item For all $z \in Z \cup \{y_0\}$ $\rho_z^{ss}$ is the Galois representation associated to a cuspidal (algebraic) automorphic representation $\Pi$ of $GU(2,1)$ such that
\[ Hom_{K_f}(J,\Pi) \neq 0.\]
\item \[\rho_{y_0}^{ss} \simeq 1 \oplus \chi \oplus \chi^\perp.\]
\item For $z \in Z$, $(\rho_z^{ss})_{G_{K}}$ is crystalline of $\tau$-Hodge-Tate weights $h_1(z) < h_2(z) < h_3(z)$, and 
its $\tau$-refinement given by $F_1$ is,
\[ (p^{h_1-h_3(z)}F_1(z),1,p^{h_3-h_1(z)}F_1^{-1}(z)).\]
In particular,
\[ D_{crys}(\rho_z^{ss})_\tau^{\varphi^2 = p^{h_1-h_3}F_1(z)} \neq 0.\]
\item At $y_0$, the refinement is,
\[ (\chi^\perp_p(p),1,\chi_p(p)).\]
\end{enumerate}

\end{propen}
\begin{proof}
Recall that $p$ is inert here.
The modular form $f$ corresponds to a point $x_0 \in \mathcal E$, the Eigenvariety defined in Theorem \ref{thrHecke}, associated to the type $(K_f,J)$.
Let $B \subset B(w_0,r) \subset \mathcal W$ be the closed subset defined in the same fashion as in \cite{BC1} by, 
\begin{equation}
\label{eqB}
\left\{
\begin{array}{c}
   w_2 = 0    \\
    2w_1 + w_3 = a-2 
\end{array}
\right.
\end{equation}
Thus $w_0 \in B$. Define $X$ to be an irreducible component of $\mathcal E \otimes_{\mathcal W} B$ containing $x_f$.
We get $w_B : X \fleche B$ which is finite (if $r$ small enough) surjective.
We can thus look at the universal pseudo-character $T$ on $\mathcal E$ and compose it with $\mathcal O_\mathcal E \fleche \mathcal O(X)$.
Applying lemma 7.2 of \cite{BC1}, we get an affinoid $Y$, regular of dimension 1, $y_0 \in Y$ and a finite surjective morphism $m : Y \fleche X$ such that $m(y_0) = x_f$ and there exist a representation
$\rho : G_E \fleche \GL_3(K(Y))$ of trace $G_E \overset{T}{\fleche} \mathcal O_X \fleche \mathcal O_Y$ satisfying (ABS). At $y_0$, the representation $\rho_{y_0}^{ss}$ is given by $1 \oplus \chi \oplus \chi^\perp$.
The map $h$ is given as follow. First, denote $\nu = (\nu_1,\nu_2,\nu_3) = (1+w_3,-1-w_1,-w_2)$ and $h$ is given by composition of with $m$ of the map $\nu$ (the shift of $w$) of $\mathcal E$, it is still finite and surjective on $B$\footnote{We did a slight abuse of notation, it should be $B' \simeq B$, centered in $h_0$, and whose equation are given below}, and for every $y \in Y$ such that $m(y) = x_f$, then $h(y) = h_0$. In terms of automorphic weight $(k_1,k_2,k_3)$ the previous map is given by $(1-k_3,k_2-1,k_1)$, and thus gives the Hodge-Tate weights for regular discrete series. In terms of Hodge-Tate weights, the equations (\ref{eqB}) giving rise to $B$ becomes 
\begin{equation}
\label{eqC}
\left\{
\begin{array}{c}
   h_3=0 \\ h_1 -2 h_2 = 1+a
\end{array}
\right.
\end{equation}

Denote by, 
\begin{eqnarray*} \mathcal Z = \{ \underline h \in B \cap h_0 + (p-1)(p+1)^2\ZZ^{3,dom} : -h_2 > C, h_2-h_1 > C'\\  |h_1-h_3| > C'' 
\},\end{eqnarray*}
where 
$C'' >0$ is bigger than the bound given (up to reducing $r$ and thus $B$) in proposition \ref{cryspoint} for crystallinity, $C'$ is the boung given by classicity theorem \ref{thrclassbij}, and $C$ is the bound given in classicity at the level of sheaves, proposition \ref{proclasssheaf} (remark that $h_3$ is constant). Then $\mathcal Z$ is strongly Zariski dense in $B$.
Then $Z := \kappa^{-1}(\mathcal Z) \subset Y(F)$ contains only classical (and regular) points by proposition \ref{thrclassfaisc} and the classicity result of Bijakowski (theorem \ref{thrclassbij}). Moreover they are all crystalline by proposition \ref{cryspoint}. It is strongly Zariski dense by flatness (thus openness) of $\kappa$.
Let us define $F_1$. The action on a point $x_f$ associated to modular form $f$ – of (classical automorphic) weight $(k_1,k_2,k_3)$ associated to $\pi$ that is a quotient of $\ind_b^{GU(3)}(\psi)$  – of the operator
\[U_pS_p^{-1},\] 
which corresponds up to a normalisation by $\frac{1}{p^{h_1 - h_3 - 1}} = p^{k_1+k_3}$ to the classical Iwahori double coset \[p^{-1}\left(
\begin{array}{ccc}
p^2  &   &   \\
  & p  &   \\
  &   & 1  
\end{array}
\right) = \left(
\begin{array}{ccc}
p &   &   \\
  & 1  &   \\
  &   & p^{-1}  
\end{array}
\right),
\] 
corresponds to
\[ p^{k_1+ k_3}{\psi}_1^\sigma(p),\]
where $\psi = (\psi_1,\psi_2)$ is a character of $(\mathcal O^\times)^2$ and $\sigma$ the refinement of $f$ associated to the action of $\mathcal A(p)$.
Indeed, the eigenvalue of $U_p$ coincide with $p^{-k_2}\psi_1^\sigma(p)^2\psi_2^\sigma$ and the one of $S_p$ with $p^{-k_1-k_2-k_3}\psi_1^\sigma(p)\psi_2^\sigma(p)$.
Thus, $U_pS_p^{-1}$ has eigenvalue $p^{k_1+k_3}\psi_1^\sigma(p) = p^{h_3-h_1}p\psi_1^\sigma(p)$.
Thus, we set $F_1$ the function on $\mathcal E$ given by $p^{-1}U_pS_p^{-1}$. We have that $p^{h_1-h_3}F_1 = \psi^\sigma_1(p)$. 
The property (2) comes from the construction of the eigenvariety $\mathcal E$. Part (3) is the calculation of the Galois representation associated to $\pi^n(\chi)$.  
Part (4) is local-global compatibility at $\ell = p$ (\cite{Ski} as recalled in section \ref{normweight}) and proposition \ref{cryspoint} as the eigenvalues of the crystalline Frobenius $\varphi^2$ coincide with $\psi_i^\sigma(p)$. 

The last assertion is the calculation made in proposition \ref{rafpin}.

\end{proof}

\begin{propen}
If $p =v\overline v$ is split, there exists,
\begin{enumerate}
\item A dimension 1 regular integral affinoid $Y$ over $F$, and $y_0 \in Y(F)$,
\item a semi-simple continuous representation,
\[ \rho_{K(Y)} : \Gal(\overline E/E)_{Np} \fleche \GL_3(K(Y)),\]
satisfying $\rho^\perp_{K(Y)} \simeq \rho_{K(Y)},$ the property (ABS) of \cite{BC1}, and $\tr(\rho_{K(Y)})(\Gal(\overline E/E)) \subset \mathcal O_Y$,
\item A $F$-morphism, $h = (h_1,h_2,h_3) : Y \fleche \mathbb A^3$ such that $h_3 = 0$, $h(y_0) = h_0$.
\item A subset $Z \subset Y(F)$ such that $h(Z) \subset h_0 + (p-1)\ZZ^{3}_{dom}$.
\item Three functions $F_1,F_2,F_3$ in $\mathcal O(Y)$ of constant valuation.
\end{enumerate}
such that,
\begin{enumerate}
\item For every affinoid $\Omega$ containing $y_0$, $\Omega \cap Z$ is Zariski dense in $\Omega$.
\item For all $z \in Z \cup \{y_0\}$ $\rho_z^{ss}$ is the Galois representation associated to a cuspidal (algebraic) automorphic representation $\Pi$ of $GU(2,1)$ such that
\[ Hom_{K_f}(J,\Pi) \neq 0.\]
\item \[\rho_{y_0}^{ss} \simeq 1 \oplus \chi \oplus \chi^\perp.\]
\item For $z \in Z $, $(\rho_z^{ss})_{G_v}$ is crystalline of Hodge-Tate weights $h_1(z) < h_2(z) <h_3(z)$, and 
\[ (p^{h_1(z)}F_1(z),p^{h_2(z)}F_2(z),p^{h_3(z)}F_3(z))\]
is an accessible refinement of $\rho_z^{ss}$.
\item At $y_0$, this refinement is $(\chi_v^\perp(p),1,\chi_v(p))$.

\end{enumerate}

\end{propen}

\begin{proof}
As the proof is almost the same as \cite{BC1} and we chose to detail the inert case, we will just sketch it.
Choose $x_f$ the point in $\mathcal E$ associated to $\pi^n(\chi)$ and the accessible refinement $(\chi_v^\perp(p),1,\chi_v(p))$. Denote by $B\subset \mathcal W$ the closed subset defined as in the inert case by 
\begin{equation}
\label{eqD}
\left\{
\begin{array}{c}
   k_1 = 0    \\
    2k_2+k_3= 2-a  
\end{array}
\right.
\end{equation}
and choose $X$ an irreducible component of $\mathcal E \times_{\mathcal W} B$ containing $x_f$. Apply lemma 7.2 of \cite{BC1}, and get $Y$ regular and $y_0$ and a representation,
\[ \rho : G_{E,Np } \fleche \GL_3(\mathcal O_Y),\]
such that $\rho^\perp = \rho$. Denote $h$ as in the inert case ($\nu = (1-k_3,k_2-1,k_1)$), and idem for $Z$ (classicity at the level of sheaves is given by \cite{Bra}, 6.2, and classicity by Pilloni-Stroh \cite{PSdep} or (in greater generality) \cite{BPS}.). 
The four Hecke operators living on $\mathcal E$, $U_i, i = 0,\dots,3$ are normalized as in \ref{Heckenormsplit}, then set for $i=1,2,3$,
\[ F_i = U_{i-1}U_3^{-1}.\]
By subsection \ref{ModClassHecke}, and local-global compatibility at $v$ (with the fact that $v$ coincide with $\tau_\infty$), $h_i$ are the Hodge-Tate weights of $(\rho_z)_{|G_v}$ and the normalisation of the Hecke Operators recalled in $\ref{ModClassHecke}$ assure that $(p^{h_i}F_i)_i$ is a refinement at $v$ for all classical forms.
\end{proof}

\subsection{Constructing the extension}

\begin{propen}
\label{propirr}
We are in one of the following two cases :
\begin{enumerate}
\item $\rho_{K(Y)}$ is absolutely irreducible.
\item There exists a two dimensional representation $r \subset \rho$ such that $r_{K(Y)}$ is absolutely irreducible and,
\[r_{y_0}^{ss} = 
\left(
\begin{array}{cc}
  \chi     &   \\
  &   \chi^\perp   
\end{array}
\right).
\]
\end{enumerate}
\end{propen}

\begin{remaen}
As showed by the proof, the second case never happens if $a \geq 2$ and $p$ is split.
\end{remaen}

\begin{proof}
The proof uses similar ideas as in \cite{BC1} Proposition 9.1, but unfortunately in our case the refinement autorises a 2-dimensional subrepresentation and a 
1-dimensional quotient. Let us first suppose $p$ splits, and suppose $\rho_{K(Y)}$ is reducible. As $\rho^\perp \simeq \rho$, we can suppose that there exists a character $\psi \subset \rho$ and a two dimensionnal representation $r$ such that
\[ 0 \fleche \psi \fleche \rho \fleche r\fleche 0.\]
Thus there exists $i$ such that the (generalized) Hodge-Tate weight of $\psi$ at $v$ is $h_i$. Moreover, for all $z \in Z$, by weak admissibility of $\psi$, we must have that there exists a $j$ such that $v(p^{h_j}F_j(z)) = h_i(z)$. As the valuation of $F_j$ is constant on $Y$, we can calculate it at $y_0$ and
\[ \alpha = (v(F_1),v(F_2),v(F_3)) = (0,a,-a).\]
In particular, at $z \in Z$ such that $|h_i(z) - h_j(z)|  > a$ for all $i \neq j$, we find $\alpha_j = 0$ and $i=j = 1$. Thus, by density $\psi$ is of Hodge-Tate weight $h_1$.
In particular, at $y_0$, $\psi$ is of Hodge-Tate weight $h_1(y_0) = 1-a$. 
Now, if $a \neq 1$, $\psi_{y_0} = \chi$, and if $r$ were reducible, then by weak admissibility we would find $i,j \neq 1$ such that $h_i(z) = h_j(z)  + \alpha_j$, for a Zariski dense subset of $z \in Z$, which is absurd. Thus there is a unique sub-quotient of $\rho$ which is of rank 1, it is $\psi$. As $\rho^\perp = \rho$, this means that $\psi^\perp = \psi$, which is impossible as $\psi_{y_0} = \chi$.
If $a=1$ then $\psi_{y_0}$ has $v$-Hodge-Tate weight 0, and if $\psi_{y_0} = \chi$ the same as previously happens, thus suppose that $\psi_{y_0} = 1$. In this case, $r$ is still irreducible but $r_{y_0}^{ss} = \chi \oplus \chi^\perp$.

Now, we focus on $p$ inert, which is similar. Suppose we are not in the case where $\rho_{K(Y)}$ is irreducible. We can thus find a 2-dimensional subrepresentation $r \subset \rho$ (if $r$ is one dimensional, take the quotient and apply $(.)^\perp$, as $\rho^\perp = \rho$).
Suppose that $r$ is reducible. Take $z \in \mathcal Z$, 
as the valuation $\alpha_1$ of $F_1$ is constant, we can calculate it at $y_0$ and we get, from $p^{h_1-h_3}F_1(y_0) = \chi_p^\perp(p)$,
\[ \alpha_1 = a.\]
But if $r_z^{ss}$ is not irreductible, this means, following Rogawski's Classification recalled in \cite{BC1}, section 3.2.3, and the fact that the representations associated by Blasius-Rogawski are irreducible (but not necessarily 3-dimensional), that $z$ is either endoscopic-tempered of type $(1,1,1)$, endoscopic non tempered or stable non tempered.
In the case endoscopic non tempered, looking at the Arthur parameter at infinity, the Hodge-Tate weights verifies $h_1 = h_2$ or $h_2 = h_3$, which is not possible by choice of $Y$ and $Z$. In the stable non tempered case, the Hodge-Tate weights are $(k,k,k)$, which is not allowed in $\mathcal Z$. So we need to check that $z$ is not endoscopic of type (1,1,1).
But in this case, this would mean by weak admissibility for $\rho_z^{ss}$ (which would thus be a totally split sum of three characters) that,
\[ \{h_1-h_3 + \alpha_1,0, h_3-h_1-\alpha_1\} = \{ h_1-h_3,0,h_3-h_1\},\]
but the previous equality is impossible for $-h_1< -a$ (which is the generic situation).
Thus $z \in \mathcal Z$ is endoscopic, tempered, of type (2,1), and $r$ is irreducible. By weak admissibility, and the previous calculations (or because $\rho^\perp= \rho)$), $r_{y_0}^{ss}$ has to be $\chi^\perp\oplus \chi$.

\end{proof}

\subsection{Good reduction outside $p$}

\begin{propen}
\label{prop1113}
In the previous case 1), denote $\rho' = \rho_{K(Y)}\otimes (\chi_p^\perp)^{-1}$.
Let $v | \ell \neq p$ be a place of $E$. Then,
\begin{enumerate}
\item If $v \notdivides \Cond(\chi)$, then $\rho_{K(Y)}$ and $\rho'$ are unramified at $v$.
\item If $v | \Cond(\chi)$, then $\dim_{K(Y)} (\rho'_{K(Y)})^{I_v} = 2$.
\end{enumerate}
In case 2), denote $r' = r_{K(Y)}\otimes(\chi_p^\perp)^{-1}$. Let $v | \ell \neq p$ be a place of $E$. Then,
\begin{enumerate}
\item If $v \notdivides \Cond(\chi)$, then $r_{K(Y)}$ and $r'$ are unramified at $v$.
\item If $v | \Cond(\chi)$, then $\dim_{K(Y)} (r')^{I_v} = 1$.
\end{enumerate}

\end{propen}

\begin{proof}
After all the constructions, this can be deduced as in \cite{BC1}. First there exists $g \in \mathcal O_Y$ such that $g(y_0) \neq 0$ and $\rho_{K(Y)}$ has a $\mathcal O_{Y,(g)}$ 
stable lattice. Denote $\rho$ the representation valued in $\mathcal O_{Y,(g)}$, and for all $y \in \Spm(\mathcal O_{Y,(g)}) = Y(g^{-1})$, $\rho_y$ the reduction at $y$.
In case 1), as $\rho_{K(Y)}$ is semi-simple, $\rho_z$ is semi simple for $z \in Z'$, a cofinite subset of $Z \cap Y(g^{-1})$. But now, for $z \in Z'$, $\rho_z = \rho_z^{ss}$ is the Galois representation 
associated to a regular automorphic representation $\Pi_z$ of $GU(2,1)$. 
In case 2), $r_{K(Y)}$ is semi-simple, thus for all $z \in  Z'$, still cofinite in $ Z$, $r_z^{ss} = r_z \subset \rho_z^{ss}$, and,
\[ \dim_{K(Y)}{r'}^{I_v} \geq \dim_{K(Y)} (\rho'^{ss})^{I_v} - 1,\]
and $\dim_{K(Y)} (\rho')^{I_v}$ is related to the ramification of a (tempered endoscopic of type (2,1)) automorphic representation of $GU(2,1)$.
Thus, to show the result, we only need to control ramification at $v$ of (the base change of) $\Pi_z$.

If $v \notdivides \Cond(\chi)$, by construction of the eigenvariety and choice of the maximal compact, $(\Pi_z)_v$ has a vector fixed by $K_\ell$. We can thus conclude as if $\ell$ is 
unramified, $K_\ell$ is hyperspecial and if $\ell$ ramifies, $K_\ell$ is chosen very special and \cite{BC1} proposition 3.1 gives the result for the base change. Now by local-global 
compatibility (for example \cite{Ski}), $\rho_z^{ss}$ (and thus $r_z$ in case 2) is unramified at $v$.

If $v | \Cond(\chi)$, by construction $\Pi_z$ has type $(K_\ell, J_\ell \otimes \chi_{0,\ell}^{-1} \circ \det)$, and thus by proposition \ref{proptypes} the local langlands representation 
associated to $(\Pi_z)_v$ is $\phi_1\oplus \phi_2 \oplus \phi_3\chi_{0,\ell}$ for three unramfied characters $\phi_i$. Thus, by local-global compatibility again, there exists $I_v'$ a finite 
index subgroup of $I_v$ such that $\rho'_z({I_v'}) = 1$. Thus, $(\rho')({I_v'}) =1$. But up to extending scalars, $\rho'_{|I_v}$ is a finite representation $\theta$ of $I_v/I_v'$, defined on $F'$ 
a finite extension of $F$. Thus, $\rho'_{I_v} \otimes_F F'$ is well defined, semi-simple, and evaluating the trace, we get,
\[ 1 \oplus 1 \oplus (((\chi_p)^{\perp})^{-1})_{I_v} = (\rho'_{|I_v} \otimes F')_{y_0}^{ss} \simeq \theta.\]
We thus get the result.\end{proof}

\subsection{Elimination of case (2)}
\label{sect115}
We want to prove that $\rho_{K(Y)}$ is always irreducible, and thus prove that case 2. can never happen. Thus suppose we are in case 2.

\begin{propen}
There exists a continuous representation $\overline r : G_E \fleche \GL_3(F)$ such that $\overline r$ is a non split extension of $\chi^\perp$ by $\chi$,
\[ \overline r = 
\left(
\begin{array}{cc}

 \chi & \star     \\
  &  \chi^\perp  
\end{array}
\right)
\]
such that $\overline r = \overline r^\perp$ and verifying,
\begin{enumerate}
\item $\dim_F (\overline r \otimes \overline \chi)^{I_v} = 2$ if $v \notdivides \Cond(\chi)$.

\item $\dim_F (\overline r \otimes \overline \chi)^{I_v} \geq 1$ if $v | \Cond(\chi)$.
\item  If $p$ splits $D_{cris,\overline v}(\overline r)^{\phi = \chi_{\overline v}^\perp(p)} \neq 0$ and $D_{cris,v}(\overline r)^{\phi = \chi_v^\perp(p)} \neq 0$
\item  If $p$ is inert $D_{cris,\tau}(\overline r)^{\phi^2 = \chi^\perp(p)} \neq 0$
\end{enumerate}
\end{propen}

\begin{proof}
We first sketch the proof in case $p$ inert we will detail a bit the argument in proposition \ref{propab}.
First, by Ribet's Lemma (see \cite{BellRibet} Corollaire 1 or \cite{CheApp} Appendice, Lemma 3.1 and \cite{BC1} Lemme 7.3) there exists a $g \neq 0 \in \mathcal O_Y$ and a $\mathcal O_{Y,(g)}$-lattice $\Lambda$ stable by $r_{K(Y)}$ such that
$\overline r := \overline{r_{\Lambda}} = \left(
\begin{array}{cc}

 \chi & \star     \\
  &  \chi^\perp  
\end{array}
\right)$
is a non split extension. We can moreover easily assure that this lattice is a direct factor of a lattice stable by $\rho_{K(Y)}$.  Then, condition 1. and 2. follows from the proposition \ref{prop1113}. For condition 3., we can use the analog of Kisin's argument as extended by Liu, \cite{Liu}, as in the proof of the next proposition, as for all $z \in Z$, 
\[D_{cris,\tau}(r_z)^{\phi^2 = p^{h_1}F_1} = D_{cris,\tau}(\rho_z)^{\phi^2=p^{h_1}F_1} \neq 0\]
as shown by proposition \ref{propirr}.
The split case is similar to \cite{BC1} Proposition 9.3 and Lemme 9.1 (see also Remarque 6.3.8 or apply the results to $\overline r_{\overline v}$).  
 \end{proof}

Denote by $r' = r \otimes (\chi_p^\perp)^{-1} = r \otimes \overline{\chi_p}$, which is an extension of $1$ by $\chi_p\overline{\chi_p} = \omega_p$ (the cyclotomic character).

\begin{lemmen}
\label{lemmacrys}
The representation $r'$ is crystalline at $p$.
\end{lemmen}

\begin{proof}
 As $\chi^\perp$ is crystalline (at $v$ and $\overline v$ if $p$ splits, at $p$ if $p$ is inert), it is enough to prove that $r$ is crystalline. Suppose first $p$ is inert.
As $D_{crys,\tau}$ is left-exact, because $\overline r$ is extension of $\chi_p^\perp$ by $\chi_p$ we have,
\[ D_{crys,\tau}(\chi_p) \subset D_{crys,\tau}(\overline r),\]
but on $D_{crys,\tau}(\chi_p)$ $\varphi^2$ acts as $\chi_p(p) = \chi_p^\perp(p)p^{-2}$, thus this line is distinct from $D_{crys,\tau}(r)^{\varphi^2=\chi^\perp(p)}$ and thus $D_{crys,\tau}(r)$ is of dimension 2.
But because of the action of $\varphi$, $D_{crys}(r)$ is a $K\otimes_{\ZZ_p} F$-module of dimension 2, i.e. $r$ is crystalline. If $p$ splits, as $\overline r^\perp \simeq r$, it suffices to prove that $\overline r$ is crystalline at $\overline v$. But we use $D_{cris,\overline v}(\overline r)^{\phi = \chi_{\overline v}^\perp(p)}$ as in the inert case to get the result.
\end{proof}

Thus $r'$ gives a non zero element in $H^1_f(E,\omega_p)$ but by \cite{BC1} Lemme 9.3, which is a well-know result, $H^1_f(E,\omega_p) = \{0\}$ thus $r'$ must be trivial, which
gives a contradiction. We are thus in case where $\rho_{K(Y)}$ is irreducible.

\subsection{Good reduction at $p$}	

Suppose $p$ inert. The result for $p$ split is analogous to \cite{BC1} Proposition 9.3. Denote $u = \chi_p^\perp(p)$.
\begin{propen}
\label{propab}
There exists a continuous representation $\overline{\rho} : G_E \fleche \GL_3(F)$ such that,
\begin{enumerate}
\item For all place $v$ of $E$ not dividing $p$, we have, 
\begin{enumerate}
\item $\dim_F(\overline \rho \otimes (\chi_p^\perp)^{-1})^{I_v} \geq 2$ if $v | \Cond(\chi)$.
\item $\dim_F(\overline \rho \otimes (\chi_p^\perp)^{-1})^{I_v} = 3$ if $v \notdivides \Cond(\chi)$.
\end{enumerate}
\item $D_{cris,\tau}(\overline{\rho})^{\varphi^2 = u}$ is non zero.
\item $\overline{\rho}^{ss} \simeq \chi_p \oplus \chi_p^\perp \oplus 1$ and one of the two assertions is true :
\begin{enumerate}
\item Either $\overline{\rho}$ has a subquotient $r$ of dimension 2, such that $r^\perp \simeq r$ and $r$ is a non trivial extension of $\chi_p^\perp$ by $\chi_p$.
\item Either $\overline{\rho} \simeq \overline{\rho}^\perp$; $\overline{\rho}$ has a unique sub-representation $r_1$ of dimension 2 and a unique subquotient $r_2$ of dimension 2, with $r_1$ a non trivial extension of 1 by $\chi_p$ and $r_2$ a non trivial extension of $\chi_p^\perp$ by 1, and $r_1^\perp \simeq r_2$.
\end{enumerate}
\end{enumerate}
\end{propen}

\begin{proof}
Denote by $\mathcal O$ the rigid local ring of $Y$ at $y_0$, a discrete valuation ring of residual field $F$, denote $L$ its fraction field, and $\rho_L$ the representation which is the scalar extension of $\rho$ to $L$.
As $\overline{\rho_L}^{ss} = 1 \oplus \chi_p \oplus \chi_p^\perp$ which are pairwise distincts characters, we can use also \cite{BC1} proposition 7.1, the analog to Ribet's theorem, to find $\Lambda \subset L^3$ a lattice stable by ${\rho}_L$ such that the reduced representation $\overline{\rho} = \overline{\rho_{\Lambda}}$ satisfies condition (3)(a) or (3)(b). The condition (i) is true by what preceed. We can argue as in \cite{BC1} to get (ii), but we will need a generalisation to $G_K$ if $p$ is inert. Fortunately what we need is in \cite{Liu}.
As in  \cite{BC1} Lemma 7.3 there is an affinoid $Y \supset \Omega \ni y_0$ such that $\rho_L$ as a $\mathcal O_{\Omega}$-stable lattice $\Lambda_\Omega$ such that $\overline\rho_{\Lambda_\Omega,y_0} = \overline{\rho}$. Denote $\rho = \rho_{\Lambda_\Omega}$. Let thus $Z' \subset \Omega$ the points that are in $Z \subset Y$, in $\Omega$, and such that $\overline{\rho_{z}}$ is semi-simple (it is a cofinite subset of $Z \cap \Omega$ as $\rho_{K(\Omega)}$ is semi-simple (irreducible)).
By choice of $Z$, we have that for all $z \in Z'$, 
\[ D_{crys,\tau}(\rho_z)^{\phi^2 =p^{h_1(z)-h_3(z)}F_1(z)} \neq 0.\]
As $\rho$ is polarized, its $\sigma\tau$-Hodge-Tate weights are $h^\sigma(z) = (-h_3,-h_2,-h_1)$. Set $h^K_i = (h_i,h_{4-i}) \in F_\tau \times F_{\sigma\tau} = K\otimes_{\QQ_p} F$.
Thus $(\Omega,\rho, (h^K_i)_i,F_1,Z)$ is a weakly refined (polarised) $p$-adic representation of $G_K$ of dimension 3 in the sense of \cite{Liu} Definition 0.3.1.
To verify (f) of \cite{Liu} Definition 0.3.1, recall that over the weight space $\mathcal W$ we had an universal character $\chi = \chi_1 \times \chi_2 : \mathcal O^\times \times \mathcal O^1 \fleche  \mathcal{O}(\mathcal W)^\times$, and as $\mathcal W$ is regarded over $K$, we can split $\mathcal O(\mathcal W)\otimes_{\QQ_p}  K = \mathcal O(W)_\tau \times\mathcal O(W)_{\sigma\tau}$. Under this isomorphism, the derivative at $1$ of $\chi_1$, denoted $(wt_\tau(\chi_1),wt_{\sigma\tau}(\chi_1))$ is at every $\kappa \in \ZZ^3 \subset \mathcal W$ given by
\[ (k_1,k_3) = (h_3,1-h_1).\]
Thus set $\psi = \tau^{-1}(\chi_1\circ c)$. Its derivative at 1 is given by $(k_3-1,k_1) = (-h_1,h_3)$ at classical points $\kappa \in \ZZ^3$.
Thus, the character,
\[ \mathcal O^\times \overset{\psi}{\fleche} \mathcal O(\mathcal W)^\times \fleche \mathcal O(B)^\times \fleche \mathcal O(\Omega)^\times,\]
has the desired property (f).

Denote $\rho' = \rho \otimes \psi$ (where $\psi$ is precomposed by product of the two Lubin-Tate characters of $K$, $G_K \fleche O^\times$). Thus $\rho'$ has 
$\kappa_1^{K'} = (0,0)$ as smallest Hodge-Tate weight.
In Particular by \cite{Liu} Theorem 0.3.2, $\Omega_{fs} = \Omega$. But then by theorem 0.1.2 applied to $S = \Omega$, $k,n$ big enough, and $\rho'$, we have that,
\[ D_{crys}^+(\rho'_{y_0})^{\varphi^2 = F_1} \simeq D_{Sen}^+(\rho'_{y_0})^\Gamma,\]
(see remark 3.3.5 and corollary 1.5.4 of \cite{Liu}, as $0$ is the only non-positive Hodge-Tate weight\footnote{In \cite{Liu} the convenction of the Hodge-Tate weights is opposite to ours, there the Hodge-Tate weight of the cyclotomic character is 1.} of $\rho'$, corollary 1.5.4 applies), and $D_{Sen}^+(\rho'_{y_0})^\Gamma \neq 0$.

Thus $D_{crys,\tau}(\rho'_{y_0})^{\varphi^2=F_1} \neq 0$ which means \[D_{crys,\tau}(\rho_{y_0})^{\varphi^2 = p^{h_1(y_0)-h_3(y_0)}F_1} = D_{crys,\tau}(\rho_{y_0})^{\varphi^2 = u} \neq 0.\]
 \end{proof}

\subsection{Elimination of case (a)}

We can do as in \cite{BC1}, and as we eliminated case 2. of proposition \ref{propirr}. Suppose we are in case (a), there is thus a subquotient $r$ of $\overline{\rho}$ such that $r^\perp \simeq r$ and $r$ is an extension of $\chi_p^\perp$ by $\chi_p$. Denote by $r' = r \otimes (\chi_p^\perp)^{-1} = r \otimes \overline{\chi_p}$, which is an extension of $1$ by $\chi_p\overline{\chi_p} = \omega_p$ (the cyclotomic character).

\begin{lemmen}
The representation $r'$ is crystalline at $p$ (at $v_1,v_2 | p$ is $p$ is split).
\end{lemmen}

\begin{proof}
The split case is identical to \cite{BC1}, lemma 9.1 and lemma \ref{lemmacrys}
 Suppose $p$ is inert. As $\chi_p^\perp$ is crystalline, it is enough to prove that $r$ is crystalline. But $V \mapsto D_{crys,\tau}(V)^{\varphi^2 = u}$ is left-exact, thus if we denote $u = \chi^\perp(p)$,
\[ \dim_F D_{crys,\tau}(\overline \rho)^{\varphi^2=u} \leq \dim_F D_{crys,\tau}(r)^{\varphi^2=u} + D_{crys,\tau}(1)^{\varphi^2=u}.\]
But $D_{crys,\tau}(1)^{\varphi^2=u} = 0$ thus $D_{crys,\tau}(r)^{\varphi^2=u} \neq 0$.
The end of the proof is identical to lemma \ref{lemmacrys}.
\end{proof}

\begin{lemmen}
The representation $r'$ is unramified at every place $w \notdivides p$.
\end{lemmen}

\begin{proof}
This is exactly identical to \cite{BC1} Lemme 9.2.
\end{proof}

Thus by \cite{BC1} Lemme 9.3, $r$ must be trivial, which contradicts Proposition \ref{propab} 3)(a).

\subsection{Conclusion}
 We are thus in case 3)(b), with $r_1$ a non trivial extension of $1$ by $\chi_p$.

\begin{lemmen}
$r_1$ is crystalline at $p$ if $p$ is inert, and at $v_1,v_2 |p$ if $p$ splits.
\end{lemmen}

\begin{proof}
Suppose $p$ inert. As $r_1 \simeq r_2^\perp$, we only need to prove that $r_2$ is crystalline.
Because $D_{crys,\tau}(\cdot)^{\varphi^2 = u}$ is left-exact, we again have,
\[ \dim_F D_{crys,\tau}(\overline \rho)^{\varphi^2=u} \leq \dim_F D_{crys,\tau}(r_2)^{\varphi^2=u} + \dim_F D_{crys,\tau}(\chi_p)^{\varphi^2=u}.\]
As $D_{crys,\tau}(\chi_p)^{\varphi^2=u} = \{0\}$ and $\dim_F D_{crys,\tau}(\overline \rho)^{\varphi^2=u} \neq 0$, we have $\dim_F D_{crys,\tau}(r_2)^{\varphi^2=u}\neq \{0\}$.
Moreover,
\[ D_{crys,\tau}(1) \subset D_{crys,\tau}(r_2),\]
by left-exacness of $D_{crys,\tau}$, which gives a line where $\varphi^2$ acts as $1 \neq u$. Thus there are at least two different lines in $D_{crys,\tau}(r_2)$ which means this is $2$-dimensional and by existence of $\varphi$, $r_2$ (thus $r_1$) is crystalline.

Now suppose $p$ splits. Then the proof is identical to \cite{BC1} Lemme 9.4, as $1 \neq \chi_v(p) \neq \chi_v^\perp(p) \neq 1$ (Recall $|\chi_v(p)|_\CC = |\chi_{\overline v}(p)|_\CC = p^{-\frac{1}{2}}$). 
\end{proof}

\begin{theoren}
The representation $r_1$ gives a non-zero element of $H^1_f(E,\chi_p)$.
\end{theoren}

\begin{proof}
We only need to prove that $r_1$ has good reduction outside $p$. 
But then as $\rho$ in unramified outside $p\Cond(\chi)$, by proposition \ref{propab}, we only need to check for $v | \Cond(\chi)$.
We have shown in the proof of Proposition \ref{prop1113} that there exists an open subgroup $I_w' \subset I_w$ such that $\rho'_{|I_w}$ factors through $I_w/I_{w}'$ and  $\rho'_{|I_w}= 1 \oplus 1 \oplus (\chi_p^{\perp})^{-1}_{|I_w}$. Thus $r_1^{I_w}$ is then of dimension 1.  
\end{proof}

\appendix
\section{Calculations on the weight space}
\label{AppW}

In this appendix we explain a bit more the structure of the weight space $\mathcal W$ defined in section \ref{sectW}.
$\mathcal W$ is represented by a disjoint union of ($(p+1)(p^2-1)$) 3-dimensional open balls over $\mathcal O$. Indeed (if $p \neq 2$)
\[ \mathcal O^\times \simeq (\mathbb F_{p^2})^\times \times (1 + p\mathcal O),\]
which induced, up to the choice of a basis of $\mathcal O$ over $\ZZ_p$, an isomorphism,
\[ \Hom_{cont}(\mathcal O^\times,\mathbb G_m) \simeq \widehat{(\ZZ/(p^2-1)\ZZ)} \times B_2(1,1^-),\]
where $B_2(1,1^-)$ is the open 2-dimensional ball centered in 1, of radius 1. And, as a $\ZZ_p$-module
\[ \mathcal O^1 \simeq S \times \ZZ_p,\]
where $S$ is a finite group of cardinal $p+1$.

\begin{proof}[Proof]
We have the exact sequence,
\[ 0 \fleche \mathcal O^1 \fleche \mathcal O^\times \overset{Nm}{\fleche} \ZZ_p^\times \fleche 0,\]
(surjectivity is given by local class field theory for example). Reducing modulo $p$, we have the surjectivity of $\mathbb F_{p^2} \overset{\overline{Nm}}{\fleche} \FP$. We thus have the diagram,
\begin{center}
\begin{tikzpicture}[description/.style={fill=white,inner sep=2pt}] 
\matrix (m) [matrix of math nodes, row sep=3em, column sep=2.5em, text height=1.5ex, text depth=0.25ex] at (0,0)
{ 
& 0 & 0 & 0 & \\
0 & \{x \in \mathcal O^1 : x \equiv 1 \pmod p\} & 1+p\mathcal O & 1+p\ZZ_p &   \\
0 & \mathcal O^1 & \mathcal O^\times & \ZZ_p^\times & 0 \\
0 & \{x \in \FF_{p^2} : x^{p+1} = 1\} & \FF_{p^2}^\times & \FP^\times & 0 \\
 &   & 0 & 0 & \\
 };
\path[->,font=\scriptsize] 
(m-2-3) edge node[auto] {${Nm}^1$} (m-2-4)
(m-2-1) edge node[auto] {$$} (m-2-2)
(m-2-2) edge node[auto] {$$} (m-2-3)
%(m-2-4) edge node[auto] {$$} (m-2-5)
(m-3-3) edge node[auto] {$Nm$} (m-3-4)
(m-3-1) edge node[auto] {$$} (m-3-2)
(m-3-2) edge node[auto] {$$} (m-3-3)
(m-3-4) edge node[auto] {$$} (m-3-5)
(m-4-3) edge node[auto] {$\overline{Nm}$} (m-4-4)
(m-4-1) edge node[auto] {$$} (m-4-2)
(m-4-2) edge node[auto] {$$} (m-4-3)
(m-4-4) edge node[auto] {$$} (m-4-5)
(m-1-2) edge node[auto] {$$} (m-2-2)
(m-1-3) edge node[auto] {$$} (m-2-3)
(m-1-4) edge node[auto] {$$} (m-2-4)
%(m-4-2) edge node[auto] {$$} (m-5-2)
(m-4-3) edge node[auto] {$$} (m-5-3)
(m-4-4) edge node[auto] {$$} (m-5-4)
(m-3-2) edge node[auto] {$$} (m-4-2)
(m-3-3) edge node[auto] {$$} (m-4-3)
(m-3-4) edge node[auto] {$$} (m-4-4)
(m-2-2) edge node[auto] {$$} (m-3-2)
(m-2-3) edge node[auto] {$$} (m-3-3)
(m-2-4) edge node[auto] {$$} (m-3-4);
\end{tikzpicture}
\end{center}
The application $Nm^1 = 1 + p\mathcal O \fleche 1 + p\ZZ_p$ is surjective ; indeed, for all $z$ inside $1+p\ZZ_p$, because $Nm$ is surjective, there exists $u \in \mathcal O^\times$ such that 
$uu^\sigma = 1+pz$ (denote by $\sigma$ the conjugation, and $\overline{\bullet}$ reduction modulo $p$). We deduce that $\overline u \in \{ x \in \mathbb F_{p^2} : x^{p+1}=1\}$.
 We then set $u' = u/[\overline u]$, where $[.]$ denote the Teichmuller lift. Then 
 $u' \in 1 + p\mathcal O$ and $(u')(u')^\sigma = uu^\sigma/([\overline u][\overline u]^\sigma) = uu^\sigma/([\overline u^{p+1}]) = 1 + pz$. The second equality is because $[.]$ 
 commutes with Frobenius. (We could also prove the surjectivity by a method of successive approximations).
The map $\mathcal O^1 \fleche \{x  \in \mathbb F_{p^2} : x^{p+1} = 1\}$ is also surjective : for all $x \in \{ x \in \mathbb F_{p^2}:x^{p+1} =1 \}$, $[x][x]^\sigma = [x^{p+1}] = [1] = 1$. 
Thus, up to choosing a base of $\mathcal O$ over $\ZZ_p$, we can with the logarithm identify $1 + p\mathcal O$ to $\ZZ_p^2$; this assures that $\{ x \in \mathcal O^1 : x \equiv 1 \pmod p\} \simeq \ZZ_p$ (because logarithm exchanges trace and $Nm$).
\end{proof} 

In particular,
\[ \Hom_{cont}(\mathcal O^1,\mathbb G_m) \simeq \coprod_{\hat S} B_1(1,1^-).\]
Thus, $\mathcal W$ is isomorphic to a union of $(p+1)(p^2-1)$ open balls of dimension 3.
There is also a universal character,
\[ \kappa^{un} : T^1(\ZZ_p) \fleche \ZZ_p[[T^1(\ZZ_p)]].\]
The following lemma is essential,
\begin{lemmen}
Every weight $\kappa \in \mathcal W(K)$ is automatically locally $(\QQ_p-$)analytic.
\end{lemmen}
Actually we can be more precise,
\begin{lemmen}
Let $\mathcal U \subset \mathcal W$ a quasi-compact open, then there exists $w_\mathcal U$ such that $\kappa^{un}_{|\mathcal U}$ is $w_\mathcal U$-analytic.
\end{lemmen}

\begin{proof}[Proof]
It is \cite{Urb} Lemma 3.4.6.
\end{proof}

We will construct $\mathcal W(w)$, an open subset of $\mathcal W$ containing the $w$-analytic $\kappa$ (it is an affinoid). Set $w \in ]n-1,n]\cap v(\overline\QQ_p)$. We define it this way, following \cite{AIP}.
First set $\mathfrak W(w)^0$ to be $\Spf \mathcal O_K<<S_1,S_2,S_3>>$ where $K$ is a finite extension of $\QQ_p$ containing an element $p^w$ of valuation $w$.
Define $\mathfrak T_w$ the subtorus of $\mathfrak T$ the formal torus associated to $T^0$, given by,
\[ \mathfrak T_w (R) = \Ker(\mathfrak T(R) \fleche \mathfrak T(R/p^wR),\]
for any flat, $p$-adically complete $\mathcal O_K$-algebra $R$. Denote $X_i'$ the coordinates on $\mathfrak T_w$, so that $1+p^wX_i' = 1+X_i$ on $\mathfrak T$, and define the universal character,
\[
\kappa^{0un} :
\begin{array}{ccc}
\mathfrak T_w\times \mathfrak W(w)^0  & \fleche   & \widehat{\mathbb G_m}  \\
 (1+p^wX_1',1+p^wX_2',1+p^wX_3',S_1,S_2,S_3) & \longmapsto  & {\prod_{i=1}^3 (1+p^wX_i)^{S_ip^{-w+\frac{2}{p-1}}}}
\end{array}
\]
Then define $\mathcal W(w)^0$ to be the rigid fiber of $\mathfrak W(w)^0$ and finally, $W(w)$ to be the fiber product,
\[ \mathcal W \times_{\Hom_{cont}((1+p\mathcal O)\times(1+p\mathcal O)^1,\CC_p^\times)}\mathcal W(w)^0,\]
where the map $\mathcal W(w)^0 \fleche \Hom_{cont}((1+p\mathcal O)\times(1+p\mathcal O)^1,\CC_p^\times)$ is given by,
\[ (s_1,s_2,s_3) \longmapsto ((1+p^nx_1,1+p^nx_2,1+p^nx_3) \mapsto \prod_{i=1}^3(1+p^nx_i)^{s_ip^{-w + \frac{2}{p-1}}}.\]

Then we can write $\mathcal W = \bigcup_{w \geq 0} \mathcal W(w)$ as an increasing union of affinoids.

\section{Kernel of Frobenius}
\label{AppA}
\begin{propen}
On the stack $\mathcal{BT}^{\mathcal O}_{(2,1),pol}$ and $\overline X$, the Cartier divisor $\ha_\tau$ is reduced.
\end{propen}

\begin{proof}[Proof]
This is \cite{dsg} Theorem 2.8, which can be proved by considering the deformation space at a point. Unfortunately we can't use the result of \cite{Her1} because of the polarisation (but a similar proof works).
\end{proof}

\begin{propen}
Let $G/\Spec(\mathcal O_C)$ be a $p$-divisible $\mathcal O$-module. Suppose $\ha_\tau(G) < \frac{1}{2p^2}$, and let $K_1$ the first Frobenius-subgroup of $G$ (see theorem \ref{thrfiltcan}). Then \[K_1 \times_{\Spec(\mathcal O_C)} \Spec(\mathcal O_C/p^{\frac{1}{2p^2}}) = \Ker F^2 \times_{\Spec(\mathcal O_C/p} \Spec(\mathcal O_C/p^{\frac{1}{2p^2}}).\]
\end{propen}

Denote, for $K/\QQ_{p^2}$, by $\mathfrak X/\Spf(\mathcal O_K)$ a (smooth) presentation of $\mathcal{BT}_{r,(2,1),pol}^\mathcal O/\Spf(\mathcal O_K)$ (which is smooth, see for example \cite{Wed2})  and for $v \in v(K)$, $\mathfrak X(v)$ is the open subset of the blow-up along $I_v = (p^v,\ha_\tau)$ where $I_v$ is generated by $\ha_\tau$.
As $\mathfrak X$ is smooth and $\ha_\tau$ is reduced, $\mathfrak X(v)$ is normal and its special fiber (modulo $\pi_K$) is reduced.

Take $v = \frac{1}{2p^2}$ and $K$ a totally ramified extension of $\QQ_p$ of degree $\frac{1}{2p^2}$ (so that $v(\pi_K) = \frac{1}{2p^2}$).

Then over $X(v)$, the rigid fiber over $K$ of $\mathfrak X(v)$, we have a subgroup $K_1 \subset G[p^2]$, and by the proposition \ref{propKernormal} this subgroup extend to a subgroup over 
$\mathfrak X(v)$. Now, over $\mathfrak X(v) \otimes \mathcal O_K/\pi_K = \mathfrak X(v) \otimes \kappa_K$ the rigid fiber of $\mathfrak X(v)$, we have two subgroups,
$K_1$ and $\Ker F^2$, which coincide on every point (by \cite{Her2} section 9 or the very proof of the proposition \ref{propKernormal}) but as $\mathfrak X(v) \otimes \mathcal O_K/\pi_K$ is reduced, $K_1 = \Ker F^2$ over $X(v) \otimes \mathcal O_K/\pi_K$. As every $\mathcal O_C$-point of $\mathcal{BT}_{(2,1),pol}^\mathcal O$ gives a point of $\mathfrak X$, we have the result using $G[p^r]$ for $r$ big enough (bigger than 3 is enough).

\begin{coren}
Let $G$ as in the previous proposition, but suppose $\ha_\tau(G) < \frac{1}{2p^4}$. Then $\ha_\tau(G/K_1) = p^2\ha_\tau(G)$
\end{coren}

\begin{proof}[Proof]
Recall that $\ha_\tau = \ha_{\sigma\tau}$ and $\ha_{\tau}$ is given by $\det(V^2)$ without any division. 
By the previous proposition, the map $G[p^2] \fleche G[p^2]/K_1$ coincide modulo $\pi_K$ with the map $G[p^2] \overset{F^2}{\fleche} G[p^2]^{(p^2)}$.
Thus, there is an isomorphism modulo $\pi_K$ : $\det(\omega_{(G/K_1)^D,\sigma\tau}) \simeq \det(\omega_{G^D,\sigma\tau}^{\otimes p^2})$ which identifies (modulo $\pi_K$) $\widetilde{\ha_{\sigma\tau}}(G/K_1)$ with $\widetilde{\ha_{\sigma\tau}}(G)^{\otimes p^2}$.
Thus we get, \[\inf\{p^2\ha_\tau(G),\frac{1}{2p^2}\} = \inf\{\ha_\tau(G/K_1),\frac{1}{2p^2}\}.\] As $p^2\ha_\tau(G) < \frac{1}{2p^2}$, we get the result.
\end{proof}

\section{Devissage of the formal coherent locally analytic sheaves}

Let $\kappa \in \mathcal W(w)$ a character and $\kappa^0$ its restriction to $\mathcal W(w)^0$, and $w < m - \frac{p^2m-1}{p^2-1}$. Denote on $\mathfrak X_1(p^{2m})(v)$ the sheaf $\mathfrak w_w^{\kappa^0\dag}$ defined as,
\[ \zeta_*\mathcal O_{\mathfrak IW^+_w}[\kappa^0], \quad \text{where } \zeta : \mathfrak{IW}_w^{+} \fleche \mathfrak X_1(p^{2m})(v).\]
If we set $\pi : \mathfrak X_1(p^{2m})(v) \fleche \mathfrak X(v)$, then the sheaf $\mathfrak w_w^{\kappa\dag}$ of overconvergent forms is given by,
\[ (\pi_*\mathfrak w_w^{\kappa^0\dag})(-\kappa)^{B(\ZZ_p)\mathfrak B_w},\]
where $(-)(-\kappa')$ denote a twist of the action of $B(\ZZ_p)\mathfrak B_w$ and $(-)^{B(\ZZ_p)\mathfrak B_w}$ means taking invariants. Remark that after the twist, the action of ${B(\ZZ_p)\mathfrak B_w}$ factors through $B_n$.

Consider the projection "in family"
\[ \zeta \times 1 : \mathfrak{IW}_w^+ \times \mathfrak W(w)^0 \fleche \mathfrak X_1(p^{2m})(v) \times \mathfrak W(w)^0,\]
and denote \[\mathfrak w_w^{\kappa^{0,un}\dag} = (\zeta\times1)_*\mathcal O_{\mathfrak{IW}_w^+ \times \mathfrak W(w)^0}[\kappa^{0,un}],\]
the family of sheaves over $ \mathfrak X_1(p^{2m})(v) \times \mathfrak W(w)^0$.

Let $\Spf(R)$ a small enough open in $\mathfrak X_1(p^{2m})(v)$. Recall that we denote by $\psi$ the univeral polarized trivialisation of $K_m^D$, denote $e_1,e_2$ a basis of $\mathcal O/p^m\mathcal O\oplus \mathcal O/p^{2m}\mathcal O$, $e_1^{\sigma\tau} = \HT_{\sigma\tau,w}(e_1),e_2^{\sigma\tau} = \HT_{\sigma\tau,w}(e_2),e^{\tau} = \HT_{\tau,w}(e_2)$ the images of this basis in $\mathcal F_{\sigma\tau}/p^w,\mathcal F_{\tau}/p^w$. Denote $f_1^{\sigma\tau},f_2^{\sigma\tau},f^{\tau}$ a lift of this basis in $\mathcal F_{\sigma\tau},\mathcal F_{\tau}$.

With this choices we can identify $\mathfrak{IW}_{w|\Spf(R)}^+$ with matrices,
\[
\left(
\begin{array}{ccc}
 1 &   &   \\
 p^w\mathfrak B(0,1) &  1 &   \\
  &   &   1
\end{array}
\right) \times 
\left(
\begin{array}{c}
  1 + p^w\mathfrak B(0,1)   \\
   1 + p^w\mathfrak B(0,1)  \\
  1 + p^w\mathfrak B(0,1) 
\end{array}
\right)
\times_{\Spf(\mathcal O_K)} \Spf(R).
\]
Denote $X_0$ the coordinate in the $3$x$3$ matrix and $X_1,X_2,X_3$ the coordinates of the balls inside the column.
Thus, we can identify a function $f$ on $\mathfrak{IW}_{w|\Spf(R)}^+$ to a formal series in $R<<X_0,X_1,X_2,X_3>>$.

Now, let $\kappa^0 \in \mathfrak W(w)^0$, then $f \in \mathfrak w_w^{\kappa^0\dag}$ if it verifies,
\[ f(X_0,\lambda X_1,\lambda X_2,\lambda X_3) = (\kappa^0)'(\lambda) f(X_0,X_1,X_2,X_3), \quad \forall \lambda \in \mathfrak T_w(R).\]
In particular, we deduce that there exists a unique $g \in R<<X_0>>$ such that,
\[ f(X_0,X_1,X_2,X_3) = g(X_0)\kappa^0(X_1,X_2,X_3),\]
and thus there is a bijection $\mathfrak w_w^{\kappa^0\dag} \simeq R<<X_0>>$. The same hold in family,

\begin{lemmen}
\label{lemdec1}
For all $f \in \mathfrak w_w^{\kappa^{0un}\dag}(R \hat\otimes \mathcal O_K<<S_1,S_2,S_3>>)$, there exists a unique $g \in R<<S_1,S_2,S_3,X_0>>$ such that,
\[f(X_0,X_1,X_2,X_3) = g(X_0)(\kappa^{0un})'(1+p^wX_1,1+p^wX_2,1+p^wX_3).\]
This decomposition induces a bijection
\[  \mathfrak w_w^{\kappa^{0un}\dag}(R \hat\otimes \mathcal O_K<<S_1,S_2,S_3>>) \simeq R<<S_1,S_2,S_3,X_0>>.\]
\end{lemmen}

\begin{lemmen}
Let $\pi$ be a uniformiser of $\mathcal O_K$. Then,
\[ \kappa^{0un}((1+p^wX_i)) \in 1 + \pi \mathcal O_K<< S_1,S_2,S_3,X_1,X_2,X_3>>.\]
\end{lemmen}

\begin{proof}[Proof]
The calculation is made is \cite{AIP} Lemma 8.1.5.3.
\end{proof}

\begin{coren}
Denote  $\mathfrak w_{w,1}^{\kappa^{0un}\dag}$ the reduction modulo $\pi$ of $\mathfrak w_w^{\kappa^{0un}\dag}$.
Then the sheaf  $\mathfrak w_{w,1}^{\kappa^{0un}\dag}$ is constant on $(\mathfrak X_{1}(p^{2n})\times  \mathfrak W(w)^0)\times \Spec(\mathcal O_K/\pi)$ : it is the inverse image of a sheaf on $\mathfrak X_{1}(p^{2n}) \times \Spf(\mathcal O_K/\pi)$.
\end{coren}

Let $f_1^{\sigma\tau'},f_2^{\sigma\tau'},f^{\tau'}$ be an other lift of the basis image of $\HT_{\star,w}$. Let 
\[ P = 
\left(
\begin{array}{ccc}
1+p^wa_1  & p^wa_2  &   \\
 p^wa_3 & 1+p^wa_4  &   \\
  &   &   1+p^w a_5
\end{array}
\right)
\]
be the base change matrix from $\underline f$ to $\underline f'$ and $\underline X'$ the coordinates on $\mathfrak{IW}_{w|\Spf(R)}^+$ relatively to $\underline f'$.

\begin{lemmen}
\label{lem94}
We have the following congruences,
\[ X_0 \equiv X_0' + a_3 \pmod {p^w},\]
\[ X_1 \equiv X_1'  + a_1 \pmod {p^w};\]
\[ X_2 \equiv X_2' + a_4 \pmod {p^w},\]
\[ X_3 \equiv X_3' + a_5 \pmod {p^w}.\]
\end{lemmen}

\begin{proof}[Proof]
Indeed, as seen inside $\mathcal T^\times_an/U_an$, we have that the two systems of coordinates verifies,
\[ P(I_3 + p^w\underline X)U = I_3 + p^w \underline X',\]
where $U \in \GL_2 \times \GL_1$ is a unipotent matrix of the form $I_3 + p^wN$, $N$ upper triangular nilpotent and,
\[ \underline X = 
\left(
\begin{array}{ccc}
X_1  &   &   \\
 X_0 &  X_2 &   \\
  &   &   X_3
\end{array}
\right).
\]
Thus, write $P = I_3 + p^wP_0$, then $I_3 + p^w(P_0 + \underline X + N) \equiv I_3 + p^w\underline X'\pmod{p^{2w}}$.
\end{proof}
 We can thus deduce the following corollary for the family of sheaves,
 
 \begin{coren}
  \label{cor95}
 Let $\kappa^0 \in \mathcal W(w)(K)$. The quasi-coherent sheaf $\mathfrak w_{w}^{\kappa^{0,un}\dag}$ on $\mathfrak X_1(p^{2m})\times \mathfrak W(w)^0$ is a small Banach sheaf.
 
 \end{coren}

\begin{proof}[Proof]
We just have to check that on $\mathfrak X_1(p^{2m}) \times \mathfrak W(w)^0 \times \Spec(O_K/\pi)$ the sheaf $\mathfrak w_{w,1}^{\kappa^{0,un}\dag}$
 is an inductive limit of coherent sheaves which are extensions of the trivial sheaf.
Write $\mathfrak w_{w,1}^{\kappa^{0,un}\dag, \geq r}$ the subsheaf of sections that are locally polynomials in $X_0$ of total degree smaller than $r$. This makes sense globally by Lemma \ref{lem94}, and moreover, $\mathfrak w_{w,1}^{\kappa^{0,un}\dag}$ is the inductive limit over $r$ of these sheaves.
But then, $\mathfrak w_{w,1}^{\kappa^{0,un}\dag, \geq r} \pmod\mathfrak w_{w,1}^{\kappa^{0,un}\dag, \geq r-1}$ is isomorphic to the trivial sheaf.
\end{proof}

\section{Non tempered representations and $(\mathfrak q,K)$-cohomology}
\label{AppB}
We are interested in calculating the $(\mathfrak q,K)$-cohomology of the representation $\pi^n(\chi)$ defined in proposition \ref{transfert} to show it appears in the global sections of a 
coherent automorphic sheaf on the Picard modular surface. 

We have the following theorem of Harris (\cite{Har} Lemma 5.2.3 and proposition 5.4.2, \cite{Gold} Theorem 2.6.1)

\begin{theoren}
Let $\pi = \pi_\infty \otimes \pi_f$ be an automorphic representation of $U(2,1)$ of Harrish-Chandra parameter $\lambda$, and such that $H^0(\mathfrak q,K,\pi_\infty \otimes V_\sigma^\vee) \neq 0$, then there is a $U(2,1)(\mathbb A_f)$ equivariant embedding,
\[ \pi_f \hookrightarrow H^0(X,\mathcal V_{\sigma}^\vee),\]
where $\mathcal V_{\sigma}^\vee$ is the automorphic vector bundle associated to the representation $V_\sigma^\vee$ of $K = K_\infty$.
\end{theoren}

Thus we only need to calculate the $(\mathfrak q,K)$-cohomology of $\pi^n(\chi)_\infty$, and even the one of the restriction of the representation to $SU(2,1)$.
Fortunately we can explicitly do so, rewriting the induction $\ind-n_{B}^{SU(2,1)}(\RR)(\chi_\infty)$, as a space of function, and determining the quotient corresponding to $\pi^n(\chi)$.
In \cite{Wall}, Wallach calculated all the representations of $SU(2,1)(\RR)$ using this description of the induction. As explained in \cite{Wall} p181, the induction space $\ind-n_B^{SU(2,1)(\RR)}(\chi)$ corresponds to $X^{\Lambda}$ with $\Lambda = (a-1)\Lambda_1 + (-a)\Lambda_2$ (which is thus reducible). The shift by $-\Lambda_1 - \Lambda_2$ is due to the normalisation by the modulus character in the induction. Its discrete series subobject corresponds to one of the discrete series $D_{\widetilde{\Lambda}}^-$ described p183, and its quotient corresponds to the non-tempered representation $(T_{a-2}^-,Z_{a-2}^-)$ (defined p184, and the fact that it appears in the said induction is Lemma 7.12). As the name doesn't suggest, $T_{a-2}^-$ – which coincides with the restriction of $\pi^n(\chi)_\infty$ to $SU(2,1)(\RR)$, will be holomorphic (but we can exchange holomorphic and anti-holomorphic by changing the complex structure of the Picard surface).

\begin{propen}
\label{propH0}
Let $(\sigma,V_\sigma) = \Sym^{a-1} \otimes\det^{-a} : M \mapsto \Sym^{a-1}(\overline{M})\otimes \det(\overline{M})^{-a}$ the representation of $U(2) = SK_\infty \subset SU(2,1)(\RR)$. Then
\[ H^0(\mathfrak q,K_\infty,T_{a-2}^-\otimes V_\sigma^\vee) \neq 0.\]
\end{propen}

To show the previous proposition, denote \[J_0 = \left(
\begin{array}{ccc}
 1 &   &   \\
  &  1 &   \\
  &   &   -1
\end{array}
\right)
\]
the hermitian form of signature $(2,1)$ used in \cite{Wall}. Denote
\[ P = \left(
\begin{array}{ccc}
 \frac{1}{\sqrt 2} &   & \frac{1}{\sqrt 2}  \\
  &  1 &   \\
 \frac{1}{\sqrt 2} &   &   -\frac{1}{\sqrt 2}
\end{array}
\right)
\]
the base change matrix (so that $PJ_0P = J$, $\overline P = P^{-1} = P$).
In this new presentation, the complex structure is given by $h' = PhP$, i.e.,
\[ h' : z \in \CC \mapsto \left(
\begin{array}{ccc}
  z &   &   \\
  &   z &   \\
  &   &   \overline z
  \end{array}
\right) \in 
U_{J_0}(\RR).\]

In this form, the Lie algebra of $U_{J_0}(\RR)$ is given by,
\[ \mathfrak g = \{
\left(
\begin{array}{ccc}
ia_0  & b  & c  \\
 -\overline b &  ie_0 &  f \\
\overline c & \overline f & i l_0
\end{array}
\right), a_0,e_0,l_0 \in \RR\}.
\]
using the action of $h'(i)$ we can decompose $\mathfrak g = \mathfrak k + \mathfrak p$ with
\[ \mathfrak p = 
 \{
\left(
\begin{array}{ccc}
0  & 0  & c  \\
 &  0 &f\\
 \overline c & \overline f & 0 
\end{array}
\right), c_0,a_0 \in \RR\}.
\]
Extending scalars to $\CC$, we can further decompose,
$\mathfrak p_\CC = \mathfrak p^+ \oplus \mathfrak p^-$, where conjugacy by $h'(z)$ on $\mathfrak p^+$ and $\mathfrak p^-$ is given by $z/\overline z$ and $\overline z/z$ respectively.
Explicitely, $\mathfrak p^-$ is generated by 
\[ X^- = N^+\otimes i - N^- \otimes 1 \quad \text{and} \quad Y^- = M^+\otimes i - M^- \otimes 1,\]
\[
N^- = \left(
\begin{array}{ccc}
0  &   &  i \\
  &  0 &   \\
-i  &   & 0 
\end{array}
\right) \quad N^+ = \left(
\begin{array}{ccc}
0  &   &  1 \\
  &  0 &   \\
1 &   & 0 
\end{array}
\right)\]
\[
M^- = \left(
\begin{array}{ccc}
0  &   & 0 \\
  &  0 & i  \\
0 & -i  & 0 
\end{array}
\right) \quad M^+ = \left(
\begin{array}{ccc}
0  &   &  0 \\
  &  0 & 1  \\
0&  1 & 0 
\end{array}
\right)\]
and $\mathfrak p^+$ is generated
\[ X^+ = N^+\otimes i + N^- \otimes 1 \quad \text{and} \quad Y^+ = M^+\otimes i + M^- \otimes 1.\]

To calculate the action of $\mathfrak p^-$ on our representation, we use the following formula for a matrice $X$ and $f \in \mathfrak g$ :
\[ X\cdot f = (\frac{d}{dt} \exp(tX)\bullet f)_{t=0},\]

As $Z_{a-2}^-$ is a space of holomorphic functions, we get the following exponentials for the matrices $M^\pm,N^\pm$ :
\[ \exp(tM^-) =
\left(
\begin{array}{ccc}
 1&   & 0  \\
  & \ch t  & i\sh t  \\
&     -i\sh t & \ch t   
\end{array}
\right) \quad \exp(tM^+) =
\left(
\begin{array}{ccc}
 1&   & 0  \\
  & \ch t  & \sh t  \\
&     \sh t & \ch t   
\end{array}
\right)
\]
\[ \exp(tN^-) =
\left(
\begin{array}{ccc}
\ch t&   & i \sh t  \\
  & 1  & 0 \\
-i \sh t &    & \ch t   
\end{array}
\right) \quad \exp(tN^+) =
\left(
\begin{array}{ccc}
\ch t&   & \sh t  \\
  & 1  &0 \\
 \sh t &    & \ch t   
\end{array}
\right)
\]
and the actions of the matrices $M^\pm,N^\pm$ is given by,
\[ N^+f = -(a-2)\overline{z_1} f + (\overline{z_1}^2 - 1) \frac{df}{d\overline{z_1}} + \overline{z_1z_2}\frac{df}{d\overline{z_2}},\]
\[ N^- f = -i(a-2)\overline{z_1} f + i(\overline{z_1}^2 + 1)\frac{df}{d\overline{z_1}} + i\overline{z_1z_2}\frac{df}{d\overline{z_2}},\]
\[ M^+ f = -(a-2)\overline{z_2} f + \overline{z_1z_2}\frac{df}{d\overline{z_1}} +(\overline{z_1}^2 - 1) \frac{df}{d\overline{z_2}},\]
\[ M^- f = -i(a-2)\overline{z_2} f + i\overline{z_1z_2}\frac{df}{d\overline{z_1}} + i(\overline{z_1}^2 + 1)\frac{df}{d\overline{z_2}}.\]

We deduce that the action of $\mathfrak p^-$ is given by,
\[ Y^- f\left(
\begin{array}{c}
 z_1 \\
 z_2  
\end{array}\right) = 
-2i\frac{df}{d\overline{z_2}},\]
and
\[ X^-f\left(
\begin{array}{c}
 z_1 \\
 z_2  
\end{array}\right) = 
%-2ikz_1f + 2i(z_1^2-1)\frac{df}{dz_1} + 2iz_1z_2\frac{df}{dz_2}
-2i\frac{df}{d\overline{z_1}}
%-2kiz_1f -i\frac{df}{dz_1}
 \]
and the action of $\mathfrak p^+$ by
\[ Y^+ f\left(
\begin{array}{c}
 z_1 \\
 z_2  
\end{array}\right) = 
-2i(a-1)\overline{z_2}f + 2i\overline{z_1z_2}\frac{df}{d\overline{z_1}}+2i\overline{z_2}^2\frac{df}{d\overline{z_2}}
%z_1z_2\frac{df}{dz_1} + z_2^2 \frac{df}{dz_2} 
,\]
and
\[ X^+f\left(
\begin{array}{c}
 z_1 \\
 z_2  
\end{array}\right) = 
-2i(a-1)\overline{z_1}f + 2i\overline{z_1z_2}\frac{df}{d\overline{z_2}} + 2i\overline{z_1}^2\frac{df}{d\overline{z_1}}
%z_1^2\frac{df}{dz_1} + z_1z_2\frac{df}{dz_2}
 \]
As $Z_{a-2}^-$ is defined as a completion of the quotient of holomorphic polynomials in variables $\overline{z_1},\overline{z_2}$ by the subspace of polynomials of degrees less of equal than $(a-2)$, 
$H^0(\mathfrak p^-,Z_{a-2}^-) = (Z_{a-2}^+)^{\mathfrak p^- = 0}$ is identified with homogeneous polynomials in $\overline{z_1},\overline{z_2}$ of degree $a-1$.

As for a representation $\tau$ of $K_\infty$ we have,
\[H^q(\mathfrak q,K,V\otimes V_\tau) = (H^q(\mathfrak p^-,V)\otimes V_\tau)^K\]
(cf. \cite{Har} 4.14), we have that $H^0(\mathfrak q,K,Z_{a-1}^+\otimes V_\sigma^\vee) \neq 0$.

\begin{remaen}
Using a slightly more precise calculation for $U(2,1)$ instead of $SU(2,1)$, we could show that for $U(2,1)$,
\[ H^0(\mathfrak q,K_\infty,\pi^n(\chi)\otimes V_{(0,1-a,a)}) \neq 0,\]
in particular, the Hecke eigenvalues of $\pi^n(\chi)$ appears in the global sections over $X$, the Picard modular variety, of the automorphic sheaf $\omega^{(0,1-a,a)}$.
\end{remaen}

\bibliographystyle{alpha-fr} % D'autres styles sont disponibles. Notez que les distributions LaTeX n'incluent g√©n√©ralement pas de styles de bibliographies francis√©s ; vous aurez donc des bibliographies en anglais.
\bibliography{biblio} % Remplacer "biblio" par le nom de votre fichier de r√©f√©rences bibliographiques.

\newcommand{\etalchar}[1]{$^{#1}$}
\begin{thebibliography}{{Har}90b}
\expandafter\ifx\csname fonteauteurs\endcsname\relax
\def\fonteauteurs{\scshape}\fi

\bibitem[ABI{\etalchar{+}}16]{AIP2}
Fabrizio \bgroup\fonteauteurs\bgroup {Andreatta}\egroup\egroup{}, St\'ephane
  \bgroup\fonteauteurs\bgroup {Bijakowski}\egroup\egroup{}, Adrian
  \bgroup\fonteauteurs\bgroup {Iovita}\egroup\egroup{}, Payman~L.
  \bgroup\fonteauteurs\bgroup {Kassaei}\egroup\egroup{}, Vincent
  \bgroup\fonteauteurs\bgroup {Pilloni}\egroup\egroup{}, Beno\^{\i}t
  \bgroup\fonteauteurs\bgroup {Stroh}\egroup\egroup{}, Yichao
  \bgroup\fonteauteurs\bgroup {Tian}\egroup\egroup{} et Liang
  \bgroup\fonteauteurs\bgroup {Xiao}\egroup\egroup{} :
\newblock {\em {Arithm\'etique $p$-adique des formes de Hilbert.}}
\newblock Paris: Soci\'et\'e Math\'ematique de France (SMF), 2016.

\bibitem[AIP15]{AIP}
Fabrizio \bgroup\fonteauteurs\bgroup {Andreatta}\egroup\egroup{}, Adrian
  \bgroup\fonteauteurs\bgroup {Iovita}\egroup\egroup{} et Vincent
  \bgroup\fonteauteurs\bgroup {Pilloni}\egroup\egroup{} :
\newblock {$p$-adic families of Siegel modular cuspforms.}
\newblock {\em {Ann. Math. (2)}}, 181(2)\string:\penalty500\relax 623--697,
  2015.

\bibitem[BC04]{BC1}
Jo\"el \bgroup\fonteauteurs\bgroup {Bella\"{\i}che}\egroup\egroup{} et Ga\"etan
  \bgroup\fonteauteurs\bgroup {Chenevier}\egroup\egroup{} :
\newblock {Formes non temp\'er\'ees pour $\mathrm U(3)$ et conjectures de
  Bloch-Kato.}
\newblock {\em {Ann. Sci. \'Ec. Norm. Sup\'er. (4)}},
  37(4)\string:\penalty500\relax 611--662, 2004.

\bibitem[BC09]{BC2}
Jo\"el \bgroup\fonteauteurs\bgroup {Bella\"{\i}che}\egroup\egroup{} et Ga\"etan
  \bgroup\fonteauteurs\bgroup {Chenevier}\egroup\egroup{} :
\newblock {\em {Families of Galois representations and Selmer groups.}}
\newblock Paris: Soci\'et\'e Math\'ematique de France, 2009.

\bibitem[Bel02]{Bellthesefull}
Jo{\"e}l \bgroup\fonteauteurs\bgroup Bella{\"\i}che\egroup\egroup{} :
\newblock {\em Congruences endoscopiques et repr{\'e}sentations galoisiennes}.
\newblock Th\`ese de doctorat, Universit{\'e} de Paris-Sud. Facult{\'e} des
  Sciences d'Orsay (Essonne), 2002.

\bibitem[{Bel}03]{BellRibet}
Jo\"el \bgroup\fonteauteurs\bgroup {Bella\"{\i}che}\egroup\egroup{} :
\newblock {A propos d'un lemme de Ribet.}
\newblock {\em {Rend. Semin. Mat. Univ. Padova}}, 109\string:\penalty500\relax
  45--62, 2003.

\bibitem[{Bel}06a]{BelComp}
Jo\"el \bgroup\fonteauteurs\bgroup {Bella\"\i che}\egroup\egroup{} :
\newblock {Sur la compatibilit\'e entre les correspondances de Langlands locale
  et globale pour $\text{U}(3)$.}
\newblock {\em {Comment. Math. Helv.}}, 81(2)\string:\penalty500\relax
  449--470, 2006.

\bibitem[{Bel}06b]{BellThese}
Jo\"el \bgroup\fonteauteurs\bgroup {Bella\"{\i}che}\egroup\egroup{} :
\newblock {Rel\`evement des formes modulaires de Picard.}
\newblock {\em {J. Lond. Math. Soc., II. Ser.}}, 74(1)\string:\penalty500\relax
  13--25, 2006.

\bibitem[Bel12]{BellSelmer}
Jo{\"e}l \bgroup\fonteauteurs\bgroup Bella{\"\i}che\egroup\egroup{} :
\newblock Ranks of selmer groups in an analytic family.
\newblock {\em Transactions of the American Mathematical Society},
  364(9)\string:\penalty500\relax 4735--4761, 2012.

\bibitem[Ber96]{Ber}
Pierre \bgroup\fonteauteurs\bgroup Berthelot\egroup\egroup{} :
\newblock {\em Cohomologie rigide et cohomologie rigide {\`a} supports
  propres}.
\newblock Universit{\'e} de Rennes 1. Institut de Recherche Math{\'e}matique de
  Rennes [IRMAR], 1996.

\bibitem[BGHT11]{BLGHT}
Tom \bgroup\fonteauteurs\bgroup {Barnet-Lamb}\egroup\egroup{}, David
  \bgroup\fonteauteurs\bgroup {Geraghty}\egroup\egroup{}, Michael
  \bgroup\fonteauteurs\bgroup {Harris}\egroup\egroup{} et Richard
  \bgroup\fonteauteurs\bgroup {Taylor}\egroup\egroup{} :
\newblock {A family of Calabi-Yau varieties and potential automorphy. II.}
\newblock {\em {Publ. Res. Inst. Math. Sci.}}, 47(1)\string:\penalty500\relax
  29--98, 2011.

\bibitem[{Bij}]{Bij}
Stephane \bgroup\fonteauteurs\bgroup {Bijakowski}\egroup\egroup{} :
\newblock {Formes modulaires surconvergentes, ramification et classicite.}

\bibitem[{Bij}16]{Bijmu}
St\'ephane \bgroup\fonteauteurs\bgroup {Bijakowski}\egroup\egroup{} :
\newblock {Analytic continuation on Shimura varieties with $\mu$-ordinary
  locus.}
\newblock {\em {Algebra Number Theory}}, 10(4)\string:\penalty500\relax
  843--885, 2016.

\bibitem[BK90]{BK}
Spencer \bgroup\fonteauteurs\bgroup {Bloch}\egroup\egroup{} et Kazuya
  \bgroup\fonteauteurs\bgroup {Kato}\egroup\egroup{} :
\newblock {$L$-functions and Tamagawa numbers of motives.}
\newblock \emph{In} {\em {The presentation functor and the compactified
  Jacobian}}, pages 333--400. 1990.

\bibitem[{Bla}02]{Bla}
Laure \bgroup\fonteauteurs\bgroup {Blasco}\egroup\egroup{} :
\newblock {Description du dual admissible de $U(2,1)(F)$ par la th\'eorie des
  types de C. Bushnell et P. Kutzko.}
\newblock {\em {Manuscr. Math.}}, 107(2)\string:\penalty500\relax 151--186,
  2002.

\bibitem[BPS16]{BPS}
St\'ephane \bgroup\fonteauteurs\bgroup {Bijakowski}\egroup\egroup{}, Vincent
  \bgroup\fonteauteurs\bgroup {Pilloni}\egroup\egroup{} et Beno{\^\i}t
  \bgroup\fonteauteurs\bgroup {Stroh}\egroup\egroup{} :
\newblock {Classicit\'e de formes modulaires surconvergentes.}
\newblock {\em {Ann. Math. (2)}}, 183(3)\string:\penalty500\relax 975--1014,
  2016.

\bibitem[BR92]{BR}
Don \bgroup\fonteauteurs\bgroup {Blasius}\egroup\egroup{} et Jonathan~D.
  \bgroup\fonteauteurs\bgroup {Rogawski}\egroup\egroup{} :
\newblock {Tate classes and arithmetic quotients of the two-ball.}
\newblock \emph{In} {\em {The zeta functions of Picard modular surfaces. Based
  on lectures, delivered at a CRM workshop in the spring of 1988, Montr\'eal,
  Canada}}, pages 421--444. Montr\'eal: Centre de Recherches Math\'ematiques,
  Universit\'e de Montr\'eal, 1992.

\bibitem[{Bra}16]{Bra}
Riccardo \bgroup\fonteauteurs\bgroup {Brasca}\egroup\egroup{} :
\newblock {Eigenvarieties for cuspforms over PEL type Shimura varieties with
  dense ordinary locus.}
\newblock {\em {Canadian Journal of Mathematics}}, 2016.

\bibitem[{Buz}07]{Buz}
Kevin \bgroup\fonteauteurs\bgroup {Buzzard}\egroup\egroup{} :
\newblock {Eigenvarieties.}
\newblock \emph{In} {\em {$L$-functions and Galois representations. Based on
  the symposium, Durham, UK, July 19--30, 2004}}, pages 59--120. Cambridge:
  Cambridge University Press, 2007.

\bibitem[Cas95]{Casselman}
Bill \bgroup\fonteauteurs\bgroup Casselman\egroup\egroup{} :
\newblock Introduction to admissible representations of p-adic groups.
\newblock {\em unpublished notes}, 1995.

\bibitem[CH13]{CH}
Ga\"etan \bgroup\fonteauteurs\bgroup {Chenevier}\egroup\egroup{} et Michael
  \bgroup\fonteauteurs\bgroup {Harris}\egroup\egroup{} :
\newblock {Construction of automorphic Galois representations. II.}
\newblock {\em {Camb. J. Math.}}, 1(1)\string:\penalty500\relax 53--73, 2013.

\bibitem[Che03]{CheApp}
Ga{\"e}tan \bgroup\fonteauteurs\bgroup Chenevier\egroup\egroup{} :
\newblock {\em Familles p-adiques de formes automorphes et applications aux
  conjectures de Bloch-Kato}.
\newblock Th\`ese de doctorat, Paris 7, 2003.

\bibitem[{Che}04]{Che1}
Ga\"etan \bgroup\fonteauteurs\bgroup {Chenevier}\egroup\egroup{} :
\newblock {Familles $p$-adiques de formes automorphes pour $\text{GL}_n$.}
\newblock {\em {J. Reine Angew. Math.}}, 570\string:\penalty500\relax 143--217,
  2004.

\bibitem[{Che}05]{CheJL}
Ga\"etan \bgroup\fonteauteurs\bgroup {Chenevier}\egroup\egroup{} :
\newblock {Une correspondance de Jacquet-Langlands $p$-adique.}
\newblock {\em {Duke Math. J.}}, 126(1)\string:\penalty500\relax 161--194,
  2005.

\bibitem[CHT08]{CHT}
Laurent \bgroup\fonteauteurs\bgroup {Clozel}\egroup\egroup{}, Michael
  \bgroup\fonteauteurs\bgroup {Harris}\egroup\egroup{} et Richard
  \bgroup\fonteauteurs\bgroup {Taylor}\egroup\egroup{} :
\newblock {Automorphy for some $l$-adic lifts of automorphic mod $l$ Galois
  representations. With Appendix A, summarizing unpublished work of Russ Mann,
  and Appendix B by Marie-France Vign\'eras.}
\newblock {\em {Publ. Math., Inst. Hautes \'Etud. Sci.}},
  108\string:\penalty500\relax 1--181, 2008.

\bibitem[dG16]{dsg}
Ehud \bgroup\fonteauteurs\bgroup {de Shalit}\egroup\egroup{} et Eyal~Z.
  \bgroup\fonteauteurs\bgroup {Goren}\egroup\egroup{} :
\newblock {A theta operator on Picard modular forms modulo an inert prime.}
\newblock {\em {Res. Math. Sci.}}, 3\string:\penalty500\relax 65, 2016.

\bibitem[{Fal}83]{Fal83}
Gerd \bgroup\fonteauteurs\bgroup {Faltings}\egroup\egroup{} :
\newblock {On the cohomology of locally symmetric Hermitian spaces.}
\newblock {S\'emin. d'Alg\`ebre P. Dubreil et M.-P. Malliavin, 35\`eme Ann\'ee,
  Proc., Paris 1982, Lect. Notes Math. 1029, 55-98 (1983).}, 1983.

\bibitem[{Far}10]{FarHN}
Laurent \bgroup\fonteauteurs\bgroup {Fargues}\egroup\egroup{} :
\newblock {La filtration de Harder-Narasimhan des sch\'emas en groupes finis et
  plats.}
\newblock {\em {J. Reine Angew. Math.}}, 645\string:\penalty500\relax 1--39,
  2010.

\bibitem[Far11]{Far2}
Laurent \bgroup\fonteauteurs\bgroup Fargues\egroup\egroup{} :
\newblock La filtration canonique des points de torsion des groupes
  {$p$}-divisibles.
\newblock {\em Ann. Sci. \'Ec. Norm. Sup\'er. (4)},
  44(6)\string:\penalty500\relax 905--961, 2011.
\newblock With collaboration of Yichao Tian.

\bibitem[FP94]{FPR}
Jean-Marc \bgroup\fonteauteurs\bgroup {Fontaine}\egroup\egroup{} et Bernadette
  \bgroup\fonteauteurs\bgroup {Perrin-Riou}\egroup\egroup{} :
\newblock {Autour des conjectures de Bloch et Kato: Cohomologie galoisienne et
  valeurs de fonctions $L$.}
\newblock \emph{In} {\em {Motives. Proceedings of the summer research
  conference on motives, held at the University of Washington, Seattle, WA,
  USA, July 20-August 2, 1991}}, pages 599--706. Providence, RI: American
  Mathematical Society, 1994.

\bibitem[GN17]{GN}
Wushi \bgroup\fonteauteurs\bgroup {Goldring}\egroup\egroup{} et Marc-Hubert
  \bgroup\fonteauteurs\bgroup {Nicole}\egroup\egroup{} :
\newblock {The $\mu$-ordinary Hasse invariant of unitary Shimura varieties.}
\newblock {\em {J. Reine Angew. Math.}}, 728\string:\penalty500\relax 137--151,
  2017.

\bibitem[{Gol}14]{Gold}
Wushi \bgroup\fonteauteurs\bgroup {Goldring}\egroup\egroup{} :
\newblock {Galois representations associated to holomorphic limits of discrete
  series.}
\newblock {\em {Compos. Math.}}, 150(2)\string:\penalty500\relax 191--228,
  2014.

\bibitem[{Gor}92]{Gordon}
Brent~B. \bgroup\fonteauteurs\bgroup {Gordon}\egroup\egroup{} :
\newblock {Canonical models of Picard modular surfaces.}
\newblock \emph{In} {\em {The zeta functions of Picard modular surfaces. Based
  on lectures, delivered at a CRM workshop in the spring of 1988, Montr\'eal,
  Canada}}, pages 1--29. Montr\'eal: Centre de Recherches Math\'ematiques,
  Universit\'e de Montr\'eal, 1992.

\bibitem[{Har}84]{HarVB}
Michael \bgroup\fonteauteurs\bgroup {Harris}\egroup\egroup{} :
\newblock {Arithmetic vector bundles of Shimura varieties.}
\newblock {Automorphic forms of several variables, Taniguchi Symp., Katata/Jap.
  1983, Prog. Math. 46, 138-159 (1984).}, 1984.

\bibitem[{Har}90a]{HarCor}
Michael \bgroup\fonteauteurs\bgroup {Harris}\egroup\egroup{} :
\newblock {Automorphic forms and the cohomology of vector bundles on Shimura
  varieties.}
\newblock {Automorphic forms, Shimura varieties, and L-functions. Vol. II,
  Proc. Conf., Ann Arbor/MI (USA) 1988, Perspect. Math. 11, 41-91 (1990).},
  1990.

\bibitem[{Har}90b]{Har}
Michael \bgroup\fonteauteurs\bgroup {Harris}\egroup\egroup{} :
\newblock {Automorphic forms of ${\bar \partial}$-cohomology type as coherent
  cohomology classes.}
\newblock {\em {J. Differ. Geom.}}, 32(1)\string:\penalty500\relax 1--63, 1990.

\bibitem[Her15]{Her1}
Valentin \bgroup\fonteauteurs\bgroup Hernandez\egroup\egroup{} :
\newblock Invariants de hasse $\mu$-ordinaires.
\newblock {\em arXiv preprint arXiv:1608.06176}, 2015.

\bibitem[Her16]{Her2}
Valentin \bgroup\fonteauteurs\bgroup Hernandez\egroup\egroup{} :
\newblock {La filtration canonique des $\mathcal O$-modules p-divisibles.}
\newblock {\em arXiv preprint arXiv:1611.07396}, 2016.

\bibitem[Hsi14]{Hsieh}
Ming-Lun \bgroup\fonteauteurs\bgroup Hsieh\egroup\egroup{} :
\newblock Eisenstein congruence on unitary groups and iwasawa main conjectures
  for cm fields.
\newblock {\em Journal of the American Mathematical Society},
  27(3)\string:\penalty500\relax 753--862, 2014.

\bibitem[{Jon}11]{Jones}
Owen~T.R. \bgroup\fonteauteurs\bgroup {Jones}\egroup\egroup{} :
\newblock {An analogue of the BGG resolution for locally analytic principal
  series.}
\newblock {\em {J. Number Theory}}, 131(9)\string:\penalty500\relax 1616--1640,
  2011.

\bibitem[{Key}84]{Keys}
David \bgroup\fonteauteurs\bgroup {Keys}\egroup\egroup{} :
\newblock {Principal series representations of special unitary groups over
  local fields.}
\newblock {\em {Compos. Math.}}, 51\string:\penalty500\relax 115--130, 1984.

\bibitem[Kna16]{Knapp}
Anthony~W \bgroup\fonteauteurs\bgroup Knapp\egroup\egroup{} :
\newblock {\em Representation Theory of Semisimple Groups: An Overview Based on
  Examples (PMS-36)}.
\newblock Princeton university press, 2016.

\bibitem[{Kot}92]{Kotjams}
Robert~E. \bgroup\fonteauteurs\bgroup {Kottwitz}\egroup\egroup{} :
\newblock {Points on some Shimura varieties over finite fields.}
\newblock {\em {J. Am. Math. Soc.}}, 5(2)\string:\penalty500\relax 373--444,
  1992.

\bibitem[{Lan}12]{LanICCM}
Kai-Wen \bgroup\fonteauteurs\bgroup {Lan}\egroup\egroup{} :
\newblock {Geometric modular forms and the cohomology of torsion automorphic
  sheaves.}
\newblock \emph{In} {\em {Fifth international congress of Chinese
  mathematicians. Proceedings of the ICCM '10, Beijing, China, December 17--22,
  2010. Part 1}}, pages 183--208. Providence, RI: American Mathematical Society
  (AMS); Somerville, MA: International Press, 2012.

\bibitem[Lan13]{Lan}
Kai-Wen \bgroup\fonteauteurs\bgroup Lan\egroup\egroup{} :
\newblock {\em Arithmetic compactifications of {PEL}-type {S}himura varieties},
  volume~36 de {\em London Mathematical Society Monographs Series}.
\newblock Princeton University Press, Princeton, NJ, 2013.

\bibitem[Lan17]{LanIMRN}
Kai-Wen \bgroup\fonteauteurs\bgroup Lan\egroup\egroup{} :
\newblock Integral models of toroidal compactifications with projective cone
  decompositions.
\newblock {\em Int. Math. Res. Not. IMRN}, (11)\string:\penalty500\relax
  3237--3280, 2017.

\bibitem[{Lar}92]{Lars}
Michael~J. \bgroup\fonteauteurs\bgroup {Larsen}\egroup\egroup{} :
\newblock {Arithmetic compactification of some Shimura surfaces.}
\newblock \emph{In} {\em {The zeta functions of Picard modular surfaces. Based
  on lectures, delivered at a CRM workshop in the spring of 1988, Montr\'eal,
  Canada}}, pages 31--45. Montr\'eal: Centre de Recherches Math\'ematiques,
  Universit\'e de Montr\'eal, 1992.

\bibitem[Liu12]{Liu}
Ruochuan \bgroup\fonteauteurs\bgroup Liu\egroup\egroup{} :
\newblock Triangulation of refined families.
\newblock {\em arXiv preprint arXiv:1202.2188}, 2012.

\bibitem[LRZ92]{LRZ}
{The zeta functions of Picard modular surfaces. Based on lectures, delivered at
  a CRM workshop in the spring of 1988, Montr\'eal, Canada.}
\newblock {Montr\'eal: Centre de Recherches Math\'ematiques, Universit\'e de
  Montr\'eal. xiv, 492 p. (1992).}, 1992.

\bibitem[{Mil}88]{Mil}
J.S. \bgroup\fonteauteurs\bgroup {Milne}\egroup\egroup{} :
\newblock {Automorphic vector bundles on connected Shimura varieties.}
\newblock {\em {Invent. Math.}}, 92(1)\string:\penalty500\relax 91--128, 1988.

\bibitem[MM63]{MM}
Yoz\^o \bgroup\fonteauteurs\bgroup {Matsushima}\egroup\egroup{} et
  S.~\bgroup\fonteauteurs\bgroup {Murakami}\egroup\egroup{} :
\newblock {On vector bundle valued harmonic forms and automorphic forms on
  symmetric Riemannian manifolds.}
\newblock {\em {Ann. Math. (2)}}, 78\string:\penalty500\relax 365--416, 1963.

\bibitem[PS12]{PSdep}
Vincent \bgroup\fonteauteurs\bgroup Pilloni\egroup\egroup{} et Beno{\^\i}t
  \bgroup\fonteauteurs\bgroup Stroh\egroup\egroup{} :
\newblock Surconvergence et classicit{\'e}: le cas d{\'e}ploy{\'e}.
\newblock {\em preprint}, 86, 2012.

\bibitem[Rib]{Rib}
K~\bgroup\fonteauteurs\bgroup Ribet\egroup\egroup{} :
\newblock A modular construction of unramified p-extensions of q (up).
\newblock {\em Invent. Math}, 34.

\bibitem[{Rog}90]{Rogbook}
Jonathan~D. \bgroup\fonteauteurs\bgroup {Rogawski}\egroup\egroup{} :
\newblock {\em {Automorphic representations of unitary groups in three
  variables.}}
\newblock Princeton, NJ: Princeton University Press, 1990.

\bibitem[{Rog}92]{Rog}
Jonathan~D. \bgroup\fonteauteurs\bgroup {Rogawski}\egroup\egroup{} :
\newblock {The multiplicity formula for $A$-packets.}
\newblock \emph{In} {\em {The zeta functions of Picard modular surfaces. Based
  on lectures, delivered at a CRM workshop in the spring of 1988, Montr\'eal,
  Canada}}, pages 395--419. Montr\'eal: Centre de Recherches Math\'ematiques,
  Universit\'e de Montr\'eal, 1992.

\bibitem[{Rub}91]{Rub}
Karl \bgroup\fonteauteurs\bgroup {Rubin}\egroup\egroup{} :
\newblock {The ``main conjectures'' of Iwasawa theory for imaginary quadratic
  fields.}
\newblock {\em {Invent. Math.}}, 103(1)\string:\penalty500\relax 25--68, 1991.

\bibitem[{Shi}78]{Shi}
Goro \bgroup\fonteauteurs\bgroup {Shimura}\egroup\egroup{} :
\newblock {The arithmetic of automorphic forms with respect to a unitary
  group.}
\newblock {\em {Ann. Math. (2)}}, 107\string:\penalty500\relax 569--605, 1978.

\bibitem[{Ski}12]{Ski}
Christopher \bgroup\fonteauteurs\bgroup {Skinner}\egroup\egroup{} :
\newblock {Galois representations associated with unitary groups over $\mathbb
  Q$.}
\newblock {\em {Algebra Number Theory}}, 6(8)\string:\penalty500\relax
  1697--1717, 2012.

\bibitem[{Str}10]{Stroh}
Beno{\^\i}t \bgroup\fonteauteurs\bgroup {Stroh}\egroup\egroup{} :
\newblock {Compactification de vari\'et\'es de Siegel aux places de mauvaise
  r\'eduction.}
\newblock {\em {Bull. Soc. Math. Fr.}}, 138(2)\string:\penalty500\relax
  259--315, 2010.

\bibitem[SU02]{SU}
Christopher \bgroup\fonteauteurs\bgroup {Skinner}\egroup\egroup{} et Eric
  \bgroup\fonteauteurs\bgroup {Urban}\egroup\egroup{} :
\newblock {Sur les d\'eformations $p$-adiques des formes de Saito-Kurokawa.}
\newblock {\em {C. R., Math., Acad. Sci. Paris}},
  335(7)\string:\penalty500\relax 581--586, 2002.

\bibitem[{Urb}11]{Urb}
Eric \bgroup\fonteauteurs\bgroup {Urban}\egroup\egroup{} :
\newblock {Eigenvarieties for reductive groups.}
\newblock {\em {Ann. Math. (2)}}, 174(3)\string:\penalty500\relax 1685--1784,
  2011.

\bibitem[{Wal}76]{Wall}
Nolan~R. \bgroup\fonteauteurs\bgroup {Wallach}\egroup\egroup{} :
\newblock {On the Selberg trace formula in the case of compact quotient.}
\newblock {\em {Bull. Am. Math. Soc.}}, 82\string:\penalty500\relax 171--195,
  1976.

\bibitem[Wed99]{Wed1}
Torsten \bgroup\fonteauteurs\bgroup Wedhorn\egroup\egroup{} :
\newblock Ordinariness in good reductions of {S}himura varieties of {PEL}-type.
\newblock {\em Ann. Sci. \'Ecole Norm. Sup. (4)},
  32(5)\string:\penalty500\relax 575--618, 1999.

\bibitem[Wed01]{Wed2}
Torsten \bgroup\fonteauteurs\bgroup Wedhorn\egroup\egroup{} :
\newblock The dimension of {O}ort strata of {S}himura varieties of {PEL}-type.
\newblock \emph{In} {\em Moduli of abelian varieties ({T}exel {I}sland, 1999)},
  volume 195 de {\em Progr. Math.}, pages 441--471. Birkh\"auser, Basel, 2001.

\bibitem[Yos]{Yosh}
Teruyoshi \bgroup\fonteauteurs\bgroup Yoshida\egroup\egroup{} :
\newblock Week 2: Betti cohomology of shimura varieties—the matsushima
  formula.
\newblock {\em {Notes on his Website}}.

\end{thebibliography}

\end{document}